\documentclass[a4paper,10pt,leqno,fleqn,twoside]{article}

\usepackage[utf8]{inputenc}

\usepackage{amsmath}  
\usepackage{mathrsfs}  
\usepackage{amsthm}
\usepackage{amsfonts}
\usepackage{amssymb, latexsym}
\usepackage{verbatim}

\usepackage{kerkis}

\usepackage{graphicx}
\usepackage{color}
\usepackage{xcolor}
\usepackage{tikz}
\usetikzlibrary{calc,intersections,through,backgrounds}
\usetikzlibrary{decorations.pathreplacing}
\usetikzlibrary{shapes}

\usepackage{float} 
\usepackage{enumerate}
\usepackage[margin=1in, marginparwidth=.8in]{geometry}
\usepackage{setspace}

\usepackage[pdfauthor={}, pdftitle={}, pdftex]{hyperref}
\hypersetup{colorlinks=true}

\usepackage{fancyhdr}

\usepackage[toc,page,titletoc]{appendix}

\noappendicestocpagenum

\usepackage{tabularx}
\usepackage{array}
\newcolumntype{C}[1]{>{\centering\arraybackslash}m{#1}} 

\usepackage{macros}

\definecolor{darkred}{rgb}{0.4,0,0} 
\definecolor{darkblue}{rgb}{0,0,0.4}
\definecolor{darkgreen}{rgb}{0,.4,0}

\newcommand{\BB}{\mathcal{B}} 
\newcommand{\CC}{\mathcal{C}} 

\newcommand{\PP}{\mathbb{P}} 
\newcommand{\EE}{\mathbb{E}} 
\newcommand{\TT}{\mathbb{T}} 
\newcommand{\dd}{\,\mathrm{d}} 
\newcommand{\RR}{\mathbb{R}} 
\newcommand{\ZZ}{\mathbb{Z}} 

\newcommand{\lng}{\langle} 
\newcommand{\rng}{\rangle} 

\newcommand{\ee}{\mathrm{e}} 
\newcommand{\ii}{\mathrm{i}} 

\newcommand{\al}{\alpha}
\newcommand{\eps}{\varepsilon} 
\newcommand{\bt}{\beta}
\newcommand{\gm}{\gamma}

\newcommand{\VVert}[1]{{\left\vert\kern-0.25ex \left\vert\kern-0.25ex \left \vert #1 
    \right\vert\kern-0.25ex \right\vert\kern-0.25ex \right\vert}} 

\theoremstyle{plain}
\newtheorem{theorem}{Theorem}[section]
\newtheorem{proposition}[theorem]{Proposition}
\newtheorem{corollary}[theorem]{Corollary}
\newtheorem{lemma}[theorem]{Lemma}

\theoremstyle{definition}
\newtheorem{definition}[theorem]{Definition}

\theoremstyle{remark}
\newtheorem{remark}[theorem]{Remark}

\numberwithin{equation}{section}

\title{Exponential loss of memory for the 2-dimensional  Allen-Cahn equation with small noise}
\author{Pavlos Tsatsoulis, Hendrik Weber}
\date{}

\fancypagestyle{plain}{

  \fancyhf{}
  \fancyfoot[L]{\footnotesize \today} 
  \fancyfoot[C]{\thepage}
}

\begin{document}

\maketitle

\begin{abstract} We prove an asymptotic coupling theorem for the $2$-dimensional Allen--Cahn equation perturbed by a small space-time white noise. 
We show that with overwhelming  probability two profiles that start close to the minimisers of the potential of the deterministic system contract 
exponentially fast in a suitable topology. In the $1$-dimensional case a similar result was shown in \cite{MS88,MOS89}.

It is well-known that in more than one dimension solutions  of this equation are distribution-valued, and the equation has to be interpreted in 
a renormalised sense. Formally, this renormalisation corresponds to moving the minima of the potential infinitely far apart and making them 
infinitely deep. We show that despite this renormalisation, solutions behave like perturbations of the deterministic system without renormalisation:
they spend large stretches of time close to the minimisers of the (un-renormalised) potential and  the exponential 
contraction rate of different profiles  is given by the second derivative of the potential in these points. 

As an application we prove an Eyring--Kramers law for the transition times between the stable solutions of the deterministic system for fixed 
initial conditions.



\end{abstract}

\begin{spacing}{0.05}
 \tableofcontents
\end{spacing}

\tikzsetexternalprefix{./macros/draft/}

\section{Introduction}

We are interested in the behaviour of solutions to the  Allen--Cahn equation, perturbed by a small noise term.
The deterministic equation is given by
\begin{equation} \label{eq:det_AC}
 (\partial_t-\Delta) X = - X^3 + X,
\end{equation}
and it is well-known that \eqref{eq:det_AC} is a gradient flow with respect to the potential
\begin{equation} \label{eq:V_def}
 V(X) := \int \left(\frac{1}{2}|\nabla X(z)|^2 - \frac{1}{2} |X(z)|^2 + \frac{1}{4} |X(z)|^4 \right) \dd z.
\end{equation}
The fluctuation-dissipation theorem  suggests an additive Gaussian space-time white noise $\xi$ as a natural random perturbation of 
\eqref{eq:det_AC}; so we consider
\begin{equation} \label{eq:AC}
 (\partial_t - \Delta) X = -X^3 + X + \sqrt{2\eps} \xi,
\end{equation}
for a small parameter $\eps>0$.

In the $1$-dimensional case, i.e. the case where the solution $X$ depends on time and a $1$-dimensional spatial argument, the behaviour 
of solutions to \eqref{eq:AC} is well-understood. Solutions exhibit the phenomenon of metastability, i.e.  they typically spend large stretches of time 
close to the minimisers of the potential \eqref{eq:V_def} with rare and relatively quick noise-induced transitions between them.
Early contributions go back to the 80s where Farris and Jona--Lasinio  \cite{FJL82} studied the system on the level of large deviations.

We are particularly interested in the ``exponential loss of memory property'' first observed  by Martinelli, Olivieri and Scoppola in \cite{MS88,MOS89}.
They studied the flow map induced by \eqref{eq:AC}, i.e. the random map $x \mapsto X(t;x)$ which associates to any initial condition the corresponding 
solution at time $t$, and showed that for large $t$ the map becomes essentially constant. 
They also showed that with overwhelming probability, solutions that start within the basin of attraction of the same minimiser of $V$ contract exponentially fast, with exponential rate given 
by the smallest eigenvalue of the linearisation of $V$ in this minimiser. This implies for example that the law of such solutions at large times is essentially insensitive 
to the precise location at which they are started.

It is very natural to consider higher dimensional analogues of \eqref{eq:AC}, but unfortunately for 
space dimension $d \geq 2$, equation \eqref{eq:AC} is ill-posed. In fact, for $d \geq 2$  the space-time
white noise becomes so irregular, that solutions have to be interpreted in the sense of Schwartz distributions, and the
interpretation of the nonlinear term is a priori unclear. These kind of singular stochastic partial differential equations (SPDEs)
have received a lot of attention recently (see e.g. \cite{DPD03,Ha14,GIP15}). The solution 
proposed in these works is to \emph{renormalise} the equation, by removing some infinite terms, formally leading to the equation
\begin{equation}
 (\partial_t - \Delta) X = - X^3 + (1+3\eps \infty) X + \sqrt{2\eps} \xi. \label{eq:Phi42} 
\end{equation} 
Note that formally, this renormalisation corresponds to moving the minima of the double-well potential
out to $\pm \infty$ and  making them infinitely deep at the same time.
So at first glance,  it seems unclear why these 
renormalised distribution-valued solutions should exhibit similar behaviour to the $1$-dimensional
function-valued solutions of \eqref{eq:AC}.

In \cite{HW15} Hairer and the second named author studied the small $\eps$ asymptotics for \eqref{eq:Phi42} for space dimension $d=2$ and $3$
on the level of Freidlin-Wentzell type large deviations. They obtained a large deviation principle with rate function $\mathcal{I}$  given by 
\begin{equation} \label{eq:I_def}
 \mathcal{I}(X) := \frac{1}{4} \int_0^T \int \left( \partial_t X(t,z) - \left( \Delta X(t,z) -\left( X(t,z)^3 - X(t,z) \right) \right)\right)^2 \dd z \dd t.
\end{equation}
 In fact, a result in a similar spirit had already appeared in the 90s \cite{JLM90}. The striking fact is that this rate function is exactly the 2-dimensional version of the  rate function 
obtained in the 1-dimensional case \cite{FJL82}. The infinite renormalisation constant 
does not  affect the rate functional.  This result implies that for small $\eps$ solutions of the renormalised SPDE \eqref{eq:Phi42} stay close to solutions
of the deterministic PDE \eqref{eq:det_AC} suggesting that \eqref{eq:Phi42} may indeed be the natural small noise perturbation of \eqref{eq:det_AC}.

In this article we consider \eqref{eq:Phi42}  over a $2$-dimensional torus $\TT^2 = \RR^2/L\ZZ^2$ for $L<2 \pi$. 
It is known that under this assumption on the torus size $L$, the deterministic equation \eqref{eq:det_AC} has exactly
3 stationary solutions, namely the constant profiles $-1,0,1$  (see \cite[Appendix B.1]{KORV07}). Here $\pm 1$  are stable minimisers of $V$ and $0$ is unstable. 
We prove that in the small noise regime with overwhelming probability solutions that start close to the same stable minimiser $\pm1$ contract exponentially fast.
The exponential contraction rate is arbitrarily close to $2$, the second derivative of the double-well $x \mapsto \frac14 x^4 - \frac12 x^2$ in $\pm 1$.
This is precisely the $2$-dimensional version of \cite[Corollary 3.1]{MOS89}. 

On a technical level we work with the Da Prato--Debussche decomposition (see Section \ref{s:prel} for more details). An immediate observation is that
differences of any two profiles have much better regularity than the solutions themselves. We split the time axis into random ``good'' and ``bad'' 
intervals depending on whether a reference profile is close to $\pm1$ or not. The key idea is that on ``good'' intervals solutions should contract exponentially,
while they should not diverge too fast on ``bad'' intervals. Furthermore,  ``good'' intervals should be much longer than ``bad'' intervals.

The control on the ``good'' intervals is relatively straightforward: the exponential contraction follows by linearising the equation and the fact that 
these intervals are typically long follows from exponential moment bounds on the explicit stochastic objects appearing in the Da Prato--Debussche approach.
The control on the ``bad'' intervals is much more involved:  in the $1$-dimensional case two  profiles cannot diverge too fast, because the 
second derivative of the double-well potential  is bounded from below. But in the $2$-dimensional case, where solutions are distribution-valued, there is no 
obvious counterpart of this property. Instead we use a strong a priori estimate obtained in our previous work 
\cite{TW18} and the local Lipschitz continuity of the non-linearity. Ultimately this yields an exponential growth bound where the exponential 
rate is given by a polynomial in the explicit stochastic objects. We use a large deviation estimate to prove that these intervals cannot be too long.
In the final step we show that the exponential contraction holds for all $t$ if a certain random walk with positive drift stays positive for all times. 
This random walk is then analysed using techniques developed for  the  classical Cram\'er--Lunberg model in risk theory.

The original motivation for our work was to prove  an Eyring--Kramers law for the transition times of $X$. In \cite{BGW17} Berglund, Di Ges\'u and the second named author 
studied spectral Galerkin approximations $X_N$ of \eqref{eq:Phi42} and obtained explicit estimates on the expected first transition times from a neighbourhood 
of $-1$ to a neighbourhood of $1$. These estimates give a precise asymptotic as $\eps\to 0$ and hold uniformly in the discretisation parameter $N$. 
Their method was based on the potential theoretic approach developed in the finite-dimensional context by Bovier et al. in \cite{BEGK04}. This approach relies heavily on the reversibility of the dynamics
and  provides explicit formulas for the expected transition times in terms of certain integrals of the reversible measure. The key observation in \cite{BGW17} was that in the context
of \eqref{eq:AC} these integrals can be analysed uniformly in the parameter $N$ using the classical Nelson's estimate \cite{Ne73} from constructive Quantum Field Theory.
However, the result in \cite{BGW17} was not optimal for the following two reasons: First, it does not allow 
to pass to the limit as $N\to\infty$ to retrieve the estimate for the transition times of $X$. Second, and more important, the bounds could only be obtained for a certain  
$N$-dependent choice of initial distribution on the neighbourhood of $-1$. This problem is inherent to the potential theoretic approach, which only yields an exact 
formula for the diffusion started in this so-called normalised equilibrium measure. In fact, a large part of the original work \cite{BEGK04} was dedicated to removing this
problem using regularity theory for the finite-dimensional transition probabilities.

In this paper we overcome these two barriers. We first justify the passage to the limit $N\to\infty$ based on our previous work \cite{TW18}: we use the strong a priori estimates on the level of the approximation $X_N$ and 
the support theorem obtained there to prove uniform integrability of the transition times of $X_N$.  The only difficulty here comes from the action of 
the Galerkin projection on the non-linearity which does not allow to test the equation with powers greater than $1$. To remove the unnatural assumption on the initial distribution we make use of our main result, the exponential contraction estimate.
This estimate allows us to  couple the solution started with an arbitrary but fixed initial condition with the solution started in the normalised equilibrium measure.

\subsection{Outline}

In Section \ref{s:prel} we briefly review  the solution theory of \eqref{eq:Phi42}. In Section \ref{s:method} we state our main result, Theorem \ref{thm:exp_contr}, and some 
key propositions needed for its proof. In Section \ref{s:exp_contr_proof} we present the proof the main theorem making full use of the auxiliary propositions presented in Section \ref{s:method}. 
These propositions are proved in Sections \ref{s:diff_est_proof} and \ref{s:rwe}. Finally, in Section 
\ref{s:eyr_kram_app}  we apply our main result,  Theorem \ref{thm:exp_contr}, to prove an Eyring--Kramers law for \eqref{eq:Phi42}, generalising 
\cite[Theorem 2.3]{BGW17}. Several known results that are used throughout this article as well as some additional technical statements can be found in 
the \hyperlink{appendix.0}{Appendix}.        

\subsection{Notation}

We fix a torus $\TT^2 = \RR^2/L\ZZ^2$ of size $0<L<2\pi$. All function spaces are defined over $\TT^2$. We write $\CC^\infty$ for the space of smooth 
functions and $L^p$, $p\in[1,\infty]$, for the space of $p$-integrable periodic functions endowed with the usual norm $\|\cdot\|_{L^p}$ and the usual 
interpretation if $p=\infty$. 

We denote by $\BB^\al_{p,q}$ the (inhomogeneous) Besov space of regularity $\al$ and exponents $p,q\in[1,\infty]$ with norm 
$\|\cdot\|_{\BB^\al_{p,q}}$ (see Definition \ref{def:besov}). We write $\CC^\al$ and $\|\cdot\|_{\CC^\al}$ to denote the space 
$\BB^\al_{\infty,\infty}$ and the corresponding norm. Many useful results about Besov spaces that we repeatedly use throughout the article can be found in Appendix 
\hyperref[app:besov_spaces]{A}.

For any Banach space $(V, \|\cdot\|_V)$ we denote by $B_V(x_0;\delta)$ the open ball $\{x\in V: \|x-x_0\|_V < \delta\}$ and by $\bar B_V(x_0;\delta)$ 
its closure.

Throughout this article we write $C$ for a positive constant which might change from line to line. In proofs we sometimes write $\lesssim$ instead of $\leq C$. 
We also write $a\vee b$ and $a\wedge b$ to denote the maximum and the minimum of $a$ and $b$. 

\subsection*{Acknowledgements}

We are grateful to Nils Berglund for suggesting this problem and for many interesting discussions. PT is supported by EPSRC as part of the MASDOC DTC
at the University of Warwick, Grant No. EP/HO23364/1. HW is supported by the Royal Society through the University Research
Fellowship UF140187.

\section{Preliminaries} \label{s:prel}

Fix a probability space  $(\Omega,\mathcal{F},\PP)$ and let $\xi$ be a space-time white noise defined over $\Omega$.
More precisely, $\xi$ is a family $\{\xi(\phi)\}_{\phi\in L^2((0,\infty)\times \TT^2)}$ of centred Gaussian random variables such that 
\begin{equation*}
 \EE\xi(\phi)\xi(\psi) = \lng \phi,\psi\rng_{L^2((0,\infty)\times \TT^2)}.
\end{equation*}
A natural filtration $\{\mathcal{F}_t\}_{t\geq 0}$ is given by the usual augmentation (as in \cite[Chapter 1.4]{RY99}) of  
\begin{equation*}
 \tilde{\mathcal{F}}_t = \sigma\left(\{\xi(\phi):\phi|_{(t,\infty)\times \TT^2}\equiv 0 \} 
 \right), \quad t\geq 0.
\end{equation*}
We interpret solutions of \eqref{eq:Phi42} following \cite{DPD03} and \cite{MW17i}. We write $X( \cdot;x)$ for the solution started in $x$ and 
use the decomposition 
$X(\cdot;x) = v(\cdot;x) + \eps^{\frac{1}{2}}\<1>(\cdot)$ where $\<1>$ solves the stochastic heat equation
\begin{align}
 & \left(\partial_t - (\Delta - 1\right)\big) \<1> = \sqrt{2} \xi \label{eq:1st_wick}\\
 & \<1>\big|_{t=0} = 0. \nonumber
\end{align}
The remainder term $v$ solves  
\begin{align}
 & \left(\partial_t - \Delta\right) v = -v^3 +v - \left(3 v^2 \eps^{\frac{1}{2}}\<1> + 3 v \eps\<2> + \eps^{\frac{3}{2}}\<3> 
 - 2 \eps^{\frac{1}{2}}\<1> \right) \label{eq:rem_eq}\\
 & v\big|_{t=0} = x \nonumber
\end{align}
where $\<2>$, $\<3>$ are the $2$nd and $3$rd Wick powers of $\<1>$. The random distributions $\<2>$ and $\<3>$ can be constructed as limits of $\<1>_N^2 - \Re_N$ and 
$\<1>_N^3 - 3\Re_N \<1>_N$, where $\<1>_N$ is a spatial Galerkin approximation 
of $\<1>$, and $\Re_N$ is a renormalisation constant which diverges logarithmically in the regularisation parameter $N$. The value of $\Re_N$ is
given by
\begin{equation} \label{eq:renorm_constant}
 \Re_N := \lim_{t\to \infty} \EE \<1>_N(t,z)^2 = \frac{1}{L^2} \sum_{|k|\leq N} \frac{1}{\left(2\pi|k|/L\right)^2+1}.
\end{equation}
Note that $\<1>_N(t,z)$ is stationary in the space variable $z$, hence the expectation is independent of $z$. We refer the reader to \cite[Lemma 3.2]{DPD03}, \cite[Section 2]{TW18} for more details on the construction of the Wick powers. We recall that 
$\<1>$, $\<2>$ and $\<3>$ can be realised as continuous processes taking values in $\CC^{-\al}$ for $\al >0$ and that $\PP$-almost surely for every $T>0$,  and $\al'>0$ 
\begin{equation} \label{eq:tree_norm}
\max\left\{ \sup_{t\leq T}\|\<1>(t)\|_{\CC^{-\al}} ,\;  \sup_{t\leq T} (t\wedge 1)^{\al'}\|\<2>(t)\|_{\CC^{-\al}}, \; \sup_{t\leq T} (t\wedge 1)^{2\al'}\|\<3>(t)\|_{\CC^{-\al}}  \right\}  < \infty.
\end{equation} 
The blow-up of $\|\<2>(t)\|_{\CC^{-\al}}$
and $\|\<3>(t)\|_{\CC^{-\al}}$ for $t$ close to $0$ is due to the fact that we define the stochastic objects $\<2>$ and $\<3>$ with zero initial condition, but we work 
with a time-independent renormalisation constant $\Re_N$ (see \eqref{eq:renorm_constant}).
We define the stochastic heat equation with a Laplacian with mass $1$ because
this allows us to prove exponential moment bounds of $\<1>$, $\<2>$ and $\<3>$ which hold uniformly in time (see Proposition \ref{prop:exp_mom}).
Throughout the paper we use $\<n>$ to refer to all the stochastic objects $\<1>$, $\<2>$ and $\<3>$ simultaneously. In this notation \eqref{eq:tree_norm}
turns into
\begin{equation*} 
 \sup_{t\leq T} (t\wedge 1)^{(n-1)\al'}\|\<n>(t)\|_{\CC^{-\al}} < \infty.
\end{equation*} 

We fix $\al_0\in(0,\frac{1}{3})$ (to measure the regularity of the initial condition $x$ in $\CC^{-\al_0}$), $\beta>0$ (to measure the regularity of $v$ in $\CC^\bt$)
and $\gamma >0$ (to measure the blow-up of $\|v(t;x)\|_{\CC^\bt}$ for $t$ close to $0$) such that
\begin{equation} \label{eq:beta_gamma}
 \gamma <\frac{1}{3}, \quad \frac{\al_0+\bt}{2} < \gamma.
\end{equation}
We also assume that $\al'>0$ and $\al>0$ in \eqref{eq:tree_norm} satisfy    
\begin{equation} \label{eq:al_al'}
 \al'<\gamma, \quad \al<\al_0, \quad \frac{\al+\bt}{2} + 2 \gamma < 1.
\end{equation}
In \cite[Theorems 3.3 and 3.9]{TW18}) it was shown that for every $x\in \CC^{-\al_0}$ there exist a unique solution $v\in C\left((0,\infty);\CC^\bt\right)$ of 
\eqref{eq:rem_eq} such that for every $T>0$ 
\begin{equation*}
 \sup_{t\leq T}(t\wedge 1)^\gamma\|v(t;x)\|_{\CC^\bt}<\infty.
\end{equation*}
\begin{remark} In Condition \eqref{eq:beta_gamma}  $\beta$ has to be strictly less than $\frac{2}{3}$. This is necessary if one wants to treat 
all of the terms arising in a fixed point problem for \eqref{eq:rem_eq} with the same norm for $v$.
A simple post-processing of \cite[Theorems 3.3 and 3.9]{TW18} shows that in fact $v$ is continuous in time taking values in $\CC^{2-\lambda}$ for any  $\lambda > \al$. 
\end{remark}

Equations \eqref{eq:1st_wick}, \eqref{eq:rem_eq}  suggest that indeed $X$ can be seen as a perturbation of the Allen-Cahn equation \eqref{eq:det_AC}, 
because the terms $\<1>$, $\<2>$ and $\<3>$ in  \eqref{eq:rem_eq} all appear with a positive power of $\eps$. 
It is important to note that $v$ is  much more regular than $X$. The irregular part of $X(\cdot;x)$ is  $\eps^{\frac{1}{2}}\<1>$. 
Therefore differences of solutions are much more regular than solutions themselves.

We repeatedly work with restarted stochastic terms: we define $\<1>_s$ as the solution of
 \begin{align*}
 & \left(\partial_t - (\Delta - 1\right)\big) \<1>_s = \sqrt{2} \xi ,  \qquad t >s \\
 & \<1>_s\big|_{t=s} = 0, 
\end{align*}
and let $\<2>_s$ and $\<3>_s$ be its Wick powers. By \cite[Proposition 2.3]{TW18} for every $s>0$,  $\<n>_s(s+\cdot)$ are independent of $\mathcal{F}_s$ and equal in law to $\<n>(\cdot)$. 
For $t\geq s$ we can define a restarted remainder $v_s(t;X(s;x))$ through the identity $X(t;x) = v_s(t;X(s;x)) + \eps^{\frac{1}{2}} \<1>_s(t)$. 
Rearranging \eqref{eq:rem_eq} and using the pathwise identities in \cite[Corollary 2.4]{TW18}  one can see that $v_s$ solves 
\begin{align} 
 & \left(\partial_t - \Delta\right) v_s = -v_s^3 + v_s -\left(3 v_s^2 \eps^{\frac{1}{2}}\<1>_s + 3 v_s \eps\<2>_s + \eps^{\frac{3}{2}}\<3>_s - 2 \eps^{\frac{1}{2}}\<1>_s \right)
 \label{eq:rem_eq_restart} \\
 & v_s\big|_{t=s} = X(s;x). \nonumber
\end{align}
In \cite[Theorem 4.2]{TW18} this is used to prove the Markov property for $X(\cdot;x)$.

\section{Main result and methodology} \label{s:method}

Our main result can be expressed as follows. 

\begin{theorem}\label{thm:exp_contr} For every $\kappa >0$ there exist $\delta_0, a_0, C>0$ and $\eps_0\in(0,1)$ such that for every
$\eps\leq \eps_0$ 
\begin{align*}
 & \inf_{\|x - (\pm 1)\|_{\CC^{-\al_0}} \leq \delta_0}\PP\left( \sup_{\|y-x\|_{\CC^{-\al_0}} \leq \delta_0} 
 \frac{\|X(t;y) - X(t;x)\|_{\CC^\bt}}{\|y-x\|_{\CC^{-\al_0}}} \leq C \ee^{-(2-\kappa)t} \text{ for every } t\geq 1\right) \\
 & \quad \geq 1 - \ee^{-a_0/\eps}.
\end{align*}
\end{theorem}

\begin{proof} See Section \ref{s:exp_contr_proof}.
\end{proof}

This theorem is a variant of \cite[Corollary 3.1]{MOS89} in space dimension $d=2$. There the supremum is taken over both $x$ and $y$
inside the probability measure. We also obtain this version of the theorem as a corollary.
\begin{corollary} \label{col:sup_x_y} For every $\kappa>0$ there exist $\delta_0,a_0,C>0$ and $\eps_0\in(0,1)$ such that for every $\eps\leq \eps_0$ 
\begin{equation*}
 \PP\left( \sup_{x,y\in \bar B_{\CC^{-\al_0}}(\pm 1;\delta_0)} 
 \frac{\|X(t;y) - X(t;x)\|_{\CC^\bt}}{\|y-x\|_{\CC^{-\al_0}}} \leq C \ee^{-(2-\kappa)t} \text{ for every } t\geq 1\right)
 \geq 1 - \ee^{-a_0/\eps}.
\end{equation*}
\end{corollary}

\begin{proof} See Section \ref{s:exp_contr_proof}.
\end{proof}

\begin{remark}
The restriction $t\geq 1$ in Theorem \ref{thm:exp_contr} appears only because we measure 
$y-x$ in a lower regularity norm than $X(t;y) - X(t;x)$. To prove the theorem we first prove Theorem \ref{thm:exp_contr_smooth_i.d.} were we assume
that $y-x\in\CC^\bt$ and in this case we prove a bound which holds for every $t>0$. 
\end{remark}

\begin{remark} Theorem \ref{thm:exp_contr} is an asymptotic coupling of solutions that start close to the same minimiser. In  \cite[Proposition 3.4]{MOS89} 
it was shown that in the 1-dimensional case, solutions which start with initial conditions $x$ and $y$ close to different minimisers also contract exponentially fast, 
but only after time $T_\eps\propto\ee^{[(V(0)-V(\pm1))+\eta]/\eps}$ 
for any $\eta>0$. This is the ``typical'' time needed for one of the two profiles to jump close to the other minimiser.  We expect that 
Theorem \ref{thm:exp_contr} and the large deviation theory developed in \cite{HW15} could be combined to prove a similar result in the case $d=2$.   
\end{remark}

We now define two sequences $\{\nu_i(x)\}_{i\geq 1}$ and $\{\rho_i(x)\}_{i\geq 1}$ of stopping times which partition our time axis and allow us to keep track of the time spent close to and away from the 
minimisers $\pm1$ (see Figure  \ref{fig:axis_part} for a sketch). On the ``good'' intervals $[\rho_{i-1}(x), \nu_i(x)]$ we require both the restarted diagrams $\<n>_{\rho_{i-1}(x)}$ to be small
and the restarted remainder $v_{\rho_{i-1}(x)}$ to be close to $\pm 1$. The ``bad'' intervals $[\nu_i(x),\rho_i(x)]$ end when
$X(\cdot;x)$ re-enters a small neighbourhood of the minimisers. The stopping times $\rho_i(x)$ are defined in terms of the $\CC^{-\al_0}$ norm for $X(\cdot;x)$, 
while we define good intervals in terms of the stronger $\CC^\beta$ topology for $v_{\rho_{i-1}(x)}$. To connect the two, we need to 
allow for a blow-up close to the starting point of the ``good'' intervals.
%

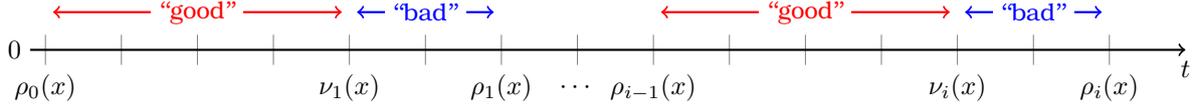
\begin{figure}
 \centering
 \begin{tikzpicture}

  \draw[step=1,very thin,color=gray] (0,-.2) grid (14,.2);
  
  \draw[->, thick] (-.2,0) node[left] {$0$} -- (0,0) -- (15,0);
  
  \path (0,-.5) node {$\rho_0(x)$} -- (4,-.5) node {$\nu_1(x)$} -- (6,-.5) node {$\rho_1(x)$} -- (7,-.5) node {$\cdots$}
  -- (8,-.5) node {$\rho_{i-1}(x)$} -- (12,-.5) node {$\nu_i(x)$} -- (14,-.5) node {$\rho_i(x)$} -- (15,-.5);
  
  \draw (15,-.25) node {$t$};
  
  \draw[<->,thick,red] (.1,.5) -- (2,.5) node[fill=white] {``good''} -- (3.9,.5);
  \draw[<->,thick,blue] (4.1,.5) -- (5,.5) node[fill=white] {``bad''} -- (5.9,.5);
  \draw[<->,thick,red] (8.1,.5) -- (10,.5) node[fill=white] {``good''} -- (11.9,.5);
  \draw[<->,thick,blue] (12.1,.5) -- (13,.5) node[fill=white] {``bad''} -- (13.9,.5);

 \end{tikzpicture}
 \caption{A partition of the time axis with respect to the times $\nu_i(x)$ and $\rho_i(x)$. The ``good'' intervals are ``typically'' much larger than 
 the ``bad'' intervals.}
 \label{fig:axis_part}
\end{figure}

\begin{definition} \label{def:stop_times} For $x\in \CC^{-\al_0}$ we define the sequence of stopping times $\{\rho_i(x)\}_{i\geq 0}$, $\{\nu_i(x)\}_{i\geq 1}$ 
recursively by $\rho_0(x) = 0$ and 
\begin{align*}
 \nu_i(x) & := \inf\Big\{t>\rho_{i-1}(x): \min_{x_*\in\{-1,1\}}((t-\rho_{i-1}(x))\wedge 1)^\gamma \|v_{\rho_{i-1}(x)}(t;X(\rho_i(x);x)) - x_*\|_{\CC^\bt}
 \geq \delta_1 \\
 & \, \quad \text{ or } 
 ((t-\rho_{i-1}(x))\wedge 1)^{(n-1) \al'} \|\eps^{\frac{n}{2}}\<n>_{\rho_{i-1}(x)}(t)\|_{\CC^{-\al}} \geq \delta_2^n\Big\} \\
 \rho_i(x) & := \inf\{t>\nu_i(x): \min_{x_*\in\{-1,1\}} \|X(t;x)- x_*\|_{\CC^{-\al_0}} \leq \delta_0\}.
\end{align*}
\end{definition}
We now define the time increments  
\begin{align}
 & \tau_i(x) = \nu_i(x) - \rho_{i-1}(x). \label{eq:tau_sigma} \\
 & \sigma_i(x) = \rho_i(x) - \nu_i(x). \nonumber
\end{align}
The process $X( \cdot; x)$ is expected to spend long time intervals close to the minimisers $\pm 1$, which corresponds to large values of $\tau_i(x)$. Large
values of $\sigma_i(x)$ are ``atypical''. This behaviour is established Propositions \ref{prop:f_i} and \ref{prop:ent_est}.

The following proposition shows contraction on the ``good'' intervals.  We distinguish between the cases 
\eqref{eq:contr_bd} and \eqref{eq:contr_bd_rough_i.d.} for $y-x$ that lie in $\CC^\bt$ and $\CC^{-\al_0}$ respectively. The Da Prato--Debussche decomposition
shows that differences of any two profiles lie in $\CC^\bt$ for any $t>0$ but at $t=0$ they maintain the irregularity of the initial conditions. Hence we only use 
\eqref{eq:contr_bd_rough_i.d.} on the first ``good'' interval.

\begin{proposition} \label{prop:contr_bd} For every $\kappa>0$ there exist $\delta_0,\delta_1,\delta_2>0$ and $C>0$ such that 
if $\|x-(\pm 1)\|_{\CC^{-\al_0}}\leq \delta_0$ and $y-x\in \CC^\bt$, $\|y-x\|_{\CC^\bt} \leq \delta_0$ then
\begin{equation} \label{eq:contr_bd}
 \|X(t;y) - X(t;x)\|_{\CC^\bt} 
 \leq C\exp\left\{-\left(2-\frac{\kappa}{2}\right) t\right\} \|y-x\|_{\CC^\bt}
\end{equation}
for every $t\leq \tau_1(x)$ defined with respect to $\delta_1$ and $\delta_2$. If we only assume that $\|y-x\|_{\CC^{-\al_0}}\leq \delta_0$ then 
\begin{equation} \label{eq:contr_bd_rough_i.d.}
 (t\wedge1)^\gamma\|X(t;y) - X(t;x)\|_{\CC^\bt} \leq C \exp\left\{-\left(2-\frac{\kappa}{2}\right) t\right\} \|y-x\|_{\CC^{-\al_0}}
\end{equation}
for every $t\leq \tau_1(x)$.
\end{proposition}

\begin{proof} See Section \ref{s:contr_bd_proof}. 
\end{proof}

Our next aim is to control the growth of the differences on the ``bad'' intervals in terms of the stochastic objects $\<n>$. This is done by
partitioning the intervals $[\nu_i(x), \rho_i(x)]$  into tiles of length one. To achieve independence we restart the stochastic objects at the starting point 
of each tile.

\begin{definition} \label{def:L_k}  For $k\geq 0$ and $\rho \geq \nu \geq 0$ let $t_k = \nu+k$. For $k \geq 1$ we define a random variable $L_k(\nu,\rho)$ by 
\begin{equation}\label{eq:L_k}
L_k(\nu,\rho) := \left(\sup_{t\in[t_{k-1},t_k\wedge \rho]} (t-t_{k-1})^{(n-1)\al'} \|\eps^{\frac{n}{2}}\<n>_{t_{k-1}}(t)\|_{\CC^{-\al}}\right)^{\frac{2}{n}}.
\end{equation}
%
%
\end{definition} 

In  our analysis we use a second tiling defined by setting $s_k=  t_k + \frac12 $, i.e. the tiles $[t_k,t_{k+1}]$ and $[s_k, s_{k+1}]$ 
overlap. In order to bound $X(t;y) - X(t;x)$ on a time interval $[t_k,s_k]$ we restart the stochastic objects at $s_{k-1}$ and write
$X(t;y) - X(t;x) = v_{s_{k-1}}(t;X(s_{k-1};y)) - v_{s_{k-1}}(t;X(s_{k-1};x))$. In Lemma \ref{lem:rem_bnd} we upgrade the strong a priori 
bound obtained in \cite[Proposition 3.7]{TW18} to get a control on the $\CC^\bt$ norm of both remainders.
This bound holds uniformly over all possible values of $X(s_{k-1};y)$ and $X(s_{k-1};x)$ and while the bound allows for a blow-up
for times $t$ close to $s_{k-1}$ it holds uniformly over all times in $[t_k,s_k]$. Ultimately, the bound only depends on 
$L_k(\nu+\frac{1}{2},\rho)$ in a polynomial way as shown in Figure \ref{fig:v_bds}. Then we can use the local Lipschitz property of the non-linearity in \eqref{eq:rem_eq} to bound the exponential growth rate of 
$X(t;y) - X(t;x)$. For the first interval $[t_0,t_1]$ we do not use this trick, because we want to avoid bounds that depend on 
the realisation of the white noise outside of $[\nu, \rho]$. On this interval, we make use of an a priori assumption that 
we have some control on $\|X(\nu;y)\|_{\CC^{-\al_0}}$ and $\|X(\nu;x)\|_{\CC^{-\al_0}}$.

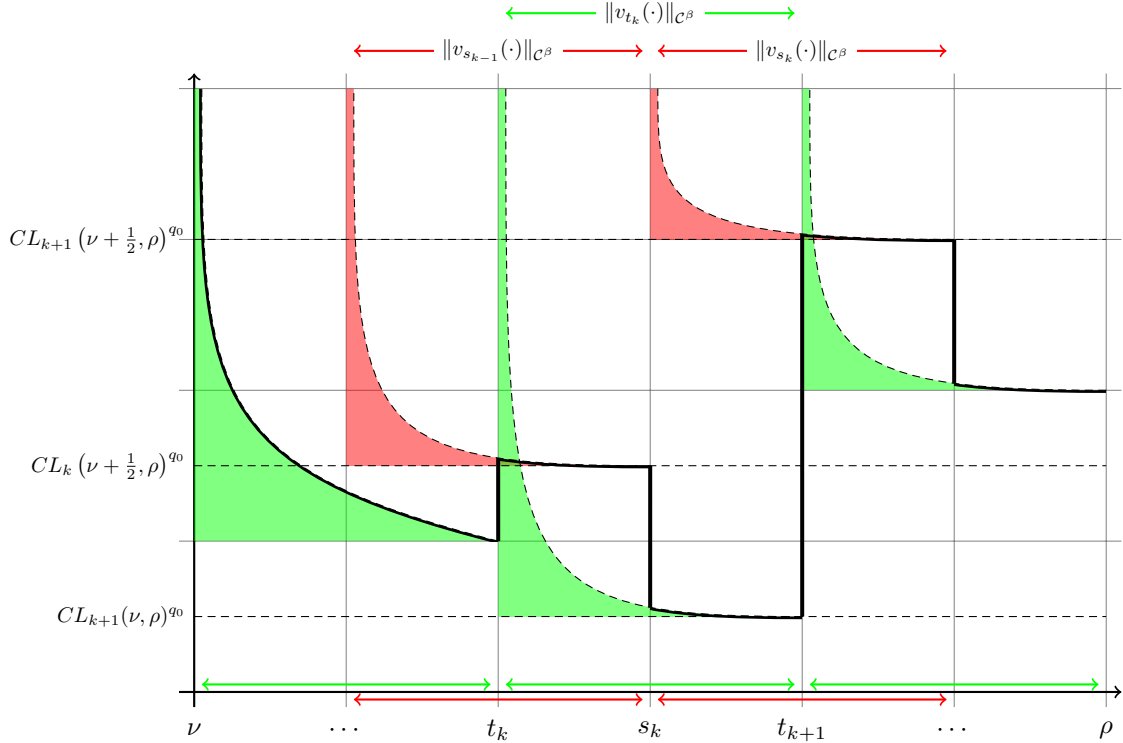
\begin{figure}
 \centering
 \begin{tikzpicture}
  
  \draw[step=2,very thin,color=gray] (-.2,-.2) grid (12.2,8.2);
  
  \draw[thick,->] (-.2,0) -- (12.2,0);
  
  \draw[thick,->] (0,-.2) -- (0,8.2);
  
  \path (0,-.5) node {$\nu$} -- (2,-.5) node {$\cdots$} -- (4,-.5) node {$t_k$} -- (6,-.5) node {$s_k$} -- (8,-.5) node {$t_{k+1}$}
  -- (10,-.5) node {$\cdots$} -- (12,-.5) node {$\rho$};
  
  \begin{scope}
  \clip (0,0) -- (4,0) -- (4,2) .. controls (.1,3) and (.1,4) .. (.1,8) -- (0,8) -- (0,0);
  \clip (0,2) -- (4,2) -- (4,8) -- (0,8) -- cycle;
  \fill[green,opacity=.5] (0,2) -- (4,2) -- (4,8) -- (0,8) -- cycle;
  \end{scope}
  
  \begin{scope}
  \clip (2,0) -- (6,0) -- (6,3) .. controls (2.1,3) and (2.1,3.5) .. (2.1,8) --(2,8) -- (2,0);
  \clip (2,3) -- (6,3) -- (6,8) -- (2,8) -- cycle;
  \fill[red,opacity=.5] (2,3) -- (6,3) -- (6,8) -- (2,8) -- cycle;
  \end{scope}
  
  \begin{scope}
  \clip (4,0) -- (8,0) -- (8,1) .. controls (4.1,1) and (4.1,1.5) .. (4.1,8) -- (4,8) -- (4,0);
  \clip (4,1) -- (8,1) -- (8,8) -- (4,8) -- cycle;
  \fill[green,opacity=.5] ((4,1) -- (8,1) -- (8,8) -- (4,8) -- cycle;
  \end{scope}
  
  \begin{scope}
  \clip (6,0) -- (10,0) -- (10,6) .. controls (6.1,6) and (6.1,6.5) .. (6.1,8) --(6,8) -- (6,0);
  \clip (6,6) -- (10,6) -- (10,8) -- (6,8) -- cycle;
  \fill[red,opacity=.5] (6,6) -- (10,6) -- (10,8) -- (6,8) -- cycle;
  \end{scope}
  
  \begin{scope}
  \clip (8,0) -- (12,0) -- (12,4) .. controls (8.1,4) and (8.1,4.5) .. (8.1,8) -- (8,8) -- (8,0);
  \clip (8,4) -- (12,4) -- (12,8) -- (8,8) -- cycle;
  \fill[green,opacity=.5] (8,4) -- (12,4) -- (12,8) -- (8,8) -- cycle;
  \end{scope}
  
  \draw[name path=c1, densely dashed] (4,2) .. controls (.1,3) and (.1,4) .. (.1,8);
  \draw[name path=c3, densely dashed] (8,1) .. controls (4.1,1) and (4.1,1.5) .. (4.1,8);
  \draw[name path=c5, densely dashed] (12,4) .. controls (8.1,4) and (8.1,4.5) .. (8.1,8);
  
  \draw[name path=c2, densely dashed] (6,3) .. controls (2.1,3) and (2.1,3.5) .. (2.1,8);
  \draw[name path=c4, densely dashed] (10,6) .. controls (6.1,6) and (6.1,6.5) .. (6.1,8);
  
  \path[name path = x1] (2,0) -- (2,8);
  \path[name path = x2] (4,0) -- (4,8);
  \path[name path = x3] (6,0) -- (6,8);
  \path[name path = x4] (8,0) -- (8,8);
  \path[name path = x5] (10,0) -- (10,8);
  
  \path[name intersections={of=c1 and x1,by=int1}];
  \path[name intersections={of=c2 and x2,by=int2}];
  \path[name intersections={of=c3 and x3,by=int3}];
  \path[name intersections={of=c4 and x4,by=int4}];
  \path[name intersections={of=c5 and x5,by=int5}];
  
  \draw[line width=1.5pt] (4,2) -- (int2);
  \draw[line width=1.5pt] (6,3) -- (int3);
  \draw[line width=1.5pt] (8,1) -- (int4);
  \draw[line width=1.5pt] (10,6) -- (int5);
  
  \begin{scope}
  \clip (0,0) -- (4,0) -- (4,2) .. controls (.1,3) and (.1,4) .. (.1,8) -- (0,8) -- (0,0);
  \clip (0,2) -- (4,2) -- (4,8) -- (0,8) -- cycle;
  \draw[line width=2pt] (4,2) .. controls (.1,3) and (.1,4) .. (.1,8);  
  \end{scope}
  
  \begin{scope} 
  \clip (2,0) -- (6,0) -- (6,3) .. controls (2.1,3) and (2.1,3.5) .. (2.1,8) --(2,8) -- (2,0);
  \clip (4,2) -- (6,2) -- (6,8) -- (4,8) -- cycle;
  \draw[line width = 2pt] (6,3) .. controls (2.1,3) and (2.1,3.5) .. (2.1,8);
  \end{scope}
  
  \begin{scope}
  \clip (4,0) -- (8,0) -- (8,1) .. controls (4.1,1) and (4.1,1.5) .. (4.1,8) -- (4,8) -- (4,0);
  \clip (6,0) -- (8,0) -- (8,8) -- (6,8) -- cycle;
  \draw[line width = 2pt] (8,1) .. controls (4.1,1) and (4.1,1.5) .. (4.1,8);
  \end{scope}
  
  \begin{scope}
  \clip (6,0) -- (10,0) -- (10,6) .. controls (6.1,6) and (6.1,6.5) .. (6.1,8) --(6,8) -- (6,0);
  \clip (8,5) -- (10,5) -- (10,8) -- (8,8) -- cycle;
  \draw[line width = 2pt] (10,6) .. controls (6.1,6) and (6.1,6.5) .. (6.1,8);
  \end{scope}
  
  \begin{scope}
  \clip (8,0) -- (12,0) -- (12,4) .. controls (8.1,4) and (8.1,4.5) .. (8.1,8) -- (8,8) -- (8,0);
  \clip (10,3) -- (12,3) -- (12,8) -- (10,8) -- cycle;
  \draw[line width = 2pt] (12,4) .. controls (8.1,4) and (8.1,4.5) .. (8.1,8);
  \end{scope}
  
  \draw[<->,red,thick] (2.1,8.5) -- (4,8.5) node[fill=white] {\scalebox{.8}{\color{black} $\|v_{s_{k-1}}(\cdot)\|_{\CC^\bt}$}} -- (5.9,8.5);
  
  \draw[<->,red,thick] (6.1,8.5) -- (8,8.5) node[fill=white] {\scalebox{.8}{\color{black} $\|v_{s_k}(\cdot)\|_{\CC^\bt}$}} -- (9.9,8.5);
  
  \draw[<->,green,thick] (4.1,9) -- (6,9) node[fill=white] {\scalebox{.8}{\color{black}$\|v_{t_k}(\cdot)\|_{\CC^\bt}$}} -- (7.9,9);
  
  \draw[densely dashed] (0,3) node[left] {\scalebox{.8}{$C L_k\left(\nu+\frac{1}{2},\rho\right)^{q_0}$}} -- (12,3);
  \draw[densely dashed] (0,1) node[left] {\scalebox{.8}{$C L_{k+1}(\nu,\rho)^{q_0}$}} -- (12,1);
  \draw[densely dashed] (0,6) node[left] {\scalebox{.8}{$C L_{k+1}\left(\nu+\frac{1}{2},\rho\right)^{q_0}$}} -- (12,6);
  
  \draw[<->,green,thick] (.1,.1) -- (3.9,.1); 
  
  \draw[<->,red,thick] (2.1,-.1) -- (5.9,-.1);
  
  \draw[<->,green,thick] (4.1,.1) -- (7.9,.1);
  
  \draw[<->,red,thick] (6.1,-.1) -- (9.9,-.1);
  
  \draw[<->,green,thick] (8.1,.1) -- (11.9,.1);
  
 \end{tikzpicture}
 \caption{Bounds on the $\CC^\bt$ norm of the restarted remainder $v$ on the overlapping tiles of the partition of $[\nu,\rho]$. On a
 time interval $[t_k,s_k]$ we restart the stochastic objects at time $s_{k-1}$ and bound $v_{s_{k-1}}$ by a polynomial function of
 $L_k\left(\nu+\frac{1}{2},\rho\right)$. On a time interval $[s_k,t_{k+1}]$ we restart the stochastic objects at time $t_k$ and bound 
 $v_{t_k}$ by a polynomial function of $L_k\left(\nu,\rho\right)$.} 
 \label{fig:v_bds}
\end{figure}

\begin{proposition}\label{prop:sep_bd} Let $R>0$. Then there exists a constant $C\equiv C(R)>0$ such that for every $\|X(\nu;x)\|_{\CC^{-\al_0}}$, $\|X(\nu;y)\|_{\CC^{-\al_0}} \leq R$, 
$\rho>\nu\geq0$ and $t\in [\nu,\rho]$ 
\begin{equation} \label{eq:sep_bd}
 \|X(t;y) - X(t;x)\|_{\CC^\bt} \leq C \exp\left\{L(\nu,\rho;t-\nu)\right\} \|X(\nu;y) - X(\nu;x)\|_{\CC^\bt},
\end{equation}
where 
\begin{equation} \label{eq:L_form}
L(\nu,\rho;t-\nu) = \frac{c_0}{2} \sum_{k=1}^{\lfloor t-\nu \rfloor} 
\sum_{l=0,\frac{1}{2}} \left(1\vee L_k(\nu+l,\rho)\right)^{p_0} + L_0 (t-\nu) 
\end{equation}
for $L_k$ as in \eqref{eq:L_k}, and for some constants $p_0\geq 1$ and $c_0\equiv c_0(R),L_0\equiv L_0(R)\geq 0$.
\end{proposition}

\begin{proof} See Section \ref{s:sep_bd_proof}.
\end{proof}

If we assume that $y-x\in \CC^\bt$, combining the estimates in Propositions \ref{prop:contr_bd} and \ref{prop:sep_bd} suggest the bound 
\begin{align} 
 & \|X(\rho_N(x);y) - X(\rho_N(x);x)\|_{\CC^\bt} \label{eq:formal_gluing} \\
 & \quad \leq \exp\left\{\sum_{i\leq N} \left[-\left(2-\frac{\kappa}{2}\right) \tau_i + L(\nu_i(x),\rho_i(x);\sigma_i(x))+ 2\log C\right]\right\}
 \|y-x\|_{\CC^\bt}, \nonumber
\end{align}
for any $N\geq 1$. If we can show that the exponents satisfy
\begin{equation*} 
 \sum_{i\leq N}\left[-\left(2-\frac{\kappa}{2}\right) \tau_i + L(\nu_i(x),\rho_i(x);\sigma_i(x))+2\log C\right] \leq 
 -(2-\kappa)\rho_N(x),
\end{equation*}
then \eqref{eq:formal_gluing} yields exponential contraction at time $\rho_N(x)$ with rate $2-\kappa$. The difference of the
right hand side and the left hand side of the last inequality is given by the random walk $S_N(x)$ in the next definition. 


\begin{definition} \label{def:rw} Let $\|x-(\pm 1)\|_{\CC^{-\al_0}} \leq \delta_0$. We define the random 
walk $(S_N(x))_{N\geq 1}$ by 
\begin{equation*}
 S_N(x) := \sum_{i\leq N} \left[\frac{\kappa}{2} \tau_i(x) - \big(L(\nu_i(x),\rho_i(x);\sigma_i(x)) + (2-\kappa) \sigma_i(x) + M_0\big)\right]
\end{equation*}
where $M_0 = 2\log C$ for $C>0$ as in Propositions \ref{prop:contr_bd} and \ref{prop:sep_bd}.  
\end{definition}

The next proposition shows that the random walk $S_N(x)$ stays positive for every $N \geq 1$ with overwhelming probability (see Figure 
\ref{fig:rw_plots} for an illustration). The proof is based on a variant of the classical Cram\'er--Lundberg model in risk theory (see \cite[Chapter 1.2]{EKM97}).
In this classical model a random walk $S_N=\sum_{i\leq N}(f_i - g_i)$ with i.i.d. exponential random variables $f_i$ and 
i.i.d. non-negative random variables $g_i$ is considered. The probability for $S_N$ to stay positive for every $N\geq 1$ can be 
calculated explicitly in terms of the expectations of $f_i$ and $g_i$ using a renewal equation. 
In our case we use the Markov property and Propositions \ref{prop:f_i} and Proposition~\ref{prop:g_i}
to compare the random walk $S_N(x)$ in Definition \ref{def:rw} to this classical case.

\begin{remark} If the family $\{L(\nu_i(x),\rho_i(x);\sigma_i(x)) + (2-\kappa) \sigma_i(x) + M_0\}_{i\geq 1}$ had exponential moments, a simple exponential Chebyshev argument 
would imply the following  proposition without any reference to the Cram\'er--Lundberg model. However, by \eqref{eq:L_k} and \eqref{eq:L_form} 
one sees that $L(\nu_i(x),\rho_i(x);\sigma_i(x))$ is a polynomial of potentially high degree in the explicit stochastic objects (which are themselves 
polynomials of the Gaussian noise $\xi$). Hence, we cannot expect more than stretched exponential moments, and indeed, such bounds are 
established in Proposition \ref{prop:g_i}. In the proof of the next proposition we also use an exponential Chebyshev argument, but only to compare 
$\frac{\kappa}{2}\tau_i(x)$ with a suitable exponential random variable which does not depend on $x$. 
\end{remark}

\begin{figure}
 \centering
 \includegraphics[scale=.5]{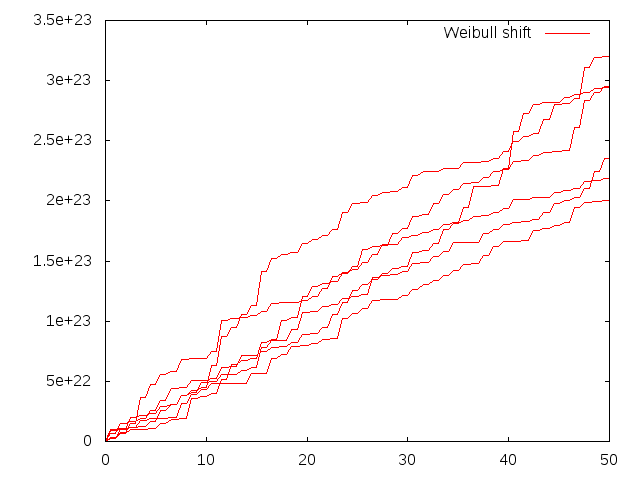}
 \caption{``Typical'' realisations of a random walk $S_N = \sum_{i\leq N} (f_i - g_i)$ for $f_i \sim \ee^{0.5/\eps}\exp(1)$,
 $g_i \sim \ee^{0.1/\eps} \mathrm{Weibull}(0.5,1)$, $N=50$ and $\eps=0.01$. The choice of a Weibull distribution here captures the fact that the random variables 
 $L(\nu_i(x),\rho_i(x);\sigma_i(x)) + (2-\kappa) \sigma_i(x) + M_0$ in Definition \ref{def:rw} have stretched exponential tails as shown in Proposition 
 \ref{prop:g_i}.}
 \label{fig:rw_plots}
\end{figure}

\begin{proposition}\label{prop:rwe} For every $\kappa>0$ there exist $a_0>0$ and $\eps_0\in(0,1)$ such that for every $\eps\leq \eps_0$
\begin{equation}\label{eq:rwe}
 \inf_{\|x - (\pm 1)\|_{\CC^{-\al_0}} \leq \delta_0} \PP(S_N(x) \geq 0  \text{ for every } N\geq 1) 
 \geq 1-\ee^{- a_0/\eps}.
\end{equation}
\end{proposition}

\begin{proof} See Section \ref{s:rwe_proof}.
\end{proof}

\section{Proof of the main theorem} \label{s:exp_contr_proof}

We first treat the case where $y-x\in \CC^\bt$: 
let $x\in \CC^{-\al_0}$ such that $\|x-(\pm 1)\|_{\CC^{-\al_0}} \leq \delta_0$ and let $y$ be such that $y-x\in \CC^\bt$
and $\|y-x\|_{\CC^\bt}\leq \delta_0$. We also write $Y(t) =
X(t;y) - X(t;x)$. We consider the event 
\begin{equation}\label{eq:rw_pos}
 \mathcal{S}(x) = \left\{S_N(x) \geq 0 \text{ for every } N\geq 1\right\}
\end{equation}
for $S_N(x)$ as in Definition \ref{def:rw}.

We first prove the following proposition which provides explicit estimates on the 
differences at the stopping times $\nu_N(x)$ and $\rho_N(x)$ for every $N\geq 1$ and $\omega \in \mathcal{S}(x)$
by iterating Propositions~\ref{prop:contr_bd} and \ref{prop:sep_bd}.
To shorten the notation we drop the explicit dependence on the starting point $x$ in the stopping times  $\nu_N$ and $\rho_N$ 
and the random walk $S_N$.
We also drop the dependence on the realisation $\omega$ but we assume throughout that  $\omega \in \mathcal{S}(x)$. 

\begin{proposition}\label{prop:exp_bds} For any $\kappa>0$ let $C>0$ be as in Proposition \ref{prop:contr_bd}. 
Then for every $\omega\in \mathcal{S}(x)$ and $N\geq 1$
\begin{align}
 \|Y(\nu_N)\|_{\CC^\bt} & \leq C \exp\left\{-S_{N-1}\right\}
 \exp\left\{-\frac{\kappa}{2}\tau_N\right\} \exp\left\{-(2-\kappa) \nu_N\right\} \|Y(0)\|_{\CC^\bt} \label{eq:precontr_1} \\
 \|Y(\rho_N)\|_{\CC^\bt} & \leq \exp\left\{-S_N\right\} \exp\left\{-(2-\kappa)\rho_N\right\}
 \|Y(0)\|_{\CC^\bt} \label{eq:precontr_2}.
\end{align}
\end{proposition}

\begin{proof} We prove our claim by induction on $N\geq 1$, observing that it is obvious for $N=0$.


To prove \eqref{eq:precontr_1} for $N+1$ we first notice that by the definition of $\rho_N$ we have that $\|X^\eps(\rho_N;x)-(\pm 1)\|_{\CC^{-\al_0}}\leq \delta_0$ 
and since $\omega\in \mathcal{S}(x)$ \eqref{eq:precontr_2} implies that $\|Y(\rho_N)\|_{\CC^\bt} \leq \delta_0$. Hence we can use \eqref{eq:contr_bd} to get
\begin{equation*}
 \|Y(\nu_{N+1})\|_{\CC^\bt} 
 \lesssim \exp\left\{-\frac{\kappa}{2}\tau_{N+1}\right\} \exp\left\{-(2-\kappa)\tau_{N+1}\right\} 
 \|Y(\rho_{N})\|_{\CC^\bt}.
\end{equation*}
Combining with the estimate on $\|Y(\rho_N)\|_{\CC^\bt}$ the above implies \eqref{eq:precontr_1} for $N+1$.

To prove \eqref{eq:precontr_2} for $N+1$ we first notice that by Proposition \ref{prop:v_ex_est} 
$\|X(\nu_{N+1};x)\|_{\CC^{-\al_0}}\leq 2\delta_1+1$. This bound, \eqref{eq:precontr_1} for $N+1$ and the triangle inequality imply that 
$\|X(\nu_{N+1};y)\|_{\CC^{-\al_0}} \leq \delta_0 + 2\delta_1+1$. Hence we can use Proposition \ref{prop:sep_bd} 
for $\nu = \nu_{N+1}$, $\rho = \rho_{N+1}$ and $R=\delta_0 + 2\delta_1+1$ to obtain
\begin{equation*}
 \|Y(\rho_{N+1})\|_{\CC^\bt} \lesssim \exp\left\{L(\nu_{N+1},\rho_{N+1};\sigma_{N+1})\right\} \|Y(\nu_{N+1})\|_{\CC^\bt}.
\end{equation*}
If we combine with \eqref{eq:precontr_1} for $N+1$ we have that
\begin{align*}
 \|Y(\rho_{N+1})\|_{\CC^\bt} & \leq \exp\left\{L(\nu_{N+1},\rho_{N+1};\sigma_{N+1})+M_0\right\}
 \exp\left\{-S_{N}\right\} \\
 & \quad \times \exp\left\{-\frac{\kappa}{2}\tau_{N+1}\right\} \exp\left\{-(2-\kappa) \nu_{N+1}\right\} \|Y(0)\|_{\CC^\bt}.
\end{align*}
We then rearrange the terms to obtain \eqref{eq:precontr_2}, which completes the proof.
\end{proof}

We are ready to prove the following version of Theorem \ref{thm:exp_contr} for sufficiently smooth initial conditions.

\begin{theorem} \label{thm:exp_contr_smooth_i.d.} For every $\kappa>0$ there exist $\delta_0, a_0, C>0$ and 
$\eps_0\in(0,1)$ such that for every $\eps\leq \eps_0$ 
\begin{align*}
 & \inf_{\|x - (\pm 1)\|_{\CC^{-\al_0}} \leq \delta_0}\PP\left( \sup_{\substack{y-x\in \CC^\bt \\ 
 \|y-x\|_{\CC^\bt} \leq \delta_0}} 
 \frac{\|X(t;y) - X(t;x)\|_{\CC^\bt}}{\|y-x\|_{\CC^\bt}} \leq C \ee^{-(2-\kappa)t} \text{ for every } t\geq 0\right) \\
 & \quad \geq 1 - \ee^{-a_0/\eps}.
\end{align*}
\end{theorem}

\begin{proof} 
Let $\omega \in \mathcal{S}(x)$ as in \eqref{eq:rw_pos}. For any 
$t>0$ there exists $N\equiv N(\omega)\geq 0$ such that $t\in [\rho_N, \nu_{N+1})$ or $t\in [\nu_{N+1}, \rho_{N+1})$. 

If $t\in [\rho_N, \nu_{N+1})$ then
\begin{align*}
 &\|X(t;y) - X(t;x)\|_{\CC^\bt} \\
 & \quad \stackrel{\eqref{eq:contr_bd},\eqref{eq:precontr_2}}{\lesssim} \exp\left\{-\left(2-\frac{\kappa}{2}\right) (t - \rho_N)\right\} \|X(\rho_N;y)) - X(\rho_N;x)\|_{\CC^\bt} \\
 & \quad = \exp\left\{-\frac{\kappa}{2}(t-\rho_N)\right\} \exp\left\{-(2-\kappa)(t-\rho_N) \right\}
 \|X(\rho_N;y)) - X(\rho_N;y)\|_{\CC^\bt} \\
 & \quad \stackrel{\eqref{eq:precontr_2}}{\lesssim} \exp\left\{-(2-\kappa) t \right\} \|y - x\|_{\CC^\bt}. 
\end{align*}
%

If $t\in [\nu_{N+1}, \rho_{N+1})$ then 
\begin{align*}
 \|X(t;y) - X(t;x)\|_{\CC^\bt} & \stackrel{\eqref{eq:sep_bd}}{\lesssim}  \exp\{L(\nu_{N+1},\rho_{N+1};t-\nu_{N+1})\}
 \|X(\nu_{N+1};y) - X(\nu_{N+1};x)\|_{\CC^\bt} \\
 & = \exp\{L(\nu_{N+1},\rho_{N+1};t-\nu_{N+1})+(2-\kappa)(t-\nu_{N+1})\} \\
 & \quad \times \exp\{-(2-\kappa)(t-\nu_{N+1})\} \|X(\nu_{N+1};y) - X(\nu_{N+1};x)\|_{\CC^\bt} \\
 & \stackrel{\eqref{eq:precontr_1},\omega\in \mathcal{S}(x)}{\lesssim} \exp\{-(2-\kappa)t\} \|y - x\|_{\CC^\bt}.
\end{align*}
%
%
%
By Proposition \ref{prop:rwe} there exist $a_0>0$ and $\eps_0\in(0,1)$ such that for every $\eps\leq \eps_0$
\begin{equation*}
 \inf_{\|x - (\pm 1)\|_{\CC^{-\al_0}} \leq \delta_0}\PP(\mathcal{S}(x)) \geq 1 - \ee^{-a_0/\eps}
\end{equation*}
which completes the proof.
\end{proof}

We are now ready to prove Theorem \ref{thm:exp_contr} and Corollary \ref{col:sup_x_y}.

\begin{proof}[Proof of Theorem \ref{thm:exp_contr}] This is a consequence of \eqref{eq:contr_bd_rough_i.d.},
Proposition \ref{prop:f_i} and Theorem \ref{thm:exp_contr_smooth_i.d.}. 
%
%
Let $\delta_1,\delta_2>0$ sufficiently small such that $\delta_1+\delta_2<\delta_0$ and assume that $\tau_1(x)\geq 1$. By the definition of $\tau_1(x)$ 
\begin{equation*}
 \|X(1;x) - (\pm 1)\|_{\CC^{-\al_0}} \leq \|v(1;x) - (\pm1)\|_{\CC^\bt} + \|\eps^{\frac{1}{2}}\<1>(1)\|_{\CC^{-\al_0}} < \delta_1+\delta_2 <\delta_0.
\end{equation*}
If we also choose $\delta_0'< \delta_0$ by \eqref{eq:contr_bd_rough_i.d.} we have that for every $\|y-x\|_{\CC^{-\al_0}}\leq \delta_0'$
\begin{equation*}
 \|X(1;y) - X(1;x)\|_{\CC^\bt} \lesssim \|y-x\|_{\CC^{-\al_0}}.
\end{equation*}
The probability of the event $\{\tau_1(x) \geq 1\}$ can be estimated from below by Proposition \ref{prop:f_i} uniformly in 
$\|x-(\pm1)\|_{\CC^{-\al_0}}\leq \delta_0'$. Combining with Theorem \ref{thm:exp_contr_smooth_i.d.} completes the proof.
\end{proof}

\begin{proof}[Proof of Corollary \ref{col:sup_x_y}] We only prove the case where initial conditions are close to the minimiser $1$. 
We fix $\delta_0',\delta_1'>0$ such that $2\delta_0' <\delta_0$ and $\delta_0'+\delta_1'< \delta_1$. By Proposition 
\ref{prop:v_ex_est} if we chose $\delta_2$ sufficiently small then 
\begin{itemize}
\item $\sup_{t\leq 1}t^{(n-1)\al'} \|\<n>(t)\|_{\CC^{-\al}}\leq \delta_2 \Rightarrow 
\sup_{t\leq 1} t^\gamma\|v(t;y) - 1\|_{\CC^\bt} \leq 
\delta_1$ uniformly for $\|y-1\|_{\CC^{-\al_0}} \leq \delta_0'$.
\end{itemize}
This together with \eqref{eq:contr_bd_rough_i.d.} implies that for every $x,y\in B_{\CC^{-\al_0}}(1;\delta_0')$ 
\begin{equation*}
 \|X(1;y) - X(1;x)\|_{\CC^\bt} \lesssim \|y-x\|_{\CC^{-\al_0}} \lesssim \delta_0'.
\end{equation*}
Let  
\begin{equation*}
 \omega\in \mathcal{S}:= \left\{\sup_{\|y-1\|_{\CC^{-\al_0}}\leq \delta_0'}
 \frac{\|X(t;y) - X(t;1)\|_{\CC^\bt}}{\|y-1\|_{\CC^{-\al_0}}} \leq C \ee^{-(2-\kappa)t} \text{ for every } t\geq 1\right\},
\end{equation*}
$t\geq 1$ and $y\in B_{\CC^{-\al_0}}(-1;\delta_0')$. Then
\begin{itemize}
 \item $\sup_{s \leq t \leq T} (t-s)^\gamma\|v_s(t;X(s;1))- (\pm1)\|_{\CC^\bt}\leq \delta_1'
 \Rightarrow \sup_{s \leq t \leq T} (t-s)^\gamma\|v_s(t;X(s;y))- (\pm1)\|_{\CC^\bt}\leq \delta_1$
 for $T,s\geq 1$.
 \item $\|X(t;1)-(\pm 1)\|_{\CC^{-\al_0}} \leq \delta_0' \Rightarrow \|X(t;y)-(\pm1)\|_{\CC^{-\al_0}} \leq \delta_0$. 
\end{itemize}
This implies that if we consider the process $X(t;y)$ for $t\geq 1$, the times $\nu_i(X(1;y))$ and $\rho_i(X(1;y))$ of
Definition \ref{def:stop_times} for $\delta_0,\delta_1$ and $\delta_2$ can be replaced by the times $\nu_i(X(1;1))$ and $\rho_i(X(1;1))$
for $\delta_0',\delta_1'$ and the same $\delta_2$. Hence the corresponding random walk $S_N(X(1;y))$ in Definition \ref{def:rw} can be 
replaced by $S_N(X(1;1))$.

We can now repeat the proof of Theorem \ref{thm:exp_contr_smooth_i.d.} for the difference $X(t;y) - X(t;x)$, $t\geq 1$, step by step,
replacing the event in \eqref{eq:rw_pos} by 
\begin{equation}\label{eq:rw_pos_x_ind} 
 \mathcal{S}\cap\left\{\sup_{t\leq 1}t^{(n-1)\al'} \|\<n>(t)\|_{\CC^{-\al}}\leq \delta_2, \, S_N(X(1;1))\geq 0 \text{ for every } 
 N\geq 1\right\}.
\end{equation}
This allows us to prove that
\begin{equation*}
 \|X(t;y) - X(t;x)\|_{\CC^\bt} \leq C \ee^{-(2-\kappa)(t-1)} \|X(1;y) - X(1;x)\|_{\CC^\bt}
 \leq C \ee^{-(2-\kappa)t} \|y - x\|_{\CC^{-\al_0}}
\end{equation*}
uniformly in $y,x\in B_{\CC^{-\al_0}}(1;\delta_0')$. 

To estimate the event in \eqref{eq:rw_pos_x_ind} we use Theorem \ref{thm:exp_contr} and Propositions \ref{prop:diag_ex_est} 
and \ref{prop:rwe}. This completes the proof.
\end{proof}

\section{Pathwise estimates on the difference of two profiles} \label{s:diff_est_proof}

In this section we prove Propositions \ref{prop:contr_bd} and Propositions \ref{prop:sep_bd}. Our analysis here is pathwise 
and uses no probabilistic tools. 

\subsection{Proof of Proposition \ref{prop:contr_bd}} \label{s:contr_bd_proof}

\begin{proof}[Proof of Proposition \ref{prop:contr_bd}] 
We only prove \eqref{eq:contr_bd_rough_i.d.}. To prove \eqref{eq:contr_bd} we follow the same strategy as below. However in this case we do not need to 
encounter the blow-up of $\|Y(t)\|_{\CC^\bt}$ close to $0$ and hence we omit the proof since it poses no extra difficulties. 

Let $Y(t) = X(t;y) - X(t;x)$ and notice that from \eqref{eq:rem_eq} we get
\begin{equation*}
 (\partial_t-\Delta) Y = -\left(v(\cdot;y)^3 - v(\cdot;x)^3\right) + Y - 3 (v(\cdot;y) +v(\cdot;x)) \eps^{\frac{1}{2}}\<1> Y
 - 3 \eps\<2> Y.
\end{equation*}
We use the identity $v(\cdot;y) = v(\cdot;x) +Y$ to rewrite this equation in the form
\begin{equation*} 
 \left(\partial_t-(\Delta -2)\right) Y = -3 \left(v(\cdot;x)^2 - 1\right) Y +  \mathtt{Error}(v(\cdot;x);Y) - 3 ( Y + 2v(\cdot;x)) \eps^{\frac{1}{2}} \<1> Y
 - 3 \eps \<2> Y
\end{equation*}
where $\mathtt{Error}(v(\cdot;x);Y) = -Y^3 - 3v(\cdot;x) Y^2$ collects all the terms which are higher order in $Y$.  Then 
\begin{align}
 Y(t) & = \ee^{-2t} \ee^{\Delta  t} Y(0) + \int_0^t \ee^{-2(t-s)} \ee^{\Delta (t-s)}\Big[-3\left(v(s;x)^2 - 1 \right) Y(s)+ 
 \mathtt{Error}(v(s;x);Y(s)) \label{eq:Y_mild_form} \\ 
 & \quad - 3 ( Y(s) + 2v(s;x)) \eps^{\frac{1}{2}} \<1>(s) Y(s)
 - 3 \eps \<2>(s) Y(s)\Big] \dd s. \nonumber
\end{align}
We set 
\begin{equation*}
 \tilde \kappa = \sup_{t\leq \tau_1(x)} (t\wedge 1)^{2\gamma}\|-3\left(v(t;x)^2 - 1\right)\|_{\CC^\bt}. 
\end{equation*}
Let $\iota = \inf\{t>0: (t\wedge1)^\gamma\|Y(t)\|_{\CC^\bt} > \zeta\}$ for $1\geq\zeta>\delta_0$ and notice 
that for $t\leq \tau_1(x) \wedge \iota$ using \eqref{eq:Y_mild_form} we get
\begin{align*}
 \|Y(t)\|_{\CC^\bt} & \stackrel{\eqref{eq:Heat_Smooth},\eqref{eq:mult_ineq_1},\eqref{eq:mult_ineq_2}}{\leq} \ee^{-2t} C (t\wedge1)^{-\frac{\al_0+\bt}{2}} \|Y(0)\|_{\CC^{-\al_0}} + \tilde \kappa \int_0^t \ee^{-2(t-s)} (s\wedge 1)^{-2\gamma} \|Y(s)\|_{\CC^\bt} \dd s \\
 & \quad + \zeta \, C_1 \int_0^t \ee^{-2(t-s)} (s\wedge1)^{-2\gamma} \|Y(s)\|_{\CC^\bt} \dd s \\
 & \quad + \delta_2 \, C_2 \int_0^t \ee^{-2(t-s)} (t-s)^{-\frac{\al+\bt}{2}} (s\wedge 1)^{-\gamma} \|Y(s)\|_{\CC^\bt} \dd s \\
 & \quad + \delta_2 \, C_3 \int_0^t \ee^{-2(t-s)} (t-s)^{-\frac{\al+\bt}{2}} (s\wedge 1)^{-\al'} \|Y(s)\|_{\CC^\bt} \dd s
\end{align*}
were we also use that for $s\leq t$
\begin{equation*}
\|\mathtt{Error}(v(s;x);Y(s))\|_{\CC^\bt} \lesssim \zeta s^{-2\gamma}\|Y(s)\|_{\CC^\bt}.
\end{equation*}
Choosing $\zeta \leq \tilde \kappa/ C_1$ and $\delta_2 \leq \tilde \kappa/ C_2\vee C_3$ we have
\begin{equation*}
 \|Y(t)\|_{\CC^\bt} \leq \ee^{-2t} C (t\wedge 1)^{-\frac{\al+\bt}{2}}\|Y(0)\|_{\CC^{-\al_0}} + \tilde \kappa 
 \int_0^t \ee^{-2(t-s)} (t-s)^{-\frac{\al+\bt}{2}} (s\wedge 1)^{-2\gamma} \|Y(s)\|_{\CC^\bt} \dd s.
\end{equation*}
Then for $t\leq \tau_1(x) \wedge \iota$ by Lemma \ref{lem:grwl} on $f(t) = 
(t\wedge 1)^\gamma \|Y(t)\|_{\CC^{\bt}}$ there exist $c>0$ such that 
\begin{equation*}
 (t\wedge 1)^\gamma \|Y(t)\|_{\CC^\bt} \leq C \exp\left\{ - 2 t + c \tilde \kappa^{\frac{1}{1-\frac{\al+\bt}{2}-3\gm}} t + M\right\} 
 \|Y(0)\|_{\CC^{-\al_0}}.
\end{equation*}
We now fix $\delta_1>0$ such that $c \tilde \kappa^{\frac{1}{1-\frac{\al+\bt}{2}-3\gm}} \leq \frac{\kappa}{2}$. This implies that for 
$t\leq \tau_1(x) \wedge \iota$
\begin{equation*}
 (t\wedge 1)^\gamma\|Y(t)\|_{\CC^\bt} \leq C \exp\left\{ - \left(2 - \frac{\kappa}{2}\right) t\right\} \|Y(0)\|_{\CC^{-\al_0}}.
\end{equation*}
Finally choosing $\delta_0$ sufficiently small we furthermore notice that $\tau_1(x) \wedge \iota = \tau_1(x)$ 
which completes the proof of \eqref{eq:contr_bd_rough_i.d.}. 
\end{proof}

\subsection{Proof of Proposition \ref{prop:sep_bd}} \label{s:sep_bd_proof}

Before we proceed to the proof of Proposition \ref{prop:sep_bd} we need the following lemma which upgrades the a priori estimates 
in \cite[Proposition 3.7]{TW18}. Here and below we let $S(t) = \ee^{\Delta t}$.
\begin{lemma} \label{lem:rem_bnd} There exist $\al,\gamma',C>0$ and $p_0\geq 1$ such that if $\sup_{t\leq 1} t^{(n-1)\al'} \|\eps^{\frac{n}{2}}\<n>(t)\|_{\CC^{-\al}} \leq L^n$ then
\begin{equation*}
 \sup_{x\in \CC^{-\al_0}}\sup_{t\leq 1} t^{\gamma'} \|v(t;x)\|_{\CC^\bt} \leq C (1\vee L)^{p_0}.
\end{equation*}
\end{lemma}

\begin{proof} Throughout this proof we simply write $v(t)$ to denote $v(t;x)$. By \cite[Proposition 3.7]{TW18} we have that for every $p\geq 2$ even 
\begin{equation}\label{eq:Lp_est}
 \sup_{x\in \CC^{-\al_0}}\sup_{t\leq 1} t^\frac{1}{2}\|v(t)\|_{L^p} \leq C 
 \left(1\vee \sup_{t\leq 1}t^{(n-1)\al'p_n}\|\eps^{\frac{n}{2}}\<n>(t)\|_{\CC^{-\al}}^{p_n}\right)
\end{equation}
for some exponents $p_n\geq 1$. Combining  \cite[Equations (3.13) and (3.22)]{TW18}  
and integrating from $s$ to $t$ we obtain
\begin{equation*}
 \|v(t)\|_{L^2}^2 - \|v(s)\|_{L^2}^2 + \int_s^t \|\nabla v(r)\|_{L^2}^2 \dd r 
 \leq C \int_s^t \Big(1+\sum_{n\leq 3}\|\eps^{\frac{n}{2}}\<n>(r)\|_{\CC^{-\al}}^{p_n}\Big) \dd r
\end{equation*}
which implies that
\begin{equation}\label{eq:grad_est}
 \int_s^t \|\nabla v(r)\|_{L^2}^2 \dd r 
 \leq C \int_s^t \Big(1+\sum_{n\leq 3}\|\eps^{\frac{n}{2}}\<n>(r)\|_{\CC^{-\al}}^{p_n}\Big) \dd r + \|v(s)\|_{L^2}^2.
\end{equation}
Using the mild form of \eqref{eq:rem_eq} we have for $1\geq t>s>0$
\begin{align}
 \|v(t)\|_{\CC^\bt} & \lesssim \underbrace{\|S(t-s)v(s)\|_{\CC^\bt}}_{=:I_1} + 
 \underbrace{\int_s^t \|S(t-r)v(r)^3\|_{\CC^\bt} \dd r}_{=:I_2} 
 + \underbrace{\int_s^t \|S(t-r)\left(v(r)^2 \eps^{\frac{1}{2}} \<1>(r)\right)\|_{\CC^\bt} \dd r}_{=:I_3} \label{eq:v_mild_bd} \\
 & \quad + \underbrace{\int_s^t \|S(t-r)\left(v(r) \eps \<2>(r)\right)\|_{\CC^\bt} \dd r}_{=:I_4} 
 + \underbrace{\int_s^t \|S(t-r) \eps^{\frac{3}{2}}\<3>(r)\|_{\CC^\bt} \dd r}_{=:I_5} \nonumber \\
 & \quad + \underbrace{\int_s^t \|S(t-r) \eps^{\frac{1}{2}}\<1>(r)\|_{\CC^\bt} \dd r}_{=:I_6} 
 + \underbrace{\int_s^t \|S(t-r)v(r)\|_{\CC^\bt} \dd r}_{=:I_7}. \nonumber
\end{align}
To estimate $\|v(t)\|_{\CC^\bt}$ we use the $L^p$ bound \eqref{eq:Lp_est}, the energy inequality \eqref{eq:grad_est} and the embedding 
$\BB^1_{2,\infty}$ to bound the terms appearing on the right hand side of the last inequality as shown below.

We treat each term in \eqref{eq:v_mild_bd} separately. Below $p$ may change from term to term and $\al,\lambda$ can be taken arbitrarily small. We write $p_1$ 
and $p_2$ for conjugate exponents of $p$, i.e. $\frac{1}{p} = \frac{1}{p_1}+\frac{1}{p_2}$. We also denote by $(1\vee L)^{p_0}$ a 
polynomial of degree $p_0\geq 1$ in the variable $1\vee L$ where the value of $p_0$ may change from line to line.

\underline{Term $I_1$}:
\begin{align*}
 I_1 \stackrel{\eqref{eq:Besov_Emb},\eqref{eq:Heat_Smooth}}{\lesssim} (t-s)^{-\frac{\bt+\frac{2}{p}}{2}} \|v(s)\|_{L^p} 
 \stackrel{\eqref{eq:Lp_est}}{\lesssim} (t-s)^{-\frac{\bt+\frac{2}{p}}{2}} s^{-\frac{1}{2}} (1\vee L)^{p_0}.
\end{align*}
\underline{Term $I_2$}: 
\begin{align*}
 I_2 & \stackrel{\eqref{eq:Besov_Emb},\eqref{eq:Heat_Smooth}}{\lesssim} \int_s^t (t-r)^{-\frac{\bt+\frac{2}{p}}{2}} \|v(r)^3\|_{L^p} \dd r
 \stackrel{\eqref{eq:Lp_est}}{\lesssim} (1\vee L)^{p_0} \int_s^t (t-r)^{-\frac{\bt+\frac{2}{p}}{2}} r^{-\frac{3}{2}} \dd r \\
 & \lesssim (1\vee L)^{p_0} s^{-\frac{3}{2}} \int_s^t (t-r)^{-\frac{\bt+\frac{2}{p}}{2}} \dd r.
\end{align*}
\underline{Term $I_3$}:
\begin{align*}
 I_3 & \stackrel{\eqref{eq:Besov_Emb},\eqref{eq:Heat_Smooth},\eqref{eq:mult_ineq_2},\lambda>0}{\lesssim} \int_s^t (t-r)^{-\frac{2\al+\lambda+\frac{2}{p}}{2}} \|v(r)^2\|_{\BB^{\al+\lambda}_{p,\infty}} \|\eps^{\frac{1}{2}}\<1>(r)\|_{\CC^{-\al}} \dd r \\
 & \stackrel{\eqref{eq:Lp_mult}}{\lesssim} \int_s^t (t-r)^{-\frac{2\al+\lambda+\frac{2}{p}}{2}} \|v(r)\|_{L^{p_1}} \|v(r)\|_{\BB^{\al+\lambda}_{p_2,\infty}} 
 \|\eps^{\frac{1}{2}}\<1>(r)\|_{\CC^{-\al}} \dd r \\
 & \stackrel{\eqref{eq:Besov_Emb},\eqref{eq:Lp_est}}{\lesssim} \int_s^t (t-r)^{-\frac{2\al+\lambda+\frac{2}{p}}{2}} r^{-\frac{1}{2}} \|v(r)\|_{\BB^{\al+\lambda+1-\frac{2}{p_2}}_{2,\infty}} 
 \|\eps^{\frac{1}{2}}\<1>(r)\|_{\CC^{-\al}} \dd r \\
 & \stackrel{\eqref{eq:Lp_est},\frac{2}{p_2} = \al+\lambda}{\lesssim} (1\vee L)^{p_0} s^{-\frac{1}{2}}
 \int_s^t  (t-r)^{-\frac{2\al+\lambda+\frac{2}{p}}{2}} \|v(r)\|_{\BB^1_{2,\infty}} \dd r \\
 & \stackrel{\text{Cauchy--Schwarz}}{\lesssim} (1\vee L)^{p_0} s^{-\frac{1}{2}}
 \left(\int_s^t  (t-r)^{-\left(2\al+\lambda+\frac{2}{p}\right)} \dd r \right)^{\frac{1}{2}} 
 \left(\int_s^t \|v(r)\|_{\BB^1_{2,\infty}}^2 \dd r\right)^{\frac{1}{2}}.
\end{align*}
\underline{Term $I_4$}:
\begin{align*}
 I_4 & \stackrel{\eqref{eq:Besov_Emb},\eqref{eq:Heat_Smooth},\eqref{eq:mult_ineq_2},\lambda>0}{\lesssim} \int_s^t (t-r)^{-\frac{2\al+\lambda+\frac{2}{p}}{2}} \|v(r)\|_{\BB^{\al+\lambda}_{p,\infty}} \|\eps\<2>(r)\|_{\CC^{-\al}} \dd r \\
 & \stackrel{\eqref{eq:Besov_Emb}}{\lesssim} \int_s^t (t-r)^{-\frac{2\al+\lambda+\frac{2}{p}}{2}} \|v(r)\|_{\BB^{\al+\lambda+1-\frac{2}{p}}_{2,\infty}} \|\eps\<2>(r)\|_{\CC^{-\al}} \dd r\\
 & \stackrel{\frac{2}{p} = \al+\lambda}{\lesssim} (1\vee L)^{p_0} s^{-\al'} \int_s^t (t-r)^{-\frac{2\al+\lambda+\frac{2}{p}}{2}} \|v(r)\|_{\BB^1_{2,\infty}} \dd r \\
 & \stackrel{\text{Cauchy--Schwarz}}{\lesssim} (1\vee L)^{p_0} s^{-\al'} \left(\int_s^t (t-r)^{-\left(2\al+\lambda+\frac{2}{p}\right)} \dd r\right)^{\frac{1}{2}} 
 \left(\int_s^t \|v(r)\|_{\BB^1_{2,\infty}}^2 \dd r\right)^{\frac{1}{2}}.
\end{align*}
\underline{Term $I_5$}:
\begin{align*}
 I_5 & \stackrel{\eqref{eq:Heat_Smooth}}{\lesssim} \int_s^t (t-r)^{-\frac{\al+\bt}{2}} \|\eps^{\frac{3}{2}}\<3>(r)\|_{\CC^{-\al}} \dd r
 \lesssim (1\vee L)^{p_0} \int_s^t (t-r)^{-\frac{\al+\bt}{2}} r^{-2\al'} \dd r \\
 & \lesssim (1\vee L)^{p_0} s^{-2\al'} \int_s^t (t-r)^{-\frac{\al+\bt}{2}} \dd r.
\end{align*}
\underline{Term $I_6$}:
\begin{align*}
 I_6 \stackrel{\eqref{eq:Heat_Smooth}}{\lesssim} \int_s^t (t-r)^{-\frac{\al+\bt}{2}} \|\eps^{\frac{1}{2}}\<1>(r)\|_{\CC^{-\al}} \dd r
 \lesssim (1\vee L)^{p_0} \int_s^t (t-r)^{-\frac{\al+\bt}{2}} \dd r.
\end{align*}
\underline{Term $I_7$}:
\begin{align*}
 I_7 & \stackrel{\eqref{eq:Besov_Emb},\eqref{eq:Heat_Smooth}}{\lesssim} \int_s^t (t-r)^{-\frac{\bt+ \frac{2}{p}}{2}} \|v(r)\|_{L^p} \dd r
 \stackrel{\eqref{eq:Lp_est}}{\lesssim} (1\vee L)^{p_0} \int_s^t (t-r)^{-\frac{\bt+ \frac{2}{p}}{2}} r^{-\frac{1}{2}} \dd r \\
 & \lesssim (1\vee L)^{p_0} s^{-\frac{1}{2}} \int_s^t (t-r)^{-\frac{\bt+ \frac{2}{p}}{2}} \dd r.
\end{align*}
Using Proposition \ref{prop:besov_grad_est}, \eqref{eq:Lp_est} and \eqref{eq:grad_est} we notice that 
\begin{equation*}
 \left(\int_s^t \|v(r)\|_{\BB^1_{2,\infty}}^2 \dd r\right)^{\frac{1}{2}} \lesssim 
 \left(\int_s^t \|\nabla v(r)\|_{L^2}^2 \dd r\right)^{\frac{1}{2}} + 
 \left(\int_s^t \|v(r)\|_{L^2}^2 \dd r\right)^{\frac{1}{2}} 
 \lesssim (1\vee L)^{p_0} s^{-\frac{1}{2}}   .
\end{equation*}
Combining the above and choosing $s=t/2$ we find $\gamma' > 0$ such that 
\begin{equation*}
 t^{\gamma'}\|v(t)\|_{\CC^\bt} \lesssim (1\vee L)^{p_0} 
\end{equation*}
which completes the proof.
\end{proof}

\begin{proof}[Proof of Proposition \ref{prop:sep_bd}] We denote by $(1\vee L)^{p_0}$ a polynomial of degree $p_0\geq 1$ in the variable $1\vee L$ where the 
value of $p_0$ may change from line to line.

For $k\geq 0$ recall that $t_k = \nu + k$ and $s_k = t_k + \frac{1}{2}$. As before, we  write $Y(t) = X(t;y) - X(t;x)$. 

Let $t\in (t_k,s_k]$, $k\geq 1$. We restart the stochastic terms at time $s_{k-1}$ and write $Y(t) = v_{s_{k-1}}(t;\tilde y) - v_{s_{k-1}}(t;\tilde x)$ where 
for simplicity $\tilde y = X(s_{k-1};y)$ and $\tilde x = X(s_{k-1};x)$. Together with \eqref{eq:rem_eq_restart}, this implies that
\begin{align*}
 & (\partial_t-\Delta) Y 
 = -\left(v_{s_{k-1}}(\cdot;\tilde y)^3 - v_{s_{k-1}}(\cdot;\tilde x)^3\right) + Y 
 - 3 (v_{s_{k-1}}(\cdot;\tilde y) +v_{s_{k-1}}(\cdot;\tilde x) )\eps^{\frac{1}{2}}\<1>_{s_{k-1}}Y - 3 \eps\<2>_{s_{k-1}} Y.
\end{align*}
Using the mild form of the above equation, now starting at $t_k = s_{k-1}+\frac12$, we get
%
%
\begin{align*}
 \|Y(t)\|_{\CC^\bt} & \stackrel{\eqref{eq:Heat_Smooth},\eqref{eq:mult_ineq_1},\eqref{eq:mult_ineq_2}}{\lesssim} \|Y(t_k)\|_{\CC^\bt} + \int_{t_k}^t \|v_{s_{k-1}}(r;\tilde y)^3 - 
 v_{s_{k-1}}(r;\tilde x)^3\|_{\CC^\bt} \dd r \\ 
 & \quad + \int_{t_k}^t (t-r)^{-\frac{\al+\bt}{2}} \|v_{s_{k-1}}(r;\tilde y)^2 - v_{s_{k-1}}(r;\tilde x)^2\|_{\CC^\bt} \|\eps^{\frac{1}{2}}\<1>_{s_{k-1}}(r)\|_{\CC^{-\al}} \dd r \\
 & \quad + \int_{t_k}^t (t-r)^{-\frac{\al+\bt}{2}} \|Y(r)\|_{\CC^\bt} \|\eps\<2>_{s_{k-1}}(r)\|_{\CC^{-\al}} \dd r + \int_{t_k}^t \|Y(r)\|_{\CC^\bt} \dd r.
\end{align*}
By Lemma \ref{lem:rem_bnd} there exist $\gamma'>0$ such that
\begin{equation*}
 \sup_{x\in \CC^{-\al_0}}\sup_{t\in [s_{k-1}, s_k]} \left(t- s_{k-1}\right)^{\gamma'} 
 \|v_{s_{k-1}}(t;x)\|_{\CC^\bt} \lesssim \left(1\vee L_k\left(\nu+\frac{1}{2},\rho\right)\right)^{p_0}.
\end{equation*}
Combining the above we get
\begin{equation*}
 \|Y(t)\|_{\CC^\bt} \lesssim \|Y(t_k)\|_{\CC^\bt} + \left(1\vee L_k\left(\nu+\frac{1}{2},\rho\right)\right)^{p_0}
 \int_{t_k}^t (t- r)^{-\frac{\al+\bt}{2}} \|Y(r)\|_{\CC^\bt} \dd r.
\end{equation*}
By Lemma \ref{lem:grwl} there exists $c_0>0$ such that 
\begin{equation}\label{eq:gluing_1}
 \|Y(t)\|_{\CC^\bt} \lesssim \exp\left\{c_0 \left(1\vee L_k\left(\nu+\frac{1}{2},\rho\right)\right)^{p_0} (t - s) \right\} \|Y(t_k)\|_{\CC^\bt}. 
\end{equation}
Following the same strategy we prove that for $t\in [s_k,t_{k+1}]$, $k\geq 1$,
\begin{equation}\label{eq:gluing_2}
 \|Y(t)\|_{\CC^\bt} \lesssim \exp\left\{c_0 \left(1\vee L_{k+1}\left(\nu,\rho\right)\right)^{p_0} (t - s)\right\} \|Y(s_k)\|_{\CC^\bt}.
\end{equation}
Finally, we also need a bound for $t\in[t_0,t_1]$. To obtain an estimate which does not depend on any information before time $t_0$ we use local 
solution theory. By \cite[Theorem 3.3]{TW18} there exists $t_*\in(t_0,t_1)$ such that
\begin{equation*}
  \sup_{\|x\|_{\CC^{-\al_0}}\leq R}\sup_{r\in [t_0,t_*]} (r-t_0)^\gamma \|v_{t_0}(r;x)\|_{\CC^\bt} \leq 1 
\end{equation*}
and furthermore we can take
\begin{equation*}
 t_* = \left(\frac{1}{C(R\vee L_1(\nu,\rho))}\right)^{p_0}.
\end{equation*}
By Lemma \ref{lem:rem_bnd} we also have that 
\begin{equation*}
 \sup_{x\in\CC^{-\al_0}} \sup_{r\in \left(t_0,t_1\right]} (r-t_0)^{\gamma'}
 \|v_{t_0}(r;x)\|_{\CC^\bt} \lesssim (1\vee L_1(\nu,\rho))^{p_0}.
\end{equation*}
Combining these two bounds we get 
\begin{equation} \label{eq:v_local_bd}
 \sup_{\|x\|_{\CC^{-\al_0}}\leq R}\sup_{r\in[t_0,t_1]} (r-t_0)^\gamma \|v_{t_0}(r;x)\|_{\CC^\bt} \lesssim (1\vee L_1(\nu,\rho))^{p_0} 
\end{equation}
were the implicit constant depends on $R$. Note that $\gamma<\frac{1}{3}$ whereas $\gamma'$ is much larger. We write $Y(t) = v_{t_0}(t;y)-v_{t_0}(t;x)$ and use the mild form starting at $t_0$. We then use 
\eqref{eq:v_local_bd} to bound $\|v_{t_0}(t;\cdot)\|_{\CC^\bt}$ on $[t_0,t_1]$ which implies the estimate
\begin{equation*}
 \|Y(t)\|_{\CC^\bt} \lesssim \|Y(t_0)\|_{\CC^\bt} + \left(1\vee L_1\left(\nu,\rho\right)\right)^{p_0}
 \int_{t_0}^t (t- r)^{-\frac{\al+\bt}{2}} (r-t_0)^{-2\gamma} \|Y(r)\|_{\CC^\bt} \dd r.
\end{equation*}
The extra term $(r-t_0)^{-2\gamma}$ in the last inequality appears because of the blow-up of $v_{t_0}(t;\cdot)$ and 
$\<n>_{t_0}(t)$ for $t$ close to $t_0$. By Lemma \ref{lem:grwl} we obtain that
\begin{equation} \label{eq:gluing_3}
 \|Y(t)\|_{\CC^\bt} \lesssim 
 \exp\left\{c_0 \left(1\vee L_1\left(\nu,\rho\right)\right)^{p_0} (t - s)\right\} \|Y(s)\|_{\CC^\bt}.
\end{equation}

For arbitrary $t\in [\nu,\rho]$ we glue together \eqref{eq:gluing_1}, \eqref{eq:gluing_2} and \eqref{eq:gluing_3} to get
\begin{equation*}
 \|Y(t)\|_{\CC^\bt} \lesssim \exp\left\{\frac{c_0}{2} \sum_{k=1}^{\lfloor t-\nu \rfloor} 
 \sum_{l=0,\frac{1}{2}} \left(1\vee L_k(\nu+l,\rho\right))^{p_0} + L_0 (t-\nu) \right\} 
 \|Y(\nu)\|_{\CC^\bt}
\end{equation*}
for some $L_0>0$ which collects the implicit constants in the inequalities. 
\end{proof}

\section{Random walk estimates} \label{s:rwe}

In this section we prove Proposition \ref{prop:rwe} based mainly on probabilistic arguments.
%
%
%
%
In Sections \ref{s:ex_time_est} and \ref{s:ent_time_est} we provide estimates on $\frac{\kappa}{2} \tau_i(x)$ and 
$L(\nu_i(x),\rho_i(x);\sigma_i(x)) + (2-\kappa) \sigma_i(x) + M_0$ from Definition \ref{def:rw}. In Section \ref{s:rwe_proof}
we use these estimates to prove Proposition \ref{prop:rwe}.

\subsection{Estimates on the exit times} \label{s:ex_time_est}

\begin{proposition} \label{prop:diag_ex_est} Let $\delta >0$ and $\tau_{\mathrm{tree}} = \inf\{t>0: (t\wedge 1)^{(n-1) \al'} \|\eps^{\frac{n}{2}}\<n>(t)\|_{\CC^{-\al}}\geq \delta^n\}$. 
Then there exist $a_0>0$ and $\eps_0\in(0,1)$ such that for every $\eps\leq \eps_0$
\begin{equation*}
 \PP\left(\tau_{\mathrm{tree}}\leq \ee^{3a_0/\eps}\right) \leq \ee^{-3a_0/\eps}.
\end{equation*}
\end{proposition}
 
\begin{proof} First notice that for $N\geq 1$ 
\begin{equation*}
 \PP(\tau_{\mathrm{tree}}\leq N) \leq \sum_{k=0}^{N-1} \PP(\tau_{\mathrm{tree}} \in (k,k+1)) 
 \leq \sum_{k=0}^{N-1} \PP\left(\sup_{t\in(k, k+1]} (t\wedge 1)^{(n-1)\al'} \|\eps^{\frac{n}{2}}\<n>(t)\|_{\CC^{-\al}} \geq \delta^n\right).
\end{equation*}
By Proposition \ref{prop:exp_mom} and the exponential Chebyshev inequality there exists $a_0>0$ such that for every $k\geq 0$
\begin{equation*}
 \PP\left(\sup_{t\in(k,k+1]} (t\wedge 1)^{(n-1)\al'} \|\eps^{\frac{n}{2}}\<n>(t)\|_{\CC^{-\al}} \geq \delta^n\right) \leq 
 \ee^{-6a_0/\eps}.
\end{equation*}
Hence 
\begin{equation*}
 \PP(\tau_{\mathrm{tree}}\leq N) \leq N \ee^{-6a_0/\eps}
\end{equation*}
and choosing $N = \ee^{3a_0/\eps}$ completes the proof.  
\end{proof}

\begin{proposition}\label{prop:v_ex_est} For $\delta_1>0$ sufficiently small there exist $\delta_0,\delta_2>0$ such that if
\begin{equation}\label{eq:tree_contr}
 \sup_{t\leq T}(t\wedge1)^{(n-1)\al'} \|\eps^{\frac{n}{2}}\<n>(t)\|_{\CC^{-\al}} < \delta_2^n
\end{equation}
then for every $\|x-(\pm 1)\|_{\CC^{-\al_0}} \leq \delta_0$
\begin{equation*}
 \sup_{t\leq T}(t\wedge 1)^\gm\|v(t;x) - (\pm 1)\|_{\CC^\bt} < \delta_1
\end{equation*}
and
\begin{equation*}
 \sup_{t\leq T} \|X(t;x)-(\pm1)\|_{\CC^{-\al_0}} \leq 2\delta_1. 
\end{equation*}
\end{proposition}

\begin{proof} Let $u(t) = v(t;x) - (\pm 1)$. A Taylor expansion of $-v^3+v$ around $\pm1$ implies that
\begin{equation}\label{eq:u}
 (\partial_t - (\Delta-2)) u = \mathtt{Error}(u) - \left(3 v^2 \eps^{\frac{1}{2}}\<1> + 3 v \eps \<2> + \eps^{\frac{3}{2}} \<3>\right) + 2 \eps^{\frac{1}{2}} \<1>
\end{equation}
where $\mathtt{Error}(u) = -u^3\pm3u^2$ and $\|\mathtt{Error}(u)\|_{\CC^\bt} \lesssim \|u\|_{\CC^\bt}^3+\|u\|_{\CC^\bt}^2$. Let $T>0$ and $\iota = 
\inf\{t>0: (t\wedge 1)^\gm \|u(t)\|_{\CC^\bt} \geq \delta_1\}$ for some $\delta_1>0$ which we fix below. Using the 
mild form of \eqref{eq:u} we get
\begin{align*}
 (t\wedge1)^\gamma\|u(t)\|_{\CC^\bt} & \stackrel{\eqref{eq:Heat_Smooth},\eqref{eq:mult_ineq_1},\eqref{eq:mult_ineq_2}}{\lesssim} \ee^{-2t} \|x-(\pm 1)\|_{\CC^{-\al_0}} + \int_0^t \ee^{-2(t-s)} \left(\|u(s)\|_{\CC^\bt}^3
 +\|u(s)\|_{\CC^\bt}^2\right) \dd s \\
 & \quad + \int_0^t \ee^{-2(t-s)} (t-s)^{-\frac{\al+\bt}{2}} \Big(\|v(s)\|_{\CC^\bt}^2 \|\eps^{\frac{1}{2}}\<1>(s)\|_{\CC^{-\al}}
 + \|v(s)\|_{\CC^\bt} \|\eps\<2>(s)\|_{\CC^{-\al}} \\
 & \quad + \|\eps^{\frac{3}{2}}\<3>(s)\|_{\CC^{-\al}} + \|\eps^{\frac{1}{2}}\<1>(s)\|_{\CC^{-\al}}\Big) \dd s.
\end{align*}
If we furthermore assume \eqref{eq:tree_contr} for $t\leq T\wedge \iota$ we obtain that
\begin{align*}
 & (t\wedge1)^\gamma \|u(t)\|_{\CC^\bt} \\
 & \quad \lesssim  \delta_0 \ee^{-2t} + \delta_1^3 \int_0^t \ee^{-2(t-s)} (s\wedge1)^{-3\gm}
 \dd s + \delta_1^2 \int_0^t \ee^{-2(t-s)} (s\wedge1)^{-2\gm} \dd s \\
 & \quad \quad 
 + \delta_2 \int_0^t \ee^{-2(t-s)} (t-s)^{-\frac{\al+\bt}{2}} 
 \left((s\wedge1)^{-2\gm} + (s\wedge1)^{-\gm} (s\wedge1)^{-\al'} + (s\wedge 1)^{-2\al'} +1\right) \dd s.
\end{align*}
Then Lemma \ref{lem:int_bd} implies the bound
\begin{equation*}
\sup_{t\leq T\wedge \iota} (t\wedge1)^\gamma \|u(t)\|_{\CC^\bt} \lesssim \delta_0 + \delta_1^3 + \delta_1^2 + \delta_2.
\end{equation*}
Choosing $\delta_0< \frac{\delta_1}{4C}$, $\delta_1<\frac{1}{4C}$ and $\delta_2<\frac{\delta_1}{4C}$ this implies that 
$\sup_{t\leq T\wedge \iota} (t\wedge1)^\gamma \|u(t)\|_{\CC^\bt} < \delta_1$ which in turn implies that $\iota \leq T$ and proves the first bound. 

To prove the second bound we notice that for every $t\leq T$
\begin{equation*}
 \|X(t;x)-(\pm 1)\|_{\CC^{-\al_0}} \leq \|u(t)\|_{\CC^{-\al_0}} + \|\<1>(t)\|_{\CC^{-\al_0}} \leq \|u(t)\|_{\CC^{-\al_0}} + \delta_2.
\end{equation*}
Hence it suffices to prove that $\sup_{t\leq T} \|u(t)\|_{\CC^{-\al_0}} \leq \delta_1$. Using again the mild form of \eqref{eq:u} we
get
\begin{align*}
 \|u(t)\|_{\CC^{-\al_0}} & \stackrel{\eqref{eq:Heat_Smooth},\eqref{eq:alpha_beta_ineq},\eqref{eq:mult_ineq_1},\eqref{eq:mult_ineq_2}}{\lesssim} \ee^{-2t} \|x-(\pm 1)\|_{\CC^{-\al_0}} + \int_0^t \ee^{-2(t-s)} \left(\|u(s)\|_{\CC^\bt}^3
 +\|u(s)\|_{\CC^\bt}^2\right) \dd s \\
 & \quad + \int_0^t \ee^{-2(t-s)} \Big(\|v(s)\|_{\CC^\bt}^2 \|\eps^{\frac{1}{2}}\<1>(s)\|_{\CC^{-\al}}
 + \|v(s)\|_{\CC^\bt} \|\eps\<2>(s)\|_{\CC^{-\al}} \\
 & \quad + \|\eps^{\frac{3}{2}}\<3>(s)\|_{\CC^{-\al}} + \|\eps^{\frac{1}{2}}\<1>(s)\|_{\CC^{-\al}}\Big) \dd s
\end{align*}
for every $t\leq T$. Plugging in \eqref{eq:tree_contr} and the bound $\sup_{t\leq T} (t\wedge 1)^\gamma \|u(t)\|_{\CC^\bt}\leq \delta_1$ the last 
inequality implies
\begin{align*}
 \|u(t)\|_{\CC^\bt} & \lesssim \delta_0 \ee^{-2t} + \delta_1^3 \int_0^t \ee^{-2(t-s)} (s\wedge1)^{-3\gm}
 \dd s + \delta_1^2 \int_0^t \ee^{-2(t-s)} (s\wedge1)^{-2\gm} \dd s \\
 & \quad + \delta_2 \int_0^t \ee^{-2(t-s)}  
 \left((s\wedge1)^{-2\gm} + (s\wedge1)^{-\gm} (s\wedge1)^{-\al'} + (s\wedge 1)^{-2\al'} +1\right) \dd s.
\end{align*}
Using again Lemma \ref{lem:int_bd} we obtain that $\sup_{t\leq T}\|u(t)\|_{\CC^{-\al_0}} < \delta_1$, which completes the proof.
\end{proof}

%

\begin{proposition}\label{prop:f_i} For every $\kappa>0$ and $\delta_1>0$ sufficiently small there exist $a_0,\delta_0,\delta_2>0$
and $\eps_0\in(0,1)$ such that for every $\eps\leq \eps_0$
\begin{equation*}
\sup_{\|x-(\pm 1)\|_{\CC^{-\al_0}} \leq \delta_0} \PP\left(\frac{\kappa}{2}\tau_1(x) \leq \ee^{2a_0/\eps}\right)\leq \ee^{-3a_0/\eps},
\end{equation*}
where $\tau_1(x)$ is given by \eqref{eq:tau_sigma}. 
\end{proposition}

\begin{proof} We first notice that there exists $\eps_0>0$ such that for every $\eps\leq \eps_0$
\begin{equation*}
 \PP\left(\frac{\kappa}{2}\tau_1(x) \leq \ee^{2a_0/\eps}\right) \leq \PP\left(\tau_1(x) \leq \ee^{3a_0/\eps}\right).
\end{equation*}
The last probability can be estimated by Propositions \ref{prop:v_ex_est} and \ref{prop:diag_ex_est} for $\delta=\delta_2$.
\end{proof}

\subsection{Estimates on the entry times} \label{s:ent_time_est}


In this section we use large deviation theory and in particular a lower bound of the form
\begin{align}
 & \liminf_{\eps\searrow 0}\log\eps\inf_{x\in \aleph}\PP(X(\cdot;x)\in \mathcal{A}(T;x)) \label{eq:LDP} \\ 
 & \quad \geq -\sup_{x\in \aleph} \inf_{\substack{f\in \mathcal{A}(T;x) \\ f(0) = x}} 
 \left\{\underbrace{\frac{1}{4}\int_0^T \|(\partial_t - \Delta)f(t) + f(t)^3 - f(t)\|_{L^2}^2 \dd t}_{=:I(f)}\right\} \nonumber
\end{align}
where $\aleph$ is a compact subset of $\CC^{-\al}$ and $\mathcal{A}(T;x)\subset \{f:(0,T) \to \CC^{-\al}\}$ is open. This bound is an 
immediate consequence of \cite{HW15}  and  the remark that the solution map 
\begin{equation*}
 \CC^{-\al_0} \times \left(\CC^{-\al}\right)^3 \ni\left(x,\left\{\eps^{\frac{n}{2}}\<n>\right\}_{n\leq 3}\right) \mapsto X(\cdot;x) \in \CC^{-\al}
\end{equation*}
is jointly continuous on compact time intervals. This estimate implies a ``nice'' lower bound for the probabilities 
$\PP(X(\cdot;x)\in \mathcal{A}(T;x))$ if a suitable path $f\in \mathcal{A}(T;x)$ is chosen. 

In the next proposition we use the lower bound \eqref{eq:LDP} for suitable sets $\aleph$ and $\mathcal{A}(T;x)$ to estimate probabilities of the entry time of $X$ in 
a neighbourhood of $\pm1$. We construct a path $f(\cdot;x)$ and obtain bounds on $I(f(\cdot;x))$ uniformly in $x\in \aleph$.
    
\begin{proposition} \label{prop:1st_entry} Let $\delta_0>0$ and $\sigma(x) = \inf\left\{t > 0 : \min_{x_*\in\{-1,1\}} \|X(t;x) - x_*\|_{\CC^{-\al_0}} \leq \delta_0 \right\}$. 
For every $R,b>0$ there exists $T_0>0$ such that
 \begin{equation*}
  \sup_{\|x\|_{\CC^{-\al_0}}\leq R} \PP(\sigma(x) \geq T_0) \leq 1 - \ee^{-b/\eps}.
 \end{equation*}
\end{proposition}

\begin{proof} First notice that 
\begin{equation*}
 \PP(\sigma(x) \leq T_0)
 = \PP(\underbrace{\|X(T_*;x) - (\pm 1)\|_{\CC^{-\al_0}} < \delta_0 \text{ for some }
 T_*\leq T_0}_{=:\mathcal{A}(T_0;x)}).
\end{equation*}
By the large deviation estimate \eqref{eq:LDP} it suffices to bound  
\begin{equation*}
 \sup_{\|x\|_{\CC^{-\al_0}}\leq R} \inf_{\substack{f \in \mathcal{A}(T_0;x) \\ f(0) = x}} I(f(\cdot;x)).  
\end{equation*}
We construct a suitable path $g\in \mathcal{A}(T_0;x)$ and we use the trivial inequality 
\begin{equation*}
 \sup_{\|x\|_{\CC^{-\al_0}}\leq R} \inf_{\substack{f \in \mathcal{A}(T_0;x) \\ f(0) = x}} I(f(\cdot;x)) \leq 
 \sup_{\|x\|_{\CC^{-\al_0}}\leq R} I(g(\cdot;x)).
\end{equation*}
We now give the construction of $g$ which involves 5 different steps. In Steps $1$, $3$ and $5$, 
$g$ follows the deterministic flow. The contribution of these steps to the energy functional $I$ is zero.  
On Steps $2$ and $3$, $g$ is constructed by linear interpolation. The contribution of these steps is estimated
by Lemma \ref{lem:FJL}. Below we write $X_{det}(\cdot;x)$ to denote the solution of $\eqref{eq:det_AC}$ with
initial condition $x$. We also pass through the space $\BB^1_{2,2}$ to use convergence results for 
$X_{det}(\cdot;x)$ which hold in this topology (see Propositions \ref{prop:stat_conv} and \ref{prop:unst}).

\underline{Step 1 (Smoothness of initial condition via the deterministic flow)}: 

Let $\tau_1=1$. For $t\in[0,\tau_1]$ we set $g(t;x) = X_{det}(t;x)$. By Proposition 
\ref{prop:Linfty_2+lambda_reg} there exist $C\equiv C(r)>0$ and $\lambda>0$ such that
\begin{equation*}
 \sup_{\|x\|_{\CC^{-\al_0}}\leq R} \|X_{det}(1;x)\|_{\CC^{2+\lambda}} \leq C.
\end{equation*}

\smallskip

\underline{Step 2 (Reach points that lead to a stationary solution)}: 

By Step 1 $g(\tau_1;x)\in B_{\CC^{2+\lambda}}(0;C)$ uniformly for $\|x\|_{\CC^{-\al_0}}\leq R$. Let $\delta>0$ to be fixed below.
By compactness there exists $\{y_i\}_{1\leq i\leq N}$ such that  $B_{\CC^{2+\lambda}}(0;C)$ is covered by $\cup_{1\leq i\leq N} B_{\BB^1_{2,2}}(y_i;\delta)$.
Here we use that $\CC^{2+\lambda}$ is compactly embedded in $\BB^1_{2,2}$ (see Proposition \ref{prop:comp_emb}). 

Without loss of generality we assume that $\{y_i\}_{1\leq i\leq N}$ is such that $y_i\in \CC^\infty$ and $X_{det}(t;y_i)$ 
converges to a stationary solution $-1,0,1$ in $\BB^1_{2,2}$. Otherwise we choose $\{y_i^*\}_{1\leq i\leq N} \in B_{\BB^1_{2,2}}(y_i;\delta)$
such that $y_i^* \in \CC^\infty$ and relabel them. This is possible because of Proposition \ref{prop:stat_conv}.

Let $\tau_2=\tau_1+\tau$, for $\tau>0$ which we fix below. For $t\in[\tau_1,\tau_2]$ we set $g(t;x) = g(\tau_1;x) + \frac{t- \tau_1}{\tau_2-\tau_1} 
(y_i - g(\tau_1;x))$, where $y_i$ is such that $g(\tau_1;x) \in B_{\BB^1_{2,2}}(y_i;\delta)$.

\smallskip

\underline{Step 3 (Follow the deterministic flow to reach a stationary solution)}: 

Let $T_i^*$ be such that $X_{det}(t;y_i) \in B_{\BB^1_{2,2}}(x_*;\delta)$ for every $t\geq T_i^*$, where $x_*\in \{-1,0,1\}$ is the 
limit of $X_{det}(t;y_i)$ in $\BB^1_{2,2}$, for $\{y_i\}_{1\leq i\leq N}$ as in Step 2. Let $\tau_3=\tau_2+ \max_{1\leq i\leq N}T_i^*\vee 1$. For 
$t\in [\tau_2,\tau_3]$ we set $g(t;x) = X_{det}(t- \tau_2;y_i)$. If $X_{det}(\tau_3-\tau_2;y_i) \in B_{\BB^1_{2,2}}(\pm 1;\delta)$ we stop here. Otherwise 
$X_{det}(\tau_3-\tau_2;y_i)\in B_{\BB^1_{2,2}}(0;\delta)\cap B_{\CC^{2+\lambda}}(0;C)$ (here we use again Proposition \ref{prop:Linfty_2+lambda_reg} to ensure that
$X_{det}(\tau_3-\tau_2;y_i)\in B_{\CC^{2+\lambda}}(0;C)$) and we proceed to Steps 4 and 5. 

\smallskip

\underline{Step 4 (If an unstable solution is reached move to a point nearby which leads to a stable solution)}: 

We choose $y_0\in B_{\BB^1_{2,2}}(0;\delta)$ such that $y_0\in \CC^\infty$ and $X_{det}(t;y_0)$ converges to either $1$ or $-1$ 
in $\BB^1_{2,2}$. This is possible because of Proposition \ref{prop:unst}.

Let $\tau_4 = \tau_3 +\tau$ for $\tau>0$ as in Step 2 which we fix below. For $t\in [\tau_3,\tau_4]$ we set $g(t;x) = g(\tau_3;x) + 
\frac{t - \tau_3}{\tau_4 - \tau_3} (y_0 - g(\tau_3;x))$. 

\smallskip

\underline{Step 5 (Follow the deterministic flow again to finally reach a stable solution)}: 

Let $T_0^*$ be such that $X_{det}(t;y_0)\in B_{\BB^1_{2,2}}(\pm1;\delta)$ for every $t\geq T_0^*$, where $y_0$ is as in Step 4. Let $\tau_5 = \tau_4+T_0^*\vee 1$. 
For $t\in[\tau_4,\tau_5]$ we set $g(t;x) = X_{det}(t-\tau_4;y_0)$. 

\bigskip

For the path $g(\cdot;x)$ constructed above we see that after time $t \geq \tau_5$, $g(t;x)\in B_{\BB^1_{2,2}}(\pm 1;\delta)$ for every 
$\|x\|-{\CC^{-\al_0}}\leq R$. This implies that $\|g(t;x) - (\pm1)\|_{\CC^{-\al_0}} <C\delta$ since by  \eqref{eq:Besov_Emb}, $\BB^1_{2,2} \subset \CC^{-\al_0}$.
We now choose $\delta>0$ such that $C \delta< \delta_0$ and let $T_0=\tau_5+1$. Then $g\in \mathcal{A}(T_0;x)$. 

To bound $I(g(\cdot;x))$ we split our time interval based on the construction of $g$ i.e. $I_k=[\tau_{k-1},\tau_k]$ for $k=1,\ldots,4$ and $I_5=[\tau_5,T_0]$.  We 
first notice that for $k=1,3,5$ 
\begin{equation*}
 \frac{1}{4} \int_{I_k} \|(\partial_t-\Delta)g(t;x) +g(t;x)^3 - g(t;x)\|_{L^2}^2 \dd t = 0
\end{equation*}
since on these intervals we follow the deterministic flow. For the remaining two intervals, i.e. $k=2,4$, we first notice that by construction 
$\|g(\tau_{k-1};x)\|_{\CC^{2+\lambda}},\|g(\tau_k;x)\|_{\CC^{2+\lambda}}\leq C$.
By \eqref{eq:q1_q2_ineq}, $\CC^{2+\lambda} \subset \BB^2_{\infty,2}$ for every $\lambda >0$, hence we also have that $\|g(\tau_{k-1};x)\|_{\BB^2_{\infty,2}},
\|g(\tau_k;x)\|_{\BB^2_{\infty,2}}\leq C$ . We can now choose $\tau$ in Steps 2 and 4 according to Lemma \ref{lem:FJL}, which implies that  
\begin{equation*}
 \frac{1}{4}\int_{I_k} \|(\partial_t - \Delta) g(t;x) + g(t;x)^3 - g(t;x)\|_{L^2}^2 \dd t \leq C \delta.
\end{equation*}
Hence
\begin{equation*}
 \sup_{\|x\|_{\CC^{-\al_0}}\leq R}
 \frac{1}{4}\int_0^{T_0} \|(\partial_t - \Delta) g(t;x) + g(t;x)^3 - g(t;x))\|_{L^2}^2 \dd t \leq C\delta.
\end{equation*}
For $b>0$ we choose $\delta$ even smaller to ensure that $C\delta < b$. Finally, by \eqref{eq:LDP} there exists $\eps_0\in(0,1)$ such that for every $\eps\leq \eps_0$
\begin{equation*}
 \inf_{\|x\|_{\CC^{-\al_0}} \leq R}\PP(\sigma(x) \leq T_0) \geq \ee^{-b/\eps}
\end{equation*}
which completes the proof.
\end{proof}

\begin{lemma}[{\cite[Lemma 9.2]{FJL82}}] \label{lem:FJL} Let $f(t) = x + \frac{t}{\tau} (y-x)$ such that 
$\|x\|_{\BB^2_{2,2}}, \|y\|_{\BB^2_{2,2}} \leq R$ and $\|x-y\|_{L^2} \leq \delta$. There exist $\tau>0$ and 
$C\equiv C(R)$ such that
\begin{equation*}
 \frac{1}{4}\int_0^\tau \|(\partial_t - \Delta) f(t) + f(t)^3 - f(t)\|_{L^2}^2 \dd t \leq C \delta.
\end{equation*}
\end{lemma}

\begin{proof} We first notice that $\partial_t f(t) = \frac{1}{\tau} (y-x)$, hence $\|\partial_t f(t)\|_{L^2} \leq \frac{1}{\tau} \delta$.
For the term $\Delta f(t)$ we have
\begin{equation*}
 \|\Delta f(t)\|_{L^2} \leq \|\Delta x\|_{L^2} +\|\Delta y\|_{L^2} \lesssim \|x\|_{\BB^2_{2,2}} + \|y\|_{\BB^2_{2,2}} \lesssim R,
\end{equation*}
where we use that the Besov space $\BB^2_{2,2}$ is equivalent with the Sobolev space $H^1$. This is immediate from Definition \ref{def:besov}
for $p=q=2$ if we write $\|f*\eta_k\|_{L^2}$ using Plancherel's identity. For the term $f(t)^3 - f(t)$ we have
\begin{align*}
 \|f(t)^3 - f(t)\|_{L^2} & \lesssim \|f(t)\|_{L^6}^3 + \|f(t)\|_{L^2} 
 \stackrel{\eqref{eq:Lp_besov_control}}{\lesssim} \|f(t)\|_{\BB^0_{6,1}}^3 + \|f(t)\|_{\BB^0_{2,1}} \\
 & \stackrel{\eqref{eq:Besov_Emb},\lambda>0}{\lesssim} \|f(t)\|_{\BB^{\frac{2}{3}+\lambda}_{2,2}}^3 + \|f(t)\|_{\BB^\lambda_{2,2}}
 \stackrel{\eqref{eq:alpha_beta_ineq},\lambda < \frac{1}{3}}{\lesssim} \|f(t)\|_{\BB^2_{2,2}}^3 + \|f(t)\|_{\BB^2_{2,2}}.
\end{align*}
Hence for $C\equiv C(R)$
\begin{align*}
 \frac{1}{2}\int_0^\tau \|(\partial_t - \Delta) f(t) + f(t)^3 - f(t)\|_{L^2}^2 \dd t \leq \frac{1}{\tau} \delta^2 + C \tau.
\end{align*}
Choosing $\tau = \delta$ completes the proof.  
\end{proof}

In the next proposition we estimate the tails of the entry time of $X$ in a neighbourhood of $\pm1$ uniformly in 
the initial condition $x$. This is achieved by Proposition \ref{prop:1st_entry} and the Markov property combined with \cite[Corollary 3.10]{TW18} 
which implies that after time $t=1$ the process $X(\cdot;x)$ enters a compact subset of the state space with positive probability uniformly in $x$.  

\begin{proposition} \label{prop:ent_est} Let $\delta_0>0$ and $\sigma(x) = \inf\left\{t > 0 : \min_{x_*\in\{-1,1\}} \|X(t;x) - x_*\|_{\CC^{-\al_0}} \leq \delta_0 \right\}$.
For every $b>0$ there exist $T_0>0$ and $\eps_0\in(0,1)$ such that for every $\eps\leq \eps_0$ 
\begin{equation*}
 \sup_{x\in \CC^{-{\al_0}}} 
 \PP(\sigma(x) \geq mT_0) \leq \left(1-\ee^{-b/\eps}\right)^m
\end{equation*}
for every $m\geq 1$. 
\end{proposition}

\begin{proof} By \cite[Corollary 3.10]{TW18} and a simple application of Markov's 
inequality there exist $R_0>0$ such that
\begin{equation} \label{eq:compact_entry}
 \sup_{x\in \CC^{-\al_0}} \sup_{\eps\in(0,1]} \PP(\|X(1;x)\|_{\CC^{-\al}} > R_0) \leq \frac{1}{2}.
\end{equation}
By Proposition \ref{prop:1st_entry} for every $b>0$ there exists $T_0>0$ and $\eps_0\in(0,1)$ such that for every $\eps\leq \eps_0$
\begin{equation} \label{eq:entry_est}
 \sup_{\|x\|_{\CC^{-\al_0}}\leq R_0} \PP(\sigma(x) \geq T_0) \leq 1-\ee^{-b/\eps}.
\end{equation}
Then for every $x\in \CC^{-\al_0}$ and $\eps\leq \eps_0$
\begin{align}
 \PP(\sigma(x) \geq T_0+1) & \leq 
 \EE\left(\mathbf{1}_{\{\|X(1;x)\|_{\CC^{-\al_0}}\leq R_0\}} \PP(\sigma(X(1;x)) \geq T_0)\right) 
 + \PP(\|X(1;x)\|_{\CC^{-\al_0}} > R_0) \label{eq:m=1_bd} \\
 & \stackrel{\eqref{eq:compact_entry},\eqref{eq:entry_est}}{\leq} 1 - \frac{1}{2} \ee^{-b/\eps} \nonumber 
\end{align}
Using the Markov property successively implies for every $m\geq 1$ and $x\in \CC^{-\al_0}$
\begin{equation} \label{eq:MP_success}
 \PP(\sigma(x) \geq m (T_0+1)) \leq \sup_{y\in \CC^{-\al_0}}\PP(\sigma(y) \geq (T_0+1)) 
 \, \PP(\sigma(x) \geq (m-1) (T_0+1)).
\end{equation}
Combining \eqref{eq:m=1_bd} and \eqref{eq:MP_success} we obtain that
\begin{equation*}
 \sup_{x\in \CC^{-\al_0}} \PP(\sigma(x) \geq m (T_0+1)) \leq \left(1 - \frac{1}{2} \ee^{-b/\eps}\right)^m.
\end{equation*}
The last inequality completes the proof if we relabel $b$ and $T_0$.
\end{proof}

\begin{proposition}\label{prop:g_i} Let $\delta_0>0$, $\nu_1(x)$, $\rho_1(x)$ as in Definition \ref{def:stop_times}, $\sigma_1(x)$ as in 
\eqref{eq:tau_sigma} and $L(\nu_1(x),\rho_1(x);\sigma_1(x))$ as in \eqref{eq:L_form}. For every $\kappa, M_0,b>0$ there exist $T_0 > 0$
and $\eps_0\in(0,1)$ such that for every $\eps\leq \eps_0$
\begin{equation*}
\sup_{\|x-(\pm1)\|_{\CC^{-\al_0}}\leq \delta_0} \PP\left(\left[L(\nu_1(x),\rho_1(x);\sigma_1(x)) + (2-\kappa) \sigma_1(x) + M_0\right]^\frac{1}{p_0} \geq m T_0\right) 
\leq \left(1-\ee^{-b/\eps}\right)^m
\end{equation*}
for every $m\geq 1$ and $p_0\geq 1$ as in \eqref{eq:L_form}.
\end{proposition}

\begin{proof} We first condition on $\nu_1(x)$ to obtain the bound
\begin{align*}
 & \sup_{\|x-(\pm1)\|_{\CC^{-\al_0}}\leq \delta_0} \PP\left(\left[L(\nu_1(x),\rho_1(x);\sigma_1(x)) + (2-\kappa) \sigma_1(x) + M_0\right]^\frac{1}{p_0} 
 \geq m T_0\right) \\ 
 & \quad \leq \sup_{x\in\CC^{-\al_0}} 
 \underbrace{\PP\left(\left[L(0,\sigma(x);\sigma(x)) + (2-\kappa) \sigma(x) + M_0\right]^\frac{1}{p_0} 
 \geq m T_0\right)}_{=:\PP\left(g(\sigma(x))^{\frac{1}{p_0}} \geq m T_0\right)},
\end{align*}
where $\sigma(x) = \inf\left\{t > 0 : \min_{x_*\in\{-1,1\}} \|X(t;x) - x_*\|_{\CC^{-\al_0}} \leq \delta_0 \right\}$. Let $T_0 \geq 1$ to be fixed below
and notice that for any $T_1>0$
\begin{align*}
\PP\left(g(\sigma(x))^{\frac{1}{p_0}} \geq m T_0\right) & \leq \PP\left(g(\sigma(x))^{\frac{1}{p_0}} \geq m T_0, \, \sigma(x)\leq m T_1\right) 
+ \PP(\sigma(x) \geq m T_1) \\
& \leq \PP\left(\sum_{k=1}^{\lfloor m T_1 \rfloor} \sum_{l= 0,\frac{1}{2}} L_k(l,mT_1) \geq m (T_0 - C)\right)
+ \PP(\sigma(x) \geq m T_1) 
\end{align*}
for some $C>0$, where in the second inequality we use convexity of the mapping $g\mapsto g^{\frac{1}{p_0}}$ and the fact that 
$L_k(l,\sigma)$ is increasing in $\sigma$ by Definition \ref{def:L_k}. By Proposition \ref{prop:ent_est} we can choose $T_1>0$ and 
$\eps_0\in(0,1)$ such that for every $\eps\leq \eps_0$
\begin{equation*}
\sup_{x\in \CC^{-\al_0}} \PP(\sigma(x) \geq m T_1) \leq \left(1- \ee^{-b/\eps}\right)^m.
\end{equation*}
We also notice that 
\begin{align*}
 \PP\left(\sum_{k=1}^{\lfloor m T_1 \rfloor} \sum_{l= 0,\frac{1}{2}} L_k(l,mT_1) \geq m (T_0 - C)\right) 
 & \leq \sum_{l=0,\frac{1}{2}} \PP\left(\sum_{k=1}^{\lfloor m T_1 \rfloor}L_k(l,l+k) \geq m \left(\frac{T_0 - C}{2}\right)\right) \\
 & \leq \sum_{l=0,\frac{1}{2}}\exp\left\{-c m \left(\frac{T_0 - C}{2\eps}\right)\right\} 
 \left(\EE\ee^{cL_1(l,1)/\eps}\right)^{mT_1},
\end{align*}
where in the first inequality we use that $L_k(l,mT_1) \leq L_k(l,l+k)$, for every $1\leq k\leq \lfloor m T_1 \rfloor$, and in the second 
we use an exponential Chebyshev inequality, independence and equality in law of the $L_k(l,l+k)$'s. For any $T>0$ we choose $c\equiv c(n)>0$ 
according to Proposition \ref{prop:exp_mom}, $T_0$ sufficiently large and $\eps_0\in(0,1)$ sufficiently small such that for every 
$\eps\leq \eps_0$  
\begin{equation*}
 \sum_{l=0,\frac{1}{2}}\exp\left\{-c m \left(\frac{T_0 - C}{2\eps}\right)\right\} 
 \left(\EE\ee^{cL_1(l,1)/\eps}\right)^{mT_1} \leq \ee^{-mT/\eps}.
\end{equation*}
Combining all the previous inequalities imply that 
\begin{equation*}
 \sup_{\|x-(\pm1)\|_{\CC^{-\al_0}}\leq \delta_0} \PP\left(\left[L(\nu_1(x),\rho_1(x);\sigma_1(x)) + (2-\kappa) \sigma_1(x) + M_0\right]^\frac{1}{p_0} \geq m T_0\right) 
 \leq \ee^{-m T/\eps} + \left(1 - \ee^{-b/\eps}\right)^m.
\end{equation*}
This completes the proof if we relabel $b$ since $T$ is arbitrary.
\end{proof}

\subsection{Proof of Proposition \ref{prop:rwe}} \label{s:rwe_proof}

In this section we set
\begin{align*}
& f_i(x) := \frac{\kappa}{2} \tau_i(x). \\
& g_i(x) := L(\nu_i(x),\rho_i(x);\sigma_i(x)) + (2-\kappa) \sigma_i(x) + M_0.
\end{align*}
In this notation the random walk $S_N(x)$ in Definition \ref{def:rw} is given by $\sum_{i\leq N} (f_i(x) - g_i(x))$. 

To prove Proposition \ref{prop:rwe} we first consider a sequence of i.i.d. random variables 
$\{\tilde f_i\}_{i\geq 1}$ such that $\tilde f_1 \sim \exp(1)$. We furthermore assume that the family 
$\{\tilde f_i\}_{i\geq 1}$ is independent from both $\{f_i(x)\}_{i\geq 1}$ and $\{g_i(x)\}_{i\geq 1}$. For $\lambda >0$ which we 
fix later on, we set 
\begin{equation*} 
 \tilde S_N(x) := \lambda \sum_{i\leq N} \tilde f_i - \sum_{i\leq N} g_i(x).
\end{equation*}
In the proof of Proposition \ref{prop:rwe} below we compare the random walk $S_N(x)$ with $\tilde S_N(x)$. The idea is 
that $\sum_{i\leq N} f_i(x)$ behaves like $\lambda\sum_{i\leq N} \tilde f_i$ for suitable $\lambda>0$. 

In the next proposition we estimate the new random walk $\tilde S_N(x)$ using stochastic dominance. In particular we 
assume that the family of random variables $\{g_i(x)\}_{i\geq 1}$ is stochastically dominated by a family of i.i.d. random variables
$\{\tilde g_i\}_{i\geq 1}$ which does not depend on $x$ and obtain a lower bound on $\PP(-\tilde S_N(x) \leq u \text{ for every } 
N\geq 1)$.

From now on we denote by $\mu_Z$ the law of a random variable $Z$. 

\begin{proposition}\label{prop:stoch_dom} Assume that there exists a family of i.i.d. random variables $\{\tilde g_i\}_{i\geq 1}$, independent
from both $\{g_i(x)\}_{i\geq 1}$ and $\{\tilde f_i\}_{i \geq1}$, such that

\begin{equation*}
 \sup_{\|x-(\pm1)\|_{\CC^{-\al_0}}\leq \delta_0}\PP(g_i(x) \geq g) \leq \PP(\tilde g_i \geq g)
\end{equation*}
for every $g\geq 0$. Let $\tilde S_N = \lambda \sum_{i\leq N} \tilde f_i - \sum_{i\leq N} \tilde g_i$. Then
\begin{equation*}
 \inf_{\|x-(\pm1)\|_{\CC^{-\al_0}}\leq \delta_0}\PP(-\tilde S_N(x) \leq u \text{ for every } N\geq 1) 
 \geq \PP(-\tilde S_N \leq u \text{ for every } N \geq 1).
\end{equation*}
\end{proposition}

\begin{proof} Let 
\begin{align*}
 & G_N(x,u) = \PP(-\tilde S_M(x) \leq u \text{ for every } N \geq M \geq 1). \\
 & G_N(u) = \PP(-\tilde S_M \leq u \text{ for every } N \geq M \geq 1).
\end{align*}
We first prove that for every $N\geq 1$ and every $x$
\begin{equation}\label{eq:G_bd}
 G_N(x,u) \geq G_N(u).
\end{equation}
For $N = 1$ we have that 
\begin{align*}
 G_1(x,u) & = \PP(-\lambda\tilde f_1 + g_1(x) \leq u) 
 = \int_0^\infty \PP(g_1(x) \leq u+ \lambda f) \, \mu_{\tilde f_1}(\dd f) \\
 & \geq \int_0^\infty \PP(\tilde g_1 \leq u+\lambda f) \, \mu_{\tilde f_1}(\dd f) 
 = \PP(-\lambda\tilde f_1 + \tilde g_1 \leq u) = G_1(u).
\end{align*}
Let us assume that \eqref{eq:G_bd} holds for $N$. Let $\partial B_0 = \{y: \|y- (\pm 1)\|_{\CC^{-\al_0}} =\delta_0\}$. Conditioning 
on $\left(\tilde f_1, g_1(x), X(\nu_2(x);x)\right)$ and using independence of $\tilde f_1$ from the joint law of $\left(g_1(x), X(\nu_2(x);x)\right)$
we notice that
\begin{align}
 G_{N+1}(x,u) & = \int_0^\infty \int_{[0,u+\lambda f]\times \partial B_0} G_N(y,u+\lambda f-g) 
 \, \mu_{(g_1(x),X(\nu_2(x);x))}(\dd g, \dd y) \, \mu_{\tilde f_1}(\dd f) \label{eq:GN+1_1st_step}\\
 & \stackrel{\eqref{eq:G_bd}}{\geq} \int_0^\infty \int_{[0,u+\lambda f]\times \partial B_0} G_N(u+\lambda f-g) 
 \, \mu_{(g_1(x),X(\nu_2(x);x))}(\dd g, \dd y) \, \mu_{\tilde f_1}(\dd f) \nonumber \\
 & = \int_0^\infty \int_{[0,u+\lambda f]} G_N(u+\lambda f-g) 
 \, \mu_{g_1(x)}(\dd g) \, \mu_{\tilde f_1}(\dd f). \nonumber
\end{align}
In the last equality above we use that $G_N(u+\lambda f-g)$ does not depend on $y$, hence we can drop the integral with respect to $y$. Let
\begin{equation*}
 H(g) = \mathbf{1}_{\{g\leq u+ \lambda f\}} G_N(u+\lambda f - g). 
\end{equation*}
Then for fixed $u,f\geq 0$, $H$ is decreasing with respect to $g$. By Lemma \ref{lem:decr_exp}
\begin{align*}
 \int_{[0,u+\lambda f]} G_N(u+\lambda f-g) 
 \, \mu_{g_1(x)}(\dd g) & = \int H(g) \mu_{g_1(x)}(\dd g) 
 \geq \int H(g) \mu_{\tilde g_1}(\dd g) \\
 & = \int_{[0,u+\lambda f]} G_N(u+\lambda f-g) 
 \, \mu_{\tilde g_1}(\dd g).
\end{align*}
Integrating the last inequality with respect to $f$ with $\mu_{\tilde f_1}$ and combining with \eqref{eq:GN+1_1st_step} we obtain
\begin{align*}
 G_{N+1}(x,u) \geq \int_0^\infty \int_{[0,u+\lambda f]} G_N(u+\lambda f-g) 
 \, \mu_{\tilde g_1}(\dd g) \, \mu_{\tilde f_1}(\dd f) = G_{N+1}(u)
\end{align*}
which proves \eqref{eq:G_bd}. If we now take $N\to \infty$ in \eqref{eq:G_bd} we
get for arbitrary $x$ 
\begin{equation*}
 G(x,u) \geq G(u) 
\end{equation*}
which completes the proof. 
\end{proof}

In the next proposition we prove existence of a family of random variables $\{\tilde g_i\}_{i\geq 1}$ that satisfy the 
assumption of Proposition \ref{prop:stoch_dom} and estimate their first moment. 

\begin{proposition}\label{prop:dom_constr} There exists a family of i.i.d. random variables $\{\tilde g_i\}_{i\geq 1}$, 
independent from both $\{g_i(x)\}_{i\geq 1}$ and $\{\tilde f_i\}_{i\geq 1}$, such that 
\begin{equation*}
 \sup_{\|x-(\pm1)\|_{\CC^{-\al_0}}\leq \delta_0}\PP(g_i(x) \geq g) \leq \PP(\tilde g_i \geq g),
\end{equation*}
and furthermore for every $b>0$ there exist $\eps_0\in(0,1)$ and $C>0$ such that for every $\eps\leq \eps_0$ 
\begin{equation*}
 \EE \tilde g_1 \leq C \ee^{b/\eps}.
\end{equation*}
\end{proposition}

\begin{proof} We first notice that by the Markov property 
\begin{equation*}
 \sup_{\|x-(\pm1)\|_{\CC^{-\al_0}} \leq \delta_0}\PP(g_i(x) \geq g) \leq 
 \sup_{\|x-(\pm1)\|_{\CC^{-\al_0}} \leq \delta_0 }\PP(g_1(x) \geq g).
\end{equation*}
Let $F(g)$ be the right continuous version of the increasing function $1 - \sup_{x\in\CC^{-\al_0}} \PP(g_1(x) \geq g)$. We consider a family of i.i.d. random variables such 
$\{\tilde g_i\}_{i\geq 1}$ independent from both $\{g_i(x)\}_{i\geq 1}$ and $\{\tilde f_i\}_{i\geq1}$ such that $\PP(\tilde g_i\leq g) = F(g)$. To estimate
$\EE \tilde g_1$ let $c_\eps>0$ to be fixed below. We notice that  
\begin{align}
 \EE \tilde g_1 \leq \sup_{g\geq 0} g \ee^{-c_\eps g^\frac{1}{p_0}} 
 \EE\exp\left\{c_\eps \tilde g_1^\frac{1}{p_0}\right\} \label{eq:1st_mom} \leq \left(\frac{p_0\ee^{-1}}{c_\eps}\right)^{p_0} 
 \EE\exp\left\{c_\eps \tilde g_1^\frac{1}{p_0}\right\}.
\end{align}
For $b>0$ we choose $T_0>0$ and $\eps_0\in(0,1)$ as in Proposition \ref{prop:g_i}. Then for every $\eps\leq \eps_0$
\begin{align*}
 \EE\exp\left\{c_\eps \tilde g_1^\frac{1}{p_0}\right\} &  = 1 + \int_0^\infty c_\eps\ee^{c_\eps g} 
 \PP\left(\tilde g_1^\frac{1}{p_0} \geq g\right) \dd g \leq 1+ \sum_{m\geq 0} \PP\left(\tilde g_1^\frac{1}{p_0} \geq mT_0\right) 
 \int_{mT_0}^{(m+1)T_0} c_\eps \ee^{c_\eps g} \dd g \\
 & = 1+ \sum_{m\geq 0} \sup_{\|x-(\pm 1)\|_{\CC^{-\al_0}}\leq \delta_0} \PP\left(g_1(x)^\frac{1}{p_0} \geq mT_0\right) 
 \int_{mT_0}^{(m+1)T_0} c_\eps \ee^{c_\eps g} \dd g \\
 & \leq 1 + \ee^{c_\eps T_0} \sum_{m\geq 0} \ee^{mc_\eps T_0} \left(1-\ee^{-b/\eps}\right)^m.   
\end{align*}
where in the last inequality we use Proposition \ref{prop:g_i} to estimate $\PP\left(g_1(x)^\frac{1}{p_0} \geq mT_0\right)$. We now choose 
$c_\eps>0$ such that $c_\eps T_0 = \log\left(1+\ee^{-b/\eps}\right)$. Then
\begin{align*}
 \EE\exp\left\{c_\eps \tilde g_1^\frac{1}{p_0}\right\} & \leq 1 + \left(1+\ee^{-b/\eps}\right) 
 \sum_{m\geq 0} \left(1+\ee^{-b/\eps}\right)^m \left(1-\ee^{-b/\eps}\right)^m 
 \leq 1+2 \sum_{m\geq 0} \left(1-\ee^{-2b/\eps}\right)^m \\
 & = 1+ 2 \ee^{2b/\eps}.
\end{align*}
Finally, by \eqref{eq:1st_mom} we obtain that
\begin{equation*}
 \EE \tilde g_1 \leq 
 \left(\frac{p_0\ee^{-1}T_0}{\log\left(1+\ee^{-b/\eps}\right)}\right)^{p_0}
 \left(1+ 2\ee^{2b/\eps}\right)
\end{equation*}
which completes the proof if we relabel $b$.
\end{proof}

\begin{remark} In the proof of Proposition \ref{prop:dom_constr} we use stretched exponential moments 
of $\tilde g_1$, although we only need 1st moments (see Lemma \ref{lem:CL_est} below). This simplifies our calculations. 
\end{remark}

From now on we let $\tilde S_N = \lambda \sum_{i\leq N} \tilde f_i - \sum_{i\leq N} \tilde g_i$ for $\{\tilde g_i\}_{i\geq 1}$ as in
Proposition \ref{prop:dom_constr}. 

In the next proposition we explicitly compute the probability $\PP(-\tilde S_N \leq 0 \text{ for every } N\geq 1)$. The proof is essentially
the same as the classical Cram\'er--Lundberg estimate (see \cite[Chapter 1.2]{EKM97}). We present it here for the reader's convenience.    

\begin{proposition} \label{prop:reneq} For the random walk $\tilde S_N$ the following estimate holds,
\begin{equation*}
 \PP(-\tilde S_N \leq 0 \text{ for every } N \geq 1) 
 = 1 - \frac{1}{\lambda} \EE\tilde g_1.
\end{equation*}
\end{proposition}

\begin{proof} Let $G(u) = \PP(-\tilde S_N \leq u \text{ for every } N \geq 1)$. Conditioning on $(\tilde f_1,\tilde g_1)$ and using independence 
we notice that
\begin{align}
 G(u) & = \PP\left( -\lambda \sum_{i=2}^N \tilde f_i + \sum_{i=2}^N \tilde g_i \leq 
 u + \lambda \tilde f_1 - \tilde g_1 \text{ for every } N\geq 2,
 \, -\lambda \tilde f_1 + \tilde g_1 \leq u\right) \label{eq:1st_cond}\\
 & = \int_0^\infty \int_0^{u+\lambda f} G(u+\lambda f-g) 
 \, \mu_{\tilde g_1}(\dd g) \, \mu_{\tilde f_1}(\dd f) \nonumber \\
 & = \frac{1}{\lambda} \ee^{u/\lambda}
 \int_u^\infty \ee^{-\bar f/\lambda} \int_0^{\bar f} G(\bar f-g) 
 \, \mu_{\tilde g_1}(\dd g) \dd \bar f \nonumber
\end{align}
where in the last equality we use that $\tilde f_1\sim\exp(1)$ and we also make the change of variables $\bar f = u+\lambda f$. This 
implies that $G(u)$ is differentiable with respect to $u$ and in particular
\begin{equation*} 
 \partial_{\bar u} G(\bar u) = \frac{1}{\lambda} G(\bar u) - \frac{1}{\lambda} 
 \int_0^{\bar u} G(\bar u-g) 
 \, \mu_{\tilde g_1}(\dd g).
\end{equation*}
Integrating the last equation form $0$ to $u$ we obtain that
\begin{equation} \label{eq:G_renewal_eq_0}
 G(u) = G(x,0) + \frac{1}{\lambda} \int_0^u G(u -\bar u) \dd \bar u - \frac{1}{\lambda}
 \int_0^u \int_0^{\bar u} G(\bar u-g) 
 \, \mu_{\tilde g_1}(\dd g) \dd \bar u. 
\end{equation}
Let $F(g) := \mu_{\tilde g_1}([0,g])$. A simple integration by parts implies
\begin{align}
 \int_0^u \int_0^{\bar u} G(\bar u-g) 
 \, \mu_{\tilde g_1}(\dd g) \dd \bar u &  = \int_0^u \left( \left[G(\bar u - g)\right]_{g=0}^{\bar u} + 
 \int_0^{\bar u} \partial_g G(\bar u - g) F(g) \dd g\right) \dd \bar u \label{eq:int_by_parts} \\
 & = \int_0^u G(0) F(\bar u) \dd \bar u 
 + \int_0^u \int_g^u \partial_g G(\bar u-g) \dd \bar u F(g) \dd g \nonumber \\
 & = \int_0^u G(0) F(\bar u) \dd \bar u - \int_0^u  
 \left[-G(\bar u -g)\right]_g^u F(g) \dd g \nonumber \\
 & = \int_0^u G(u-g) F(g) \dd g. \nonumber
\end{align}
Combining \eqref{eq:G_renewal_eq_0} and \eqref{eq:int_by_parts} we get
\begin{equation*} 
 G(u) = G(0) + \frac{1}{\lambda} \int_0^u G(u - \bar u) \dd \bar u
 - \frac{1}{\lambda} \int_0^u G(u - \bar u ) F(\bar u) \dd \bar u.
\end{equation*}
By taking $u\to \infty$ in the last equation and using the dominated convergence theorem and the law of large numbers we finally obtain 
\begin{equation*}
 1 = G(0) + \frac{1}{\lambda} \EE \tilde g_1
\end{equation*}
which completes the proof. 
\end{proof}

Combining Propositions \ref{prop:stoch_dom}, \ref{prop:dom_constr} and \ref{prop:reneq} we obtain the 
following lemma. 

\begin{lemma}\label{lem:CL_est} For any $b>0$ there exist $\eps_0\in(0,1)$ and $C>0$ such that for every $\eps\leq \eps_0$
\begin{equation*}
 \inf_{\|x-(\pm1)\|_{\CC^{-\al_0}}\leq \delta_0} \PP(-\tilde S_N(x) \leq 0 \text{ for every } N\geq 1) 
 \geq 1 - C \frac{\ee^{b/\eps}}{\lambda}.
\end{equation*}
\end{lemma}

\begin{proof} By Propositions \ref{prop:stoch_dom}, \ref{prop:dom_constr} and \ref{prop:reneq} and 
\begin{align*}
\inf_{\|x-(\pm1)\|_{\CC^{-\al_0}} \leq \delta_0}
\PP(-\tilde S_N(x) \leq 0 \text{ for every } N\geq 1) \geq \PP(-\tilde S_N \leq 0 \text{ for every } N\geq 1) 
= 1 - \frac{1}{\lambda} \EE \tilde g_1.
\end{align*}
Moreover, by Proposition \ref{prop:dom_constr} for every $b>0$ there exist $\eps_0\in(0,1)$ and $C>0$ such that 
for every $\eps\leq \eps_0$, $\EE \tilde g_1\leq C \ee^{b/\eps}$ which completes the proof.
\end{proof}

%
%

We are now ready to prove Proposition \ref{prop:rwe} which is the main goal of this section.

\begin{proof}[Proof of Proposition \ref{prop:rwe}.] We estimate $\PP(S_N(x) \leq 0 \text{ for some } N\geq 1)$ in the 
following way, 
\begin{align}
 \PP(-S_N(x) \geq 0 \text{ for some } N\geq 1) & \leq 
 \PP\left(-\sum_{i\leq N} f_i(x) + \lambda \sum_{i\leq N} \tilde f_i \geq 0 \text{ for some } N \geq 1\right) \label{eq:surv_prob} \\
 & \quad + \PP(-\tilde S_N(x) \geq 0 \text{ for some } N \geq 1). \nonumber
\end{align}
The second term on the right hand side can be estimated by Lemma \ref{lem:CL_est} which provides a bound of the form
\begin{equation} \label{eq:2nd_term}
 \sup_{\|x-(\pm1)\|_{\CC^{-\al_0}}\leq \delta_0} \PP(-\tilde S_N(x) \geq 0 \text{ for some } N \geq 1) 
 \leq C \frac{\ee^{b/\eps}}{\lambda}.
\end{equation}
For the first term we notice that
\begin{align*}
 & \PP\left(-\sum_{i\leq N} f_i(x) + \lambda \sum_{i\leq N} \tilde f_i \geq 0 \text{ for some } N \geq 1\right) \\
 & \quad \leq \sum_{N\geq 1} \PP\left(-\sum_{i\leq N} f_i(x) + \lambda \sum_{i\leq N} \tilde f_i \geq 0\right)  
 \leq \sum_{N\geq 1} \PP\left(\exp\left\{-\frac{1}{2\lambda}\sum_{i\leq N} f_i(x) 
 + \frac{1}{2} \sum_{i\leq N} \tilde f_i\right\}\geq 1\right).
\end{align*}
By Markov's inequality, independence of $\{f_i(x)\}_{i\geq1}$ and $\{\tilde f_i\}_{i\geq1}$ and equality
in law of the $\tilde f_i$'s the last inequality implies that
\begin{equation} \label{eq:IN_JN}
 \PP\left(-\sum_{i\leq N} f_i(x) + \lambda \sum_{i\leq N} \tilde f_i \geq 0 \text{ for some } N \geq 1\right)
 \leq \sum_{N\geq 1} \underbrace{\EE \exp\left\{-\frac{1}{2\lambda}\sum_{i\leq N}f_i(x)\right\}}_{=:I_N(x)} 
 \underbrace{\left(\EE \exp\left\{\frac{\tilde f_1}{2}\right\}\right)^N}_{\leq 2^N \text{ since } \tilde f_1\sim \exp(1)}.
\end{equation}
Let $\eps_0\in(0,1)$ as in Proposition \ref{prop:f_i}. For the term $I_N(x)$ we notice that for every $\eps\leq \eps_0$ 
\begin{align*} 
 I_N(x) & \leq 
 \left(\sup_{\|x-(\pm1)\|_{\CC^{-\al_0}}\leq \delta_0}\EE\exp\left\{-\frac{1}{2\lambda}f_1(x)\right\}\right)^N \\
 & \leq \left(\sup_{\|x-(\pm1)\|_{\CC^{-\al_0}}\leq \delta_0}\left[\EE\exp\left\{-\frac{1}{2\lambda}f_1(x)\right\} \mathbf{1}_{\{f_1(x) \geq \ee^{2a_0/\eps}\}}
 + \PP\left(f_1(x) \leq \ee^{2a_0/\eps}\right) \right] \right)^N \\
 & \leq \left(\ee^{-\ee^{2a_0/\eps}/2\lambda} + \ee^{-3a_0/\eps}\right)^N,
\end{align*}
where in the first inequality we use the Markov property and in the last we use Proposition \ref{prop:f_i}. If we choose 
$\frac{1}{2\lambda}=\ee^{-(2a_0-b)/\eps}$ and choose $\eps_0\in(0,1)$ even smaller the last inequality implies that for 
every $\eps\leq \eps_0$
\begin{equation*} 
 \sup_{\|x-(\pm1)\|_{\CC^{-\al_0}}\leq \delta_0} I_N(x) \leq \left(\ee^{-\ee^{b/\eps}} + \ee^{-3a_0/\eps}\right)^N
 \leq \ee^{-5a_0N/2\eps}.
\end{equation*}
Combining with \eqref{eq:IN_JN} we find $\eps_0\in(0,1)$ such that for every $\eps\leq \eps_0$
\begin{align}
 & \sup_{\|x-(\pm1)\|_{\CC^{-\al_0}}\leq \delta_0} 
 \PP\left(-\sum_{i\leq N} f_i(x) + \lambda \sum_{i\leq N} \tilde f_i \geq 0 \text{ for some } N \geq 1\right)
 \label{eq:1st_term} \\ 
 & \quad \leq
 \sum_{N\geq 1} \ee^{-5a_0N/2\eps} 2^N \leq \sum_{N\geq 1} \ee^{-2a_0N/\eps} 
 = \frac{\ee^{-2a_0/\eps}}{1-\ee^{-2a_0/\eps}}. \nonumber
\end{align}
Finally \eqref{eq:surv_prob}, \eqref{eq:2nd_term} and \eqref{eq:1st_term} imply that 
\begin{equation*}
 \sup_{\|x-(\pm1)\|_{\CC^{-\al_0}}\leq \delta_0}\PP(-S_N(x) \geq 0 \text{ for some } N\geq 1) \leq 
 C\frac{\ee^{b/\eps}}{\ee^{(2a_0-b)/\eps}} 
 + \frac{\ee^{-2a_0/\eps}}{1- \ee^{-2a_0/\eps}}
\end{equation*}
which completes the proof since $b$ is arbitrary.
\end{proof}

\section{Applications to Eyring--Kramers law} \label{s:eyr_kram_app}

In this section we consider the spatial Galerkin approximation $X_N(\cdot;x)$ of $X(\cdot;x)$ given by  
\begin{align}
 & (\partial_t - \Delta) X_N  = - \Pi_N\left(X_N^3 - X_N - 3 \eps \Re_N X_N \right) 
 + \sqrt{2\eps} \xi_N \label{eq:X_approx} \\
 & X_N\big|_{t=0} = x_N \nonumber
\end{align}
where $\Pi_N$ is the projection on $\{f \in L^2: f(z)=\sum_{|k|\leq N} \hat f(k) L^{-2}\ee^{2\pi\ii k\cdot z/L}\}$,
$\xi_N = \Pi_N \xi$, $x_N = \Pi_N x$ and $\Re_N$ is as in \eqref{eq:renorm_constant}. Here for $k\in \ZZ^2$ we set 
$|k|=|k_1|\vee|k_2|$. In this notation we have that $\Pi_N f = f*D_N$, where $D_N$ is the $2$-dimensional square Dirichlet
kernel given by $D_N(z) = \sum_{|k|\leq N} L^{-2}\ee^{2\pi\ii k\cdot z/L}$.  

To treat \eqref{eq:X_approx} we write $X_N(\cdot;x) = v_N(\cdot;x) + \eps^{\frac{1}{2}}\<1>_N(\cdot;x)$ for 
\begin{align*}
 & (\partial_t - (\Delta - 1)) \<1>_N =  \sqrt{2} \xi_N \\
 & \<1>_N(0) = 0.
\end{align*}
Then $v_N(\cdot;x)$ solves
\begin{align}
 & \left(\partial_t - \Delta\right) v_N = -\Pi_N(v^3_N) + v_N - \Pi_N\left(3 v^2_N \eps^{\frac{1}{2}} \<1>_N + 3 v_N \eps \<2>_N + \eps^{\frac{3}{2}} \<3>_N \right) 
 + 2 \eps^{\frac{1}{2}} \<1>_N \label{eq:v_approx} \\
 & v_N\big|_{t=0} = x_N \nonumber
\end{align}
where $\<2>_N = \<1>_N^2 - \Re_N$ and $\<3>_N = \<1>_N^3- 3 \Re_N \<1>_N$.

For $\delta\in(0,1/2)$ and $\al>0$ we define the symmetric subsets $A$ and $B$ of $\CC^{-\al}$ by  
\begin{align}
 & A := \left\{f\in \CC^{-\al} : \bar f \in [-1-\delta,-1+\delta], 
 \, f - \bar f \in D_{\perp} \right\} \label{eq:A}\\
 & B := \left\{f\in \CC^{-\al} : \bar f \in [1-\delta,1+\delta], 
 \, f -  \bar f \in D_{\perp}\right\} \label{eq:B}
\end{align}
where $D_\perp$ is a closed ball of radius $\delta$ in $\CC^{-\al}$ and $\bar f =  L^{-2} \lng f, 1 \rng$. If necessary we write
$A(\al;\delta)$ and $B(\al;\delta)$ to denote the specific value of the parameters $\al$ and $\delta$. Last for $x\in A$
we define
\begin{equation*}
 \tau_B(X_N(\cdot;x)) := \inf\left\{t>0: X_N(t;x)\in B\right\}
\end{equation*}
and
\begin{equation*}
 \tau_B(X(\cdot;x)) := \inf\left\{t>0: X(t;x)\in B\right\}.
\end{equation*}
For $k\in \ZZ^2$ let
\begin{equation*}
 \lambda_k := \left(\frac{2\pi|k|}{L}\right)^2 -1 \text{ and } \quad \nu_k : = \left(\frac{2\pi|k|}{L}\right)^2 +2 = \lambda_k + 3.
\end{equation*}
The sequences $\{\lambda_k\}_{k\in \ZZ^2}$ and $\{\nu_k\}_{k\in \ZZ^2}$ are the eigenvalues of the operators $-\Delta -1 $ and
$-\Delta + 2$ endowed with periodic boundary conditions. 

The next theorem is essentially \cite[Theorem 2.3]{BGW17}.  

\begin{theorem}[{\cite[Theorem 2.3]{BGW17}}] \label{thm:aver_approx} Let $0 < L < 2\pi$. For every $\al>0$, $\delta\in(0,1/2)$ and 
$\eps\in(0,1)$ there exists a sequence $\{\mu_{\eps,N}\}_{N\geq 1}$ of probability measures concentrated on $\partial A$ such that
\begin{align}
 & \limsup_{N\to\infty}\int \EE\tau_B(X_N(\cdot;x)) \, \mu_{\eps,N}(\dd x)
 \leq \frac{2\pi}{|\lambda_0|} \sqrt{\prod_{k\in \ZZ^2}\frac{|\lambda_k|}{\nu_k} 
 \exp\left\{\frac{\nu_k-\lambda_k}{\lambda_k+2}\right\}} \ee^{\left(V(0) - V(-1)\right)/\eps}\left(1+c_+\sqrt{\eps}\right) 
 \label{eq:aver_approx}\\
 & \liminf_{N\to\infty}\int \EE\tau_B(X_N(\cdot;x)) \, \mu_{\eps,N}(\dd x) \geq \frac{2\pi}{|\lambda_0|} \sqrt{\prod_{k\in \ZZ^2}\frac{|\lambda_k|}{\nu_k} 
 \exp\left\{\frac{\nu_k-\lambda_k}{\lambda_k+2}\right\}} \ee^{\left(V(0) - V(-1)\right)/\eps}\left(1-c_-\eps\right) \nonumber
\end{align}
where the constants $c_+$ and $c_-$ are uniform in $\eps$.
\end{theorem}

\begin{proof} The proof of \eqref{eq:aver_approx} is given in \cite[Sections 4 and 5]{BGW17}, but the following should be modified.
\begin{itemize} 
 \item In \cite{BGW17}, the sets $A$ and $B$ are defined as in \eqref{eq:A} and \eqref{eq:B} with $D_\perp$ replaced by a ball in $H^s$ for $s<0$. The explicit form of $D_\perp$ is 
 only used in \cite[Lemma 5.9]{BGW17}. There the authors consider the $0$-mean Gaussian measure $\gamma^\perp_0$ with quadratic form $\frac{1}{2\eps} \left(\|\nabla f\|_{L^2}^2 - 
 \|f-\bar f\|_{L^2}^2\right)$, and prove that $D_\perp$ has probability bounded from below by $1- c\eps^2$. Here we assume that $D_\perp$ is a ball in 
 $\CC^{-\al}$. To obtain the same estimate for this set, we first notice that the random field $f$ associated with the measure $\gamma^\perp_0$ satisfies  
 \begin{equation*}
  \EE \lng f, L^{-2}\ee^{2\ii\pi k \cdot/L} \rng  \lesssim \frac{\eps \log\eps^{-1} \log \lambda_k}{1+\lambda_k},
 \end{equation*}
 for every $k\in \ZZ^2$, where the explicit constant depends on $L$. This decay of the Fourier modes of $f$ and \cite[Proposition 3.6]{MWX17} imply that the 
 measure $\gamma^\perp_0$ is concentrated in $\CC^{-\al}$, for every $\al>0$, which in turn implies \cite[Lemma 5.9]{BGW17} for the set $D_\perp$ considered here.   
 \item In \cite{BGW17}, the authors consider \eqref{eq:X_approx} with $\Re_N$ replaced by 
 \begin{equation*}
 C_N = \frac{1}{L^2} \sum_{|k|\leq N} \frac{1}{|\lambda_k|}
 \end{equation*}
 and obtain \eqref{eq:aver_approx} with the pre-factor given by
 \begin{equation*}
  \frac{2\pi}{|\lambda_0|} \sqrt{\prod_{k\in \ZZ^2}\frac{|\lambda_k|}{\nu_k} 
  \exp\left\{\frac{\nu_k-\lambda_k}{\lambda_k}\right\}} = 
  \lim_{N\to\infty} \frac{2\pi}{|\lambda_0|} \sqrt{\prod_{|k|\leq N}\frac{|\lambda_k|}{\nu_k}}
  \exp\left\{\frac{3 L^2 C_N}{2}\right\}.
 \end{equation*}
 In our case one can check by \eqref{eq:renorm_constant} that $\Re_N$ is given by
 \begin{equation*}
 \Re_N = \frac{1}{L^2} \sum_{|k|\leq N} \frac{1}{|\lambda_k+2|}. 
 \end{equation*}
 According to \cite[Remark 2.5]{BGW17} this choice of renormalisation constant modifies \cite[Theorem 2.3]{BGW17} by multiplying 
 the pre-factor there with
 \begin{equation*}
 \exp\left\{-3 L^2 \lim_{N\to \infty} (\Re_N - C_N)/2\lambda_0\right\}. 
 \end{equation*}
\end{itemize}
\end{proof}

\begin{remark} The finite dimensional measure $\mu_{\eps,N}$ in \eqref{eq:aver_approx} is given by
 \begin{equation*}
  \mu_{\eps,N}(\dd f) = \frac{1}{\mathrm{cap}_A(B)} \ee^{-V(\Pi_Nf)/\eps} \rho_{A,B}( \dd f), 
 \end{equation*}
 where $\rho_{A,B}$ is a probability measure concentrated on $\partial A$, called the equilibrium measure, and $\mathrm{cap}_A(B)$ is a normalisation 
 constant. Under this measure and the assumption that the sets $A$ and $B$ are symmetric, the integrals appearing in \eqref{eq:aver_approx} can be 
 rewritten using potential theory as
 \begin{equation*}
  \int \EE\tau_B(X_N(\cdot;x)) \, \mu_{\eps,N}(\dd x) 
  =  \frac{1}{2\mathrm{cap}_A(B)} \int_{\RR^{(2N+1)^2}} \ee^{-V(\Pi_Nf)/\eps} \dd f. 
 \end{equation*}
 This formula is derived in \cite[Section 3]{BGW17} and it is then analysed to obtain \eqref{eq:aver_approx}. 
\end{remark}
 
In the next theorem, which is the main result of this section, we generalise \eqref{eq:aver_approx} for the limiting process $X(\cdot;x)$
for fixed initial condition $x$ in a suitable neighbourhood of $-1$. By symmetry the same results holds if we swap the neighbourhoods 
of $-1$ and $1$ below. 

\begin{theorem} \label{thm:eyr_kr} There exist $\delta_0>0$ such that the following holds. For every $\al\in(0,\al_0)$ and 
$\delta\in(0,\delta_0)$ there exist $c_+,c_->0$ and $\eps_0\in(0,1)$ such that for every $\eps\leq \eps_0$ 
\begin{align}
 & \sup_{x\in {A(\al_0;\delta)}} \EE\tau_{B(\al;\delta)} (X(\cdot;x)) \leq \frac{2\pi}{|\lambda_0|} \sqrt{\prod_{k\in \ZZ^2}\frac{|\lambda_k|}{\nu_k} 
 \exp\left\{\frac{\nu_k-\lambda_k}{\lambda_k+2}\right\}} \ee^{\left(V(0) - V(-1)\right)/\eps}\left(1+c_+\sqrt{\eps}\right). 
 \label{eq:eyr_kr_final}\\
 & \inf_{x\in A(\al_0;\delta)} \EE\tau_{B(\al;\delta)}(X(\cdot;x)) \geq \frac{2\pi}{|\lambda_0|} \sqrt{\prod_{k\in \ZZ^2}\frac{|\lambda_k|}{\nu_k} 
 \exp\left\{\frac{\nu_k-\lambda_k}{\lambda_k+2}\right\}} \ee^{\left(V(0) - V(-1)\right)/\eps}\left(1-c_-\eps\right). \nonumber
\end{align}
\end{theorem}

\begin{proof} See Section \ref{s:eyr_kram_proof}.
\end{proof}

To prove this theorem we first fix $\al\in(0,\al_0)$ and pass to the limit as $N\to \infty$ in \eqref{eq:aver_approx} to prove a version of \eqref{eq:eyr_kr_final} where the initial condition $x$ 
is averaged with respect to a measure $\mu_\eps$ concentrated on a closed ball with respect to the weaker topology $\CC^{-\al_0}$ (see Proposition \ref{prop:aver_lim}). This measure is the weak limit, up to a subsequence, of the 
measures $\mu_{\eps,N}$ in Theorem \ref{thm:aver_approx}. We then use our ``exponential loss of memory'' result, Theorem \ref{thm:exp_contr}, to pass from averages of initial 
conditions with respect to the limiting measure $\mu_\eps$ to fixed initial conditions.

The rest of this section is structured as follows. In Section \ref{s:gal_conv_a_priori} we prove convergence of the Galerkin approximations $X_N(\cdot;x_N)$ and obtain estimates uniform
in the initial condition $x$ and the regularisation parameter $N$. In Section \ref{s:lim_pass} we prove uniform integrability of the stopping times $\tau_B(X(\cdot;x))$ and pass to the 
limit as $N\to \infty$ in \eqref{eq:aver_approx}. Finally in Section \ref{s:eyr_kram_proof} we prove Theorem \ref{thm:eyr_kr}.

\subsection{Convergence of the Galerkin scheme and a priori estimates} \label{s:gal_conv_a_priori}

In the next proposition we prove convergence of $X_N(\cdot;x)$ to $X(\cdot;x_N)$ in $C([0,T];\CC^{-\al})$ using convergence of the stochastic objects $\<n>_N$ 
which is proven in \cite[Proposition 2.3]{TW18}. This is a technical result and the proof is given in the Appendix. 

\begin{proposition} \label{prop:X_approx_conv} Let $\aleph\subset \CC^{-\al_0}$ be bounded and assume that for every $x\in \aleph$, 
there exists a sequence $\{x_N\}_{N\geq 1}$ such that $x_N\to x$ uniformly in $x$. Then for every $\al\in(0,\al_0)$ 
and $0<s<T$ 
\begin{equation*}
 \lim_{N\to\infty} \sup_{x\in \aleph} \sup_{t\in[s,T]} \|X_N(t;x_N) - X(t;x)\|_{\CC^{-\al}} = 0
\end{equation*}
in probability. 
\end{proposition}

\begin{proof} See Appendix \hyperlink{proof:X_approx_conv}{F}.
\end{proof}

The next proposition provides a bound for $X_N(\cdot;x)$ uniformly in the initial condition $x$ and the regularisation parameter $N$ in the $\BB^{-\al}_{2,2}$ norm, for $0<\al<\al_0$.
This result has been already established in \cite[Corollary 3.10]{TW18} for the limiting process $X(\cdot;x)$ in the $\CC^{-\al}$ norm. There \eqref{eq:rem_eq} is tested
with $v(\cdot;x)^{p-1}$, for $p\geq 2$ even, to bound $\|v(\cdot;x)\|_{L^p}$ by using the ``good'' sign of the non-linear term $-v^3$. In the case of \eqref{eq:v_approx} 
this argument allows us to bound $\|v_N(\cdot;x)\|_{L^p}$ for $p=2$ only, because of the projection $\Pi_N$ in front of the non-linearity. 

\begin{proposition}\label{prop:X_approx_bd} For every $\al\in(0,\al_0]$ and $p\geq 1$ we have that
\begin{equation}\label{eq:X_approx_bd}
 \sup_{N\geq 1}\sup_{x\in\CC^{-\al_0}}\sup_{t\leq 1} t^{\frac{p}{2}} 
 \EE\|X_N(s;x)\|_{\BB^{-\al}_{2,2}}^p <\infty. 
\end{equation}
\end{proposition}

\begin{proof} Proceeding exactly as in the proof of \cite[Proposition 3.7]{TW18} we first show that there exist $\al \in(0,1)$ 
and $p_n\geq 1$ such that for every $t\in (0,1)$
\begin{equation} \label{eq:v_approx_bd}
 \|v_N(t;x)\|_{L^2}^2 \lesssim t^{-1}\vee 
 \left(\sum_{n=1}^3 t^{-\al'(n-1)p_n} \sup_{s\leq t} s^{\al'(n-1)p_n} \|\eps^{\frac{n}{2}}\<n>_N(s)\|_{\CC^{-\al}}^{p_n}\right)^{\frac{1}{2}}
\end{equation}
for every $\al'\in(0,1)$, uniformly in $x\in \CC^{-\al_0}$. We then proceed as in the proof of \cite[Corollary 3.10]{TW18} 
and use \eqref{eq:v_approx_bd} to prove \eqref{eq:X_approx_bd}. The only difference is that here we use the norm 
$\|\cdot\|_{\BB^{-\al}_{2,2}}$ and the embedding $L^2\hookrightarrow \BB^{-\al}_{2,2}$ on the level of $v_N(\cdot;x)$ together with the 
fact that
\begin{equation*}
 \sup_{N\geq 1}\EE\left(\sup_{t\leq 1} t^{(n-1)\al'}\|\<n>_N(t)\|_{\CC^{-\al}}\right)^p < \infty
\end{equation*}
for every $\al,\al'>0$ and $p\geq1$, which is immediate from \cite[Proposition 2.2, Proposition 2.3]{TW18}. 
\end{proof}

\subsection{Passing to the limit} \label{s:lim_pass}

In this section we pass to the limit as $N\to\infty$ in \eqref{eq:aver_approx} using uniform integrability of the stopping time $\tau_B(X_N(\cdot;x))$. To obtain uniform
integrability we prove exponential moment bounds for $\tau_B(X_N(\cdot;x))$ uniformly in the initial condition $x \in \CC^{-\al_0}$ and the regularisation parameter $N$. We first bound 
$\PP\left(\tau_B(X_N(\cdot;x)) \geq 1\right)$ using a support theorem and a strong a priori bound for $X_N(\cdot;x)$ in $\CC^{-\al}$. A support theorem for the limiting process 
$X(\cdot;x)$ has been already established in \cite[Corollary 6.4]{TW18}. To use it for $X_N(\cdot;x)$ we combine it with the convergence result in Proposition \ref{prop:X_approx_conv}. To obtain a strong a priori bound 
for $X_N(\cdot;x)$ in $\CC^{-\al}$ we first use Proposition \ref{prop:X_approx_bd} which implies the bound in $\BB^{-\al}_{2,2}$ and then use Proposition \ref{prop:L2_to_Linfty} to pass from 
the $\BB^{-\al}_{2,2}$ norm to the $\CC^{-\al}$ norm. 

\begin{proposition}\label{prop:exp_tails} For every $\al\in(0,\al_0)$, $\delta\in(0,1/2)$ and $\eps\in(0,1)$ there exist $c_0 \in(0,1)$ and $N_0\geq 1$ such that for every 
$N\geq N_0$
\begin{equation*}
 \sup_{x\in \CC^{-\al_0}}\PP\left(\tau_B(X_N(\cdot;x)) \geq 1\right) \leq c_0.
\end{equation*}
\end{proposition}

\begin{proof} Let $\al\in(0,\al_0)$ and let $\aleph$ be a compact subset of $\CC^{-\al_0}$ which we fix below. Using the Markov property   
\begin{equation*}
 \PP(\tau_B(X_N(\cdot;x))\geq 1) \leq \sup_{y\in \aleph} \PP(\tau_B(X_N(\cdot;y))\geq 1/2) \, \PP(X_N(1/2;x)\in \aleph) 
 + \PP(X_N(1/2;x)\notin \aleph).
\end{equation*}
The proof is complete if for every $N\geq N_0$
\begin{equation}\label{eq:prob_bd}
 \sup_{y\in \aleph} \PP(\tau_B(X_N(\cdot;y))\geq 1/2) < 1, \,  \sup_{x\in \CC^{-\al_0}}\PP(X_N(1/2;x)\notin \aleph) <1.
\end{equation}
We notice that there exists $\delta'>0$ such that for any $y\in \aleph$
\begin{align}
 \PP(\tau_B(X_N(\cdot;y))\leq 1/2) & \geq \PP(X_N(1/2;y)\in B) \geq \PP\left(X(1/2;y)\in B_{\CC^{-\al}}(1;\delta')\right) \label{eq:bd_1} \\
 & \quad - \PP\left(\|X_N(1/2;y) - X(1/2;y)\|_{\CC^{-\al}} \geq \delta'\right). \nonumber
\end{align}
Here we use that if $\|X(1/2;y)-1\|_{\CC^{-\al}},\|X_N(1/2;y) - X(1/2;y)\|_{\CC^{-\al}} \leq \delta'$, then $X_N(1/2;y) \in B$ for 
$\delta'$ sufficiently small. By the support theorem \cite[Corollary 6.4]{TW18} there exists $c_1 \equiv c_1(\delta,\eps)>0$ such that
\begin{equation} \label{eq:bd_2}
 \inf_{y\in \aleph}\PP\left(X(1/2;y) \in B_{\CC^{-\al}}(1; \delta')\right) \geq c_1.
\end{equation}
On the other hand Proposition \ref{prop:X_approx_conv} implies convergence in probability of $X_N(1/2;y)$ to $X(1/2;y)$ in $\CC^{-\al}$ 
uniformly in $y\in \aleph$. Hence there exists $N_0\geq 1$ such that for every $N\geq N_0$
\begin{equation} \label{eq:bd_3}
 \sup_{y\in \aleph}\PP\left(\|X_N(1/2;y) - X(1/2;y)\|_{\CC^{-\al}} \geq \delta/2\right) \leq c_1/2.
\end{equation}
Plugging \eqref{eq:bd_2} and \eqref{eq:bd_3} in \eqref{eq:bd_1} implies the first bound in \eqref{eq:prob_bd}.

We now prove the second bound in \eqref{eq:prob_bd}. By the Markov inequality for every $R>0$
\begin{equation*}
 \PP\left(\|X_N(1/4;x)\|_{\BB^{-\al}_{2,2}} \geq R\right) \leq \frac{1}{R} \EE\|X_N(1/4;x)\|_{\BB^{-\al}_{2,2}}.
\end{equation*}
By \eqref{eq:X_approx_bd} the expectation on the right hand side of the last inequality is uniformly bounded over 
$x\in \CC^{-\al_0}$ and $N\geq 1$. Thus choosing $R>0$ large enough  
\begin{equation} \label{eq:ent_L2}
 \sup_{x\in \CC^{-\al_0}} \PP\left(\|X_N(1/4;x)\|_{\BB^{-\al}_{2,2}} \geq R\right) \leq \frac{1}{2}.
\end{equation}
By Proposition \ref{prop:L2_to_Linfty} for every $K,R>0$ there exist $C\equiv C(K,R)$ such that 
\begin{equation*} 
 \sup_{\|y\|_{\BB^{-\al}_{2,2}}\leq R} \PP\left(\|X_N(1/4;y)\|_{\CC^{-\al}} \geq C\right) \leq 
 \PP\left(\sup_{t\leq 1} t^{(n-1)\al'}\|\eps^{\frac{n}{2}}\<n>_N(t)\|_{\CC^{-\al}} \geq K\right).
\end{equation*}
Choosing $K$ sufficiently large, combining the last inequality with \cite[Propositions 2.2 and 2.3]{TW18} and 
using the Markov inequality imply that 
\begin{equation} \label{eq:ent_Linfty}
 \sup_{\|y\|_{\BB^{-\al}_{2,2}}\leq R} \PP\left(\|X_N(1/4;y)\|_{\CC^{-\al}} \geq C\right) \leq \frac{1}{2}.
\end{equation}
Using the Markov property and \eqref{eq:ent_L2} and \eqref{eq:ent_Linfty} we get for arbitrary $x\in \CC^{-\al_0}$
\begin{align*}
 & \PP\left(\|X_N(1/2;x)\|_{\CC^{-\al}} \geq C\right) \\ 
 & \quad \leq \PP\left(\|X_N(1/4;x)\|_{\BB^{-\al}_{2,2}} \leq R\right) \sup_{y\in \BB^{-\al}_{2,2}} \PP(\|X_N(1/4;y)\|_{\CC^{-\al}} \geq C)    
 + \PP\left(\|X_N(1/4;x)\|_{\BB^{-\al}_{2,2}} \geq R\right) 
 \leq \frac{3}{4}.
\end{align*}
We finally notice that for every $\al<\al_0$ the set $\aleph = \{f\in \CC^{-\al_0}: \|f\|_{\CC^{-\al}}\leq C\}$ 
is compact in $\CC^{-\al_0}$ which implies the second bound in \eqref{eq:prob_bd}. 
\end{proof}

In the next corollary we use Proposition \ref{prop:exp_tails} to prove exponential moments for the stopping time $\tau_B(X_N(\cdot;x))$.

\begin{corollary} \label{cor:stop_exp_mom} For every $\delta>0$ and $\eps\in(0,1)$ there exist $\eta_0>0$ and $N_0\geq 1$ such that 
\begin{equation*}
 \sup_{N\geq N_0}\sup_{x\in \CC^{-\al_0}} \EE\exp\{\eta_0 \tau_B(X_N(\cdot;x))\} <\infty.
\end{equation*}
\end{corollary}

\begin{proof} By the Markov property we have that 
\begin{align*}
 \PP(\tau_B(X_N(\cdot;x)) \geq k+1) 
 \leq \sup_{y\in \CC^{-\al_0}} \PP(\tau_B(X_N(\cdot;y))\geq 1) \, \PP(\tau_B(X_N(\cdot;x))\geq k).
\end{align*}
Iterating this inequality and using Proposition \ref{prop:exp_tails} we obtain that
\begin{equation*}
 \sup_{x\in \CC^{-\al_0}} \PP(\tau_B(X_N(\cdot;x)) \geq k+1) \leq c_0^{k+1}.
\end{equation*}
Then
\begin{align*}
 \EE\exp\{\eta_0 \tau_B(X_N(\cdot;x))\} & = 1 + \int_0^\infty \eta_0 \ee^{\eta_0 t} \PP(\tau_B(X_N(\cdot;x)) \geq t) \dd t \\
 & \leq 1 + \sum_{k=0}^\infty \PP(\tau_B(X_N(\cdot;x)) \geq k) \int_k^{k+1} \eta_0 \ee^{\eta_0 t} \dd t \\
 & \leq 1 + \ee^{\eta_0} \sum_{k=0}^\infty \ee^{\eta_0 k} c_0^k 
\end{align*}
and the proof is complete if we choose $\eta_0<\log c_0^{-1}$.
\end{proof}

In the next proposition we pass to the limit as $N\to \infty$ in \eqref{eq:aver_approx}. Here we use Corollary \ref{cor:stop_exp_mom}, 
which implies uniform integrability of $\tau_B(X_N(\cdot;x))$, and the weak convergence of the measures $\mu_{\eps,N}$. 

\begin{proposition} \label{prop:aver_lim} For every $\al\in(0,\al_0)$, $\delta\in(0,1/2)$ except possibly a countable 
subset, and $\eps\in(0,1)$ there exists a probability measure 
$\mu_\eps\in \mathcal{M}_1\left(A(\al_0;\delta)\right)$ such that
\begin{align}
 & \int \EE\tau_{B(\al;\delta)}(X(\cdot;x)) \, \mu_\eps(\dd x) \leq \frac{2\pi}{|\lambda_0|} \sqrt{\prod_{k\in \ZZ^2}\frac{|\lambda_k|}{\nu_k} 
 \exp\left\{\frac{\nu_k-\lambda_k}{\lambda_k+2}\right\}} \ee^{\left(V(0) - V(-1)\right)/\eps}\left(1+c_+\sqrt{\eps}\right) \label{eq:aver_lim}\\
 & \int \EE\tau_{B(\al;\delta)}(X(\cdot;x)) \, \mu_\eps(\dd x) \geq \frac{2\pi}{|\lambda_0|} \sqrt{\prod_{k\in \ZZ^2}\frac{|\lambda_k|}{\nu_k} 
 \exp\left\{\frac{\nu_k-\lambda_k}{\lambda_k+2}\right\}} \ee^{\left(V(0) - V(-1)\right)/\eps}\left(1-c_-\eps\right) \nonumber
\end{align}
where the constants $c_+$ and $c_-$ are uniform in $\eps$.
\end{proposition}

\begin{proof} We only prove the upper bound in \ref{eq:aver_lim}. The lower bound follows similarly. 

Let $\al\in(0,\al_0)$ and $\delta\in(0,1/2)$. Using the compact embedding $\CC^{-\al} \hookrightarrow \CC^{-\al_0}$ (see Proposition \ref{prop:comp_emb}), 
for any $\al<\al_0$, we have that $A(\al;\delta) \subset
A(\al_0;\delta)$. Let $\{\mu_{\eps,N}\}_{N\geq 1}$ be the family of probability measures in \eqref{eq:aver_approx}. 
Using again the compact embedding $\CC^{-\al} \hookrightarrow \CC^{-\al_0}$, for any $\al<\al_0$, this family is trivially tight since it is 
concentrated on $\partial A(\al;\delta)$. Hence there exists $\mu_\eps \in \mathcal{M}_1\left(A(\al_0;\delta)\right)$ such
that $\mu_{\eps,N} \stackrel{\text{weak}}{\to} \mu_\eps$ up to a subsequence. 

By Skorokhod's represantation theorem (see \cite[Theorem 2.4]{DPZ92}) there exist a probability space $(\Omega_\mu, \mathcal{F}_\mu, \PP_\mu)$ and random variables 
$\{x_N\}_{N\geq 1}$ and $x$ taking values in $A(\al_0;\delta)$ such that $x_N\stackrel{\text{law}}{=} \mu_N$,
$x\stackrel{\text{law}}{=} \mu_\eps$ and $x_N\to x$ $\PP_{\mu_\eps}$-almost surely in $\CC^{-\al_0}$. If we denote by 
$\EE_{\PP\otimes\PP_{\mu_\eps}}$ the expectation of the probability measure $\PP\otimes\PP_{\mu_\eps}$, we have that
\begin{align} 
 & \EE_{\PP\otimes\PP_{\mu_\eps}}\tau_{B(\al;\delta)}(X_N(\cdot;x_N)) = \int \EE\tau_{B(\al;\delta)}(X_N(\cdot;x)) \, \mu_N(\dd x) \label{eq:exp_eq} \\
 & \EE_{\PP\otimes\PP_{\mu_\eps}}\tau_{B(\al;\delta)}(X(\cdot;x)) = \int \EE\tau_{B(\al;\delta)}(X(\cdot;x)) \, \mu_\eps(\dd x). \nonumber
\end{align}
By Proposition \ref{prop:X_approx_conv} $X_N(\cdot;x_N)$ converges to $X(\cdot;x)$ $\PP\otimes\PP_{\mu_\eps}$-almost surely on 
compact time intervals of $(0,\infty)$ up to a subsequence. Let 
\begin{equation*}
 \mathtt{L} = \left\{\delta\in(0,1/2): \PP\left(\tau_{B(\al;\delta)}(\cdot) \text{ is discontinuous on } X(\cdot;x)\right) >0\right\}
\end{equation*}
and notice that for $x(t) = L^{-2}\lng X(t;x), 1\rng$
\begin{equation*}
 \mathtt{L} \subset \left\{\delta\in(0,1/2): 
 \PP\left(t\mapsto |x(t) - 1|\vee\|X(t;x) - x(t)\|_{\CC^{-\al}} \text{ has a local minimum at height } 
 \delta \right) > 0
 \right\}.
\end{equation*}
As in \cite[Proof of Theorem 6.1]{MW17i} the last set is at most countable, hence $\tau_{B(\al;\delta)}(X_N(\cdot;x_N))
\to \tau_{B(\al;\delta)}(X(\cdot;x))$ 
$\PP\otimes \PP_{\mu_\eps}$-almost surely up to a subsequence, except possibly a countable number of $\delta\in(0,1/2)$. 

By Corollary \ref{cor:stop_exp_mom} the family $\{\tau_{B(\al;\delta)}(X_N(\cdot;x))\}_{N\geq N_0}$ is uniformly integrable. Hence
by Vitali's convergence theorem (see \cite[Theorem 4.5.4]{Bo07}) we obtain that 
\begin{equation*}
 \EE_{\PP\otimes \PP_{\mu_\eps}}\tau_{B(\al;\delta)}(X_N(\cdot;x_N)) \to \EE_{\PP\otimes \PP_{\mu_\eps}}\tau_{B(\al;\delta)}(X(\cdot;x)).
\end{equation*}
Combining with \eqref{eq:aver_approx} and \eqref{eq:exp_eq} the proof of the upper bound is complete.
\end{proof}

\subsection{An Eyring-Kramers law} \label{s:eyr_kram_proof}

In this section we combine Proposition \ref{prop:aver_lim} and Theorem \ref{thm:exp_contr} to prove Theorem \ref{thm:eyr_kr}.
The idea we use here was first implemented in the $1$-dimensional case in \cite{BG13}. Generally speaking, if we restrict 
ourselves on the event where the first transition from a neighbourhood of $-1$ to a neighbourhood of $1$ happens after the
``exponential loss of memory'', $\tau_{B(\al;\delta)}(X(\cdot;x))$ behaves like $\int \tau_{B(\al;\delta)}(X(\cdot;x)) \, \mu_{\eps}(\dd x)$ 
for $x\in A(\al_0;\delta)$.
The probability of this event is quantified by Theorem \ref{thm:exp_contr} and Proposition \ref{prop:rev_ent_est}. 
On the complement of this event the transition time $\tau_{B(\al;\delta)}(X(\cdot;x))$ is estimated using Proposition 
\ref{prop:prefact_free_est}.  

In the next proposition we prove that the first transition from a neighbourhood of $-1$ to a neighbourhood of $1$ happens only
after some time $T_0>0$ with overwhelming probability. This is a large deviation event which can be estimated using continuity 
of $X$ with respect to the initial condition $x$ and the stochastic objects $\left\{\eps^{\frac{n}{2}}\<n>\right\}_{n\leq 3}$. 
We sketch the proof for completeness. 

\begin{proposition} \label{prop:rev_ent_est} For every $\al\in(0,\al_0)$ and $\delta\in(0,1/2)$ there exist $a_0,\delta_0,T_0>0$ and 
$\eps_0\in(0,1)$ such that for every $\eps\leq \eps_0$
\begin{equation*}
 \sup_{\|x-(-1)\|_{\CC^{-\al_0}}\leq \delta_0} \PP(\tau_{B(\al;\delta)}(X(\cdot;x)) \leq T_0) \leq \ee^{-a_0/\eps}.
\end{equation*}
\end{proposition}

\begin{proof} We first notice that for $\|x-(-1)\|_{\CC^{-\al_0}}\leq \delta_0$
\begin{equation*}
 \PP(\tau_{B(\al;\delta)}(X(\cdot;x)) \geq T_0) \geq 
 \PP\left(\sup_{t\leq T_0} \|X(t;x) - (-1)\|_{\CC^{-\al_0}} \leq \delta_1\right) 
\end{equation*}
for some $\delta_1>0$. Using continuity of $X$ with respect to $x$ and the stochastic objects $\left\{\eps^{\frac{n}{2}}\<n>\right\}_{n\leq 3}$, the last probability can be estimated from below uniformly in 
$\|x-(-1)\|_{\CC^{-\al_0}}\leq \delta_0$, for $\delta_0$ sufficiently small, by
\begin{equation*}
 \PP\left(\sup_{t\leq T_0} \|X(t;x) - (-1)\|_{\CC^{-\al_0}} \leq \delta_1\right) \geq 
 \PP\left(\sup_{t\leq T_0}(t\wedge 1)^{-(n-1)\al'}\|\eps^{\frac{n}{2}}\<n>(t)\|_{\CC^{-\al}} \leq \delta_2\right)
\end{equation*}
for some $\delta_2>0$. Last by Proposition \ref{prop:exp_mom} we find $a_0>0$ and $\eps_0\in(0,1)$ such that for ever $\eps\leq \eps_0$
\begin{equation*}
 \PP\left(\sup_{t\leq T_0}(t\wedge 1)^{-(n-1)\al'}\|\eps^{\frac{n}{2}}\<n>(t)\|_{\CC^{-\al}} 
 \leq \delta_2\right) \geq 1-\ee^{-a_0/\eps}
\end{equation*}
which completes the proof.
\end{proof}

In the next proposition we estimate the second moment of the transition time $\tau_{B(\al;\delta)}(X(\cdot;x))$ using the large 
deviation estimate \eqref{eq:LDP}. The proof combines the ideas in Propositions \ref{prop:1st_entry} and \ref{prop:ent_est}. However
here we construct a path $g$ which is different from the one in the proof of Proposition \ref{prop:1st_entry} to ensure that the process
$X(\cdot;x)$ returns to a neighbourhood of $-1$. The same proof implies exponential moments of the transition time 
$\tau_{B(\al;\delta)}(X(\cdot;x))$, but we only need to estimate the second moment in the proof of Theorem \ref{thm:eyr_kr}. 

\begin{proposition} \label{prop:prefact_free_est} Let $\al\in(0,\al_0)$ and $\delta\in(0,1/2)$. For every $\eta>0$ there exists 
$\eps_0\in(0,1)$ such that for every $\eps\leq \eps_0$
\begin{equation*}
 \sup_{x\in \CC^{-\al_0}} \EE\tau_{B(\al;\delta)}(X(\cdot;x))^2 \leq C \ee^{2\left[(V(0)-V(-1))+\eta\right]/\eps}
\end{equation*}
for some $C>0$ independent of $\eta$ and $\eps$.
\end{proposition}

\begin{proof} We first prove that for every $R,\eta>0$ there exists $T_0>0$ and $\eps_0\in(0,1)$ such that for every 
$\eps\leq \eps_0$  
\begin{equation*}
 \sup_{\|x\|_{\CC^{-\al_0}}\leq R}\PP(\tau_{B(\al;\delta)}(X(\cdot;x))\geq T_0) \leq 1 - \ee^{-\left[(V(0)-V(-1))+ \eta\right]/\eps}.
\end{equation*}
We notice that there exists $\delta'>0$ such that 
\begin{equation*}
 \PP(\tau_{B(\al;\delta)}(X(\cdot;x))\leq T_0) \geq \PP(\underbrace{\|X(T_*;x) - 1\|_{\CC^{-\al}} \leq \delta' 
 \text{ for some } T_*\leq T_0}_{=:\mathcal{A}(T_0;x)}).
\end{equation*}
Here we use that if $\|X(T_*;x) - 1\|_{\CC^{-\al}} \leq \delta'$, for $\delta'$ sufficiently small then $X(T_*;x)\in B(\al;\delta)$. By the large deviation estimate \eqref{eq:LDP} we need to bound
\begin{equation*}
 \sup_{\|x\|_{\CC^{-\al_0}}\leq R} \sup_{\substack{f \in \mathcal{A}(T_0;x) \\ f(0) = x}}I(f(\cdot;x)).
\end{equation*}
To do so we proceed as in the proof of Proposition \ref{prop:1st_entry} by constructing a suitable path $g\in \mathcal{A}(T_0;x)$. 
The construction here is similar but some of the steps differ since we need to ensure that $g$ returns to a neighbourhood of 
$1$. To avoid repeating ourselves we give a sketch of the proof highlighting the different steps of the construction.

Steps 1, 2 and 3 are exactly as in the proof of Proposition \ref{prop:1st_entry}. However we need to distinguish the value of $\delta$
there from the value of $\delta$ in the statement of the proposition. If $g(\tau_3;x) \in B_{\BB^1_{2,2}}(1;\delta)\cap 
B_{\CC^{2+\lambda}}(0;C)$ we stop at Step 3. If not then $g(\tau_3;x) \in B_{\BB^1_{2,2}}(-1;\delta)\cap B_{\CC^{2+\lambda}}(0;C)$ or $B_{\BB^1_{2,2}}(0;\delta)\cap 
B_{\CC^{2+\lambda}}(0;C)$. We only explain how to proceed in the first case since it also covers the other.  

Before we describe the remaining steps we recall that by Proposition \ref{prop:unst} there exist 
$y_{0,-},y_{0,+} \in B_{\BB^1_{2,2}}(0;\delta)$ such that $y_{0,-},y_{0,+}\in \CC^\infty$ and $X_{det}(t;y_{0,\pm}) 
\to \pm1$ in $\BB^1_{2,2}$. In particular there exists $T_0^*>0$ such that $X_{det}(T_0^*;y_{0,\pm})\in B_{\BB^1_{2,2}}(\pm1;\delta)
\cap B_{\CC^{2+\lambda}}(0;C)$.  

\smallskip

\underline{Step 4 (Jump to $X_{det}(T_0^*;y_{0,-})$)}:

Let $\tau_4 = \tau_3+\tau$, for $\tau>0$ as in Step 2 which we fix below according to Lemma \ref{lem:FJL}. For $t\in[\tau_3,\tau_4]$ we set 
$g(t;x) = g(\tau_3;x) + \frac{t - \tau_3}{\tau_4 - \tau_3}(X_{det}(T_0^*;y_{0,-}) - g(\tau_3;x))$.

\smallskip

\underline{Step 5 (Follow the deterministic flow backward to reach $0$)}:

Let $\tau_5 = \tau_4+ T_0^*$. For $t\in[\tau_4,\tau_5]$ we set $g(t;x) = X_{det}(\tau_5 - t;y_{0,-})$.

\smallskip

\underline{Step 6 (Jump to $y_{0,+}$)}:

Let $\tau_6 = \tau_5 +\tau$, for $\tau$ as in Step 4. For $t\in[\tau_5,\tau_6]$ we set 
$g(t;x) = g(\tau_5;x) + \frac{t - \tau_5}{\tau_6-\tau_5}(y_{0,+} - g(\tau_5;x))$.

\smallskip

\underline{Step 7 (Follow the deterministic flow forward to reach $1$)}:

Let $\tau_7 = \tau_6+ T_0^*$. For $t\in[\tau_6,\tau_7]$ we set $g(t;x) = X_{det}(t - \tau_6;y_{0,+})$.

For the path $g$ constructed above we notice that for every $\|x\|_{\CC^{-\al_0}} \leq R$, if $t\geq \tau_7$ then
$g(t;x)\in B_{\BB^1_{2,2}}(1;\delta)$. By \eqref{eq:Besov_Emb}, $\BB^1_{2,2} \subset \CC^{-\al}$, for every $\al>0$,
hence if we choose $\delta$ sufficiently small and set $T_0 = \tau_7+1$ then $g\in \mathcal{A}(T_0;x)$.

To bound $I(g(\cdot;x))$ we proceed exactly as in the proof of Proposition \ref{prop:1st_entry} using Lemma \ref{lem:FJL}. But 
when considering the contribution from Step 5 we get
\begin{align*}
 & \frac{1}{4}\int_{\tau_4}^{\tau_5} \|(\partial_t - \Delta) g(t;x) + g(t;x)^3 - g(t;x)\|_{L^2}^2 \dd t \\
 & \quad  = 2 \int_0^{T_0^*} \left\lng \partial_t X_{det}(t;y_{0,+}), \Delta X_{det}(t;y_{0,+}) - X_{det}(t;y_{0,+})^3 
 + X_{det}(t;y_{0,+}) \right\rng \dd t \\
 & \quad = - 2 \left(V(X_{det}(T_0^*;y_0)) - V(y_{0,+})\right) \\
 & \quad \leq 2 \left(V(0) - V(-1)\right).
\end{align*}
In total we obtain the bound
\begin{equation*}
 \sup_{\|x\|_{\CC^{-\al_0}}\leq R} I(g(\cdot;x)) \leq  2\left(V(0) - V(-1)\right) + C\delta. 
\end{equation*}
For $\eta>0$ we choose $\delta$ even smaller to ensure that $C\delta <\eta$. Then by \eqref{eq:LDP} we find $\eps_0\in(0,1)$ such
that for every $\eps\leq \eps_0$ 
\begin{equation*}
 \inf_{\|x\|_{\CC^{-\al_0}}\leq R} \PP(\tau_{B(\al;\delta)}(X(\cdot;x)) \leq T_0) \geq \ee^{-\left[(V(0) - V(-1)) + \eta\right]/\eps}.
\end{equation*}
The next step is to use the this estimate to show that for any $\eta>0$ there exists $\eps_0\in(0,1)$ and possibly a different 
$T_0>0$ such that for every $\eps\leq \eps_0$
\begin{equation*}
 \sup_{x\in\CC^{-\al_0}} \PP(\tau_{B(\al;\delta)}(X(\cdot;x)) \geq m T_0) \leq 
 \left(1 - \ee^{-\left[(V(0) - V(-1)) + \eta\right]/\eps}\right)^m.
\end{equation*}
We omit the proof since it is the same as the one of Proposition \ref{prop:ent_est}. 

Finally we notice that
\begin{align*}
 \EE\tau_{B(\al;\delta)}(X(\cdot;x))^2 & = \int_0^\infty 2t \, \PP(\tau_{B(\al;\delta)}(X(\cdot;x))\geq t) \dd t \\
 & \leq \sum_{m=0}^\infty \PP(\tau_{B(\al;\delta)}(X(\cdot;x))\geq mT_0) \int_{mT_0}^{(m+1)T_0} 2t \dd t \\
 & \leq 2T_0^2 \sum_{m=0}^\infty (m+1) 
 \left(1 - \ee^{-\left[(V(0) - V(-1)) + \eta\right]/\eps}\right)^m \\
 & = 2 T_0^2 \ee^{2\left[(V(0) - V(-1)) + \eta\right]/\eps}
\end{align*}
which completes the proof. 
\end{proof}

\begin{proof}[Proof of Theorem \ref{thm:eyr_kr}] Let 
\begin{equation*}
 \mathtt{Pr}(\eps) = \frac{2\pi}{|\lambda_0|} \sqrt{\prod_{k\in \ZZ^2}\frac{|\lambda_k|}{\nu_k} 
 \exp\left\{\frac{\nu_k-\lambda_k}{\lambda_k+2}\right\}} \ee^{\left(V(0) - V(-1)\right)/\eps} 
\end{equation*}
and $\delta\in(0,\delta_0)$, for $\delta_0\in(0,1/2)$ which we fix below. 

To prove the upper bound in \eqref{eq:eyr_kr_final} let $\delta_-<\delta$ and $T>0$ which we also fix below. 
For $x\in A(\al_0;\delta_-)$ we define the set
\begin{equation*}
 A_T(x) = \left\{\tau_{B(\al;\delta_-)}(X(\cdot;x)) > T, \sup_{\|\bar y - x\|_{\CC^{-\al_0}}\leq \delta_0}
 \frac{\|X(t;\bar y) - X(t;x)\|_{\CC^\bt}}{\|\bar y-x\|_{\CC^{-\al_0}}}
 \leq C \ee^{-(2-\kappa) t} \text{ for every } t\geq T\right\}
\end{equation*}
where $\delta_0$ and $C$ are as in Theorem \ref{thm:exp_contr}. For $y\in A(\al_0;\delta)$ and $x\in A(\al_0;\delta_-)$
we have that $\|y-x\|_{\CC^{-\al_0}}, \|x-(-1)\|_{\CC^{-\al_0}} \leq \delta_0$, if we choose $\delta_0$ sufficiently small.
Furthermore for $y\in A(\al_0;\delta)$, $x\in A(\al_0;\delta_-)$ and $\omega\in A_T(x)$ 
\begin{equation*}
 \tau_{B(\al;\delta)}(X(\cdot;y)) \leq \tau_{B(\al;\delta_-)}(X(\cdot;x)),
\end{equation*}
if we choose $T$ sufficiently large. By Proposition \ref{prop:rev_ent_est} and Theorem \ref{thm:exp_contr} there 
exist $a_1>0$ and $\eps_0\in(0,1)$ such that for every $\eps\leq\eps_0$
\begin{align*}
 \sup_{x\in A(\al_0;\delta_-)} \PP(A_T(x)^c) 
 \leq \sup_{\|x - (-1)\|_{\CC^{-\al_0}}\leq \delta_0} \PP(A_T(x)^c) \leq \ee^{-a_1/\eps}.
\end{align*}
Then for every $y\in A(\al_0;\delta)$, $x\in A(\al_0;\delta_-)$ and $\eta>0$, which we fix below, there exists $\eps_0\in(0,1)$ such 
that for every $\eps\leq \eps_0$
\begin{align}
 \EE\tau_{B(\al;\delta)}(X(\cdot;y)) & \leq \EE\tau_{B(\al;\delta_-)}(X(\cdot;x)) + \EE\tau_{B(\al;\delta)}(X(\cdot;y))\mathbf{1}_{A_T(x)^c} \label{eq:tau_control} \\
 & \stackrel{\text{Cauchy--Schwarz}}{\leq} \EE\tau_{B(\al;\delta_-)}(X(\cdot;x)) + \left(\EE\tau_{B(\al;\delta)}(X(\cdot;y))^2\right)^{\frac{1}{2}} \PP(A_T(x)^c)^{\frac{1}{2}} \nonumber \\
 & \stackrel{\text{Prop. }\ref{prop:prefact_free_est}}{\leq} \EE\tau_{B(\al;\delta_-)}(X(\cdot;x)) + C \ee^{\left((V(0) - V(-1)) + \eta - \frac{a_1}{2}\right)/\eps} \nonumber
\end{align}
for some $C>0$ independent of $\eps$. By Proposition \ref{prop:aver_lim} there exist $\delta_-\in(0,\delta)$, $c_+>0$ and 
$\mu_\eps\in \mathcal{M}_1\left(A(\al_0;\delta_-)\right)$ such that for every $\eps\in(0,1)$ 
\begin{equation*}
 \int \EE\tau_{B(\al;\delta_-)}(X(\cdot;x)) \, \mu_\eps(\dd x) \leq \mathtt{Pr}(\eps) (1+ c_+ \sqrt{\eps}).
\end{equation*}
Integrating \eqref{eq:tau_control} over $x$ with respect to $\mu_\eps$ implies that 
\begin{align*}
 & \sup_{y\in A(\al_0;\delta)} \EE\tau_{B(\al;\delta)}(X(\cdot;y)) \\
 & \quad \leq \mathtt{Pr}(\eps) \left( (1 + c_+\sqrt{\eps}) 
 +  \ee^{\left(\eta - \frac{a_1}{2}\right)/\eps} C \left(\frac{2\pi}{|\lambda_0|} \sqrt{\prod_{k\in \ZZ^2}\frac{|\lambda_k|}{\nu_k} 
 \exp\left\{\frac{\nu_k-\lambda_k}{\lambda_k+2}\right\}}\right)^{-1} \right).
\end{align*}
Let $\zeta>0$. Choosing $\eta<\frac{a_1}{2}$ we can find $\eps_0\in(0,1)$ such that for every $\eps\leq \eps_0$
\begin{equation*}
 \ee^{\left(\eta - \frac{a_1}{2}\right)/\eps} C \left(\frac{2\pi}{|\lambda_0|} \sqrt{\prod_{k\in \ZZ^2}\frac{|\lambda_k|}{\nu_k} 
 \exp\left\{\frac{\nu_k-\lambda_k}{\lambda_k+2}\right\}}\right)^{-1} \leq \zeta \sqrt{\eps}
\end{equation*}
which in turn implies that
\begin{equation*}
 \sup_{y\in A(\al_0;\delta)} \EE\tau_{B(\al;\delta)}(X(\cdot;y)) \leq \mathtt{Pr}(\eps) \left(1 + (c_+ + \zeta)\sqrt{\eps}\right) 
\end{equation*}
and proves the upper bound in \eqref{eq:eyr_kr_final}. 
 
To prove the lower bound, we let $\delta_+\in (\delta,\delta_0)$ which we fix below and for $y\in A(\al_0;\delta)$ and 
$x\in A(\al_0;\delta_+)$ we define the set
\begin{equation*}
 B_T(y,x) = \left\{\tau_{B(\al;\delta)}(X(\cdot;y)) \geq T, \sup_{\|\bar y-x\|_{\CC^{-\al_0}}\leq \delta_0} 
 \frac{\|X(t;\bar y) - X(t;x)\|_{\CC^\bt}}{\|\bar y-x\|_{\CC^{-\al_0}}}
 \leq C \ee^{-(2-\kappa) t}  \text{ for every } t\geq T\right\}.
\end{equation*}
For $y\in A(\al_0;\delta)$ and $x\in A(\al_0;\delta_+)$ we have that 
$\|y-x\|_{\CC^{-\al_0}}, \|y-(-1)\|_{\CC^{-\al_0}}, \|x-(-1)\|_{\CC^{-\al_0}} \leq \delta_0$, if we choose $\delta_0$ sufficiently 
small. We also notice that for $y\in A(\al_0;\delta)$, $x\in A(\al_0;\delta_+)$ and $\omega\in B_T(y,x)$ 
\begin{equation*}
 \tau_{B(\al;\delta_+)}(X(\cdot;x)) \leq \tau_{B(\al;\delta)}(X(\cdot;y)),
\end{equation*}
if we choose $T$ sufficiently large. By Proposition \ref{prop:rev_ent_est} and Theorem \ref{thm:exp_contr} there exists $a_1>0$ and $\eps_0\in(0,1)$ such that for
every $\eps\leq \eps_0$
\begin{equation*}
 \sup_{\substack{ y \in A(\al_0;\delta) \\ x\in A(\al_0;\delta_+)}} \PP(B_T(y,x)^c) 
 \leq \sup_{\substack{ \|y-(-1)\|_{\CC^{-\al_0}}\leq \delta_0 \\ \|x-(-1)\|_{\CC^{-\al_0}} \leq \delta_0}} \PP(B_T(y,x)^c)
 \leq \ee^{-a_1/\eps}.
\end{equation*}
Then for every $y\in A(\al_0;\delta)$, $x\in A(\al_0;\delta_+)$ and $\eps\leq \eps_0$
\begin{align*}
 \EE\tau_{B(\al;\delta)}(X(\cdot;y)) & \geq \EE\tau_{B(\al;\delta_+)}(X(\cdot;x)) \mathbf{1}_{B_T(y,x)} 
 = \EE\tau_{B(\al;\delta_+)}(X(\cdot;x)) - \EE\tau_{B(\al;\delta_+)}(X(\cdot;x)) \mathbf{1}_{B_T(y,x)^c} \\
 & \stackrel{\text{Cauchy--Schwarz}}{\geq} \EE\tau_{B(\al;\delta_+)}(X(\cdot;x)) - 
 \left(\EE\tau_{B(\al;\delta_+)}(X(\cdot;x))^2\right)^{\frac{1}{2}} \PP\left(B_T(y,x)^c\right)^{\frac{1}{2}} \\
 & \geq \EE\tau_{B(\al;\delta_+)}(X(\cdot;x)) - \left(\EE\tau_{B(\al;\delta_+)}(X(\cdot;x))^2\right)^{\frac{1}{2}} \ee^{-a_1/2\eps}
\end{align*}
and we proceed as in the case of the upper bound, using Proposition \ref{prop:prefact_free_est} for 
$\EE\tau_{B(\al;\delta_+)}(X(\cdot;x))^2$ and Proposition \ref{prop:aver_lim} to find 
$\delta_+\in(\delta,\delta_0)$, $c_->0$ and $\mu_\eps\in\mathcal{M}_1\left(A(\al_0;\delta_+)\right)$
such that for every $\eps\in(0,1)$ 
\begin{equation*}
 \int \EE\tau_{B(\al;\delta_+)}(X(\cdot;x)) \, \mu_\eps(\dd x) \geq \mathtt{Pr}(\eps) (1 - c_- \eps).
\end{equation*}
\end{proof}

\pdfbookmark{Appendix}{appendix}
\begin{appendices} 
%
\stepcounter{section}
\section*{\large Appendix \thesection} \label{app:besov_spaces} 

\begin{definition} \label{def:besov} Let $\al\in\RR$ and $p,q\in[1,\infty]$. The Besov norm $\|\cdot\|_{\BB^\al_{p,q}}$ is defined as 
\begin{equation}\label{eq:besov_norm}
\|f\|_{\mathcal{B}_{p,q}^\alpha}:=\left\|\left(2^{\alpha \kappa}\|f*\eta_\kappa\|_{L^p}\right)_{\kappa\geq-1}\right\|_{\ell^q}.
\end{equation}
Here the family of functions $\{\eta_\kappa\}_{\kappa\geq -1}$ is given by $\hat \eta_\kappa = \chi_\kappa$ in Fourier space for $\{\chi_\kappa\}_{\kappa\geq -1}$
a suitable dyadic partition of unity as in \cite[Proposition 2.10]{BCD11}. The Besov space space $\BB^\al_{p,q}$ is defined as the completion
of $\CC^\infty$ with respect to \eqref{eq:besov_norm}.
\end{definition}

In this appendix we present several useful results from \cite{MW17i, MW17ii} about Besov spaces that we repeatedly use in this article. For a complete survey of 
the full-space analogues of these results we refer the reader to \cite{BCD11}. A discussion on the validity of these results in the periodic case can be found
in \cite[Section 4.2]{MW17i}.

The following estimate is immediate from the definition of the Besov norm \eqref{eq:besov_norm},
%
%
\begin{equation}
 \|f\|_{\BB^\al_{p,q}} \leq C \|f\|_{\BB^\bt_{p,q}}, \text{ if } \bt > \al \label{eq:alpha_beta_ineq}.
\end{equation}

\begin{proposition}[{\cite[Remark 9]{MW17i}}] Let $\al\in \RR$ and $p,q_1,q_2\in[1,\infty]$ such that $q_2> q_1$. For every $\lambda>0$
\begin{equation}
 \|f\|_{\BB^\al_{p,q_2}} \leq C \|f\|_{\BB^{\al+\lambda}_{p,q_1}}. \label{eq:q1_q2_ineq}
\end{equation}
\end{proposition}

\begin{proposition}[{\cite[Remarks 10 and 11]{MW17i}}] For every $p\in[1,\infty]$
\begin{equation}
 C^{-1} \|f\|_{\BB^0_{p,\infty}} \leq \|f\|_{L^p} \leq C \|f\|_{\BB^0_{p,1}}. \label{eq:Lp_besov_control}
\end{equation}
\end{proposition}

\begin{proposition}[{\cite[Proposition 2]{MW17i}}] \label{prop:Besov_Emb} Let $\bt\geq \al$ and $p,q\geq 1$ such that $p\geq q$ and 
$\beta=\alpha+ d \left(\frac{1}{q}-\frac{1}{p}\right)$. Then
\begin{equation} \label{eq:Besov_Emb}
  \|f\|_{\BB_{p,\infty}^\alpha}\leq C \|f\|_{\BB_{q,\infty}^\beta}.
\end{equation} 
\end{proposition}

\begin{proposition}[{\cite[Proposition 5]{MW17i}}] \label{prop:Heat_Smooth} For every $\beta\geq \alpha$  
\begin{equation}\label{eq:Heat_Smooth}
  \|\ee^{t\Delta}f\|_{\BB_{p,q}^\beta}\leq C (t\wedge 1)^{\frac{\alpha-\beta}{2}}\|f\|_{\BB_{p,q}^\alpha}.
\end{equation}
\end{proposition}

\begin{proposition}[{\cite[Corollary 1]{MW17i}}] \label{prop:mult_ineq_1} Let $\al \geq 0$ and $p,q\in[1,\infty]$. Then 
\begin{equation} \label{eq:mult_ineq_1}
 \|fg\|_{\BB^\al_{p,q}} \leq C \|f\|_{\BB^\al_{p_1,q_1}} \|g\|_{\BB^\al_{p_2,q_2}},
\end{equation}
where $p = \frac{1}{p_1}+\frac{1}{p_2}$ and $p = \frac{1}{q_1}+\frac{1}{q_2}$.
\end{proposition}

\begin{proposition}[{\cite[Corollary 2]{MW17i}}] \label{prop:mult_ineq_2} Let $\al<0$, $\bt>0$ such that $\al+\bt >0$ and $p,q\in[1,\infty]$. Then 
\begin{equation} \label{eq:mult_ineq_2}
 \|fg\|_{\BB^\al_{p,q}} \leq C \|f\|_{\BB^\al_{p_1,q_1}} \|g\|_{\BB^\bt_{p_2,q_2}},
\end{equation}
where $p = \frac{1}{p_1}+\frac{1}{p_2}$ and $p = \frac{1}{q_1}+\frac{1}{q_2}$.
\end{proposition}


\begin{proposition}[{\cite[Proposition 10]{MW17i}}] \label{prop:comp_emb} For every $\al<\al'$ the embedding 
$\CC^{\al'}\hookrightarrow\BB^\al_{\infty,1}$ is compact.
\end{proposition}

\begin{proposition}[{\cite[Proposition A.6]{MW17ii}}] \label{prop:besov_grad_est} For every $p\in[1,\infty)$ 
\begin{equation*} 
 \|f\|_{\BB^1_{p,\infty}} \leq C (\|\nabla f\|_{L^p} + \|f\|_{L^p}).
\end{equation*}
\end{proposition}

\begin{proposition}[{\cite[Corollary A.8]{MW17ii}}] \label{prop:Lp_mult} Let $\al>0$ and $p,q\in[1,\infty]$. Then
\begin{equation} \label{eq:Lp_mult}
 \|f^2\|_{\BB^\al_{p,q}} \leq C \|f\|_{L^{p_1}} \|f\|_{\BB^\al_{p_2,q}},
\end{equation}
where $p = \frac{1}{p_1}+\frac{1}{p_2}$. 
\end{proposition}

In the next proposition we prove convergence of the Galerkin approximations $\Pi_Nf$ to $f$ in Besov spaces. Here we
use that the projection $\Pi_Nf$ is defined as the convolution of $f$ with the $2$-dimensional square Dirichlet kernel,
which satisfies a logarithmic growth bound in the $L^1$ norm.  

\begin{proposition} \label{prop:proj_bound} Let $\Pi_N: L^2 \to L^2$ be the projection on $\{f \in L^2: f(z)=\sum_{|k|\leq N} \hat f(k) 
L^{-2}\ee^{2\ii\pi k\cdot z/L}\}$. Then for every $\al\in \RR$, $p,q\in[1,\infty]$ and $\lambda>0$ 
\begin{align}
 & \|\Pi_Nf -f\|_{\BB^{\al}_{p,q}} \leq \frac{C(\log N)^2}{N^\lambda} \|f\|_{\BB^{\al+\lambda}_{p,q}} \label{eq:proj_diff_bound} \\
 & \|\Pi_Nf\|_{\BB^{\al}_{p,q}} \leq C \|f\|_{\BB^{\al+\lambda}_{p,q}}. \label{eq:proj_bound}
\end{align}
If we furthermore assume that $p=2$ then
\begin{align}
 & \|\Pi_Nf -f\|_{\BB^{\al}_{2,q}} \leq \frac{C}{N^\lambda} \|f\|_{\BB^{\al+\lambda}_{2,q}} \label{eq:proj_diff_bound_2}\\
 & \|\Pi_Nf\|_{\BB^{\al}_{2,q}} \leq \|f\|_{\BB^{\al}_{2,q}}. \label{eq:proj_bound_2} 
\end{align}
\end{proposition}

\begin{proof} We first notice that for $c_2>c_1>0$
\begin{equation*}
 \delta_\kappa\left(\Pi_Nf - f\right) =
 \begin{cases}
  0 & , \text{ if } 2^\kappa \leq c_1 N \\
  \delta_\kappa f &, \text{ if } 2^\kappa > c_2 N
 \end{cases}
 .
\end{equation*}
Let $D_N(z)=\sum_{|k|\leq N} L^{-2} \ee^{-2\ii\pi k\cdot z/L}$ be the square Dirichlet kernel. Then $\Pi_N f = f*D_N$. Using 
the triangle inequality and Young's inequality for convolution we have that
\begin{equation*}
 \|\delta_\kappa\left( \Pi_Nf - f\right)\|_{L^p} \leq (\|D_N\|_{L^1}+1) \|\delta_\kappa f\|_{L^p}.
\end{equation*}
Thus
\begin{equation*}
\|\delta_\kappa(\Pi_Nf - f)\|_{L^p} \leq
\begin{cases}
 0 & , \text{ if } 2^\kappa \leq c_1N \\
 C (\log N)^2 \|\delta_\kappa f\|_{L^p} & , \text{ if } c_1N \leq 2^\kappa < c_2 N \\
 \|\delta_\kappa f\|_{L^p} & , \text{ if } 2^\kappa > c_2 N
\end{cases}
\end{equation*}
where in the second case we use that $\|D_N\|_{L^1}\lesssim (\log N)^2$. This bound immediate form the fact that the $2$-dimensional
square Dirichlet kernel is the product of two $1$-dimensional Dirichlet kernels 
(see \cite[Section 3.1.3]{Gr14}). The last implies \eqref{eq:proj_diff_bound} and \eqref{eq:proj_bound}. For $p=2$ we notice that 
\begin{equation*}
\|\delta_\kappa \Pi_Nf\|_{L^2} \leq \| \delta_\kappa f\|_{L^2}
\end{equation*}
which implies \eqref{eq:proj_diff_bound_2} and \eqref{eq:proj_bound_2}.
\end{proof}

\stepcounter{section}
\section*{\large Appendix \thesection}

\begin{lemma}[Generalised Gronwall lemma] \label{lem:grwl} Let $f:[0,T] \to \RR$ be a measurable function and $\sigma_1+\sigma_2<1$ such that 
\begin{equation*}
 f(t) \leq \ee^{-c_0t} a + b \int_0^t \ee^{-c_0(t-s)} (t-s)^{-\sigma_1} s^{-\sigma_2} f(s) \dd s.
\end{equation*}
Then there exists $c,C>0$ such that
\begin{equation*}
 f(t) \leq C \exp\left\{-c_0t + c b^{\frac{1}{1-\sigma_1-\sigma_2}} t\right\} a.
\end{equation*}
\end{lemma}

\begin{proof} The lemma is essentially \cite[Lemma 5.7]{HW13} if we set $x(t) = \ee^{c_0t} f(t)$ with their notation.
\end{proof}

\begin{lemma}\label{lem:int_bd} Let $\al+\bt <1$ and $c>0$. Then
\begin{equation*}
 \sup_{t\geq 0} \int_0^t (t-s)^{-\al} (s\wedge 1)^{-\bt} \ee^{-c(t-s)} \dd s <\infty.
\end{equation*}
\end{lemma}

\begin{proof} Assume $t\geq 1$. Then
\begin{align*}
 \int_0^1 (t-s)^{-\al} (s\wedge 1)^{-\bt} \ee^{-c(t-s)} \dd s \lesssim \ee^{-ct} \int_0^t (t-s)^{-\al} (s\wedge 1)^{-\bt} \dd s 
 \lesssim t^{1-\al-\bt} \ee^{-ct}  
\end{align*}
and
\begin{align*}
 \int_1^t (t-s)^{-\al} (s\wedge 1)^{-\bt} \ee^{-c(t-s)} \dd s & \leq \int_0^t s^{-\al} \ee^{-cs} \dd s
 \lesssim 1 + \int_1^t s^{-\al} \ee^{-cs} \dd s 
 \lesssim 1+ \int_1^t \ee^{-cs} \dd s.
\end{align*}
The above implies that 
\begin{equation*}
 \sup_{t\geq 1} \int_0^t (t-s)^{-\al} (s\wedge 1)^{-\bt} \ee^{-c(t-s)} \dd s <\infty.
\end{equation*}
The bound for $t\leq 1$ follows easily.
\end{proof}

\stepcounter{section}
\section*{\large Appendix \thesection}

Propositions \ref{prop:stat_conv} and \ref{prop:unst} are a consequence of \cite[Section 8]{FJL82} and \cite[Appendix B.1]{KORV07}. 
Although the results in \cite[Section 8]{FJL82} concern 1 space-dimension they can be easily generalised in 2 space-dimensions.
For consistency we have also replaced the space $H^1$ appearing in \cite[Section 8]{FJL82} by $\BB^1_{2,2}$. The fact that these
spaces coincide is immediate from Definition \ref{def:besov} for $p=q=2$ if we rewrite $\|f*\eta_k\|_{L^2}$ using Plancherel's identity.  

\begin{proposition} \label{prop:stat_conv} For every $x\in \BB^1_{2,2}$ there exists $x_*\in \{-1,0,1\}$ such that $X_{det}(t;x) 
\stackrel{\BB^1_{2,2}}{\to} x_*$.
\end{proposition}

\begin{proposition} \label{prop:unst} For every $\delta>0$ there exists $x_\pm\in B_{\BB^1_{2,2}}(0;\delta)$ such that 
$X_{det}(t;x_\pm) \stackrel{\BB^1_{2,2}}{\to} \pm 1$.
\end{proposition}

\begin{proposition} \label{prop:Linfty_2+lambda_reg} Let $R>0$. Then there exists $C\equiv C(R)>0$ such that for every $\lambda>0$ 
sufficiently small 
\begin{equation*}
 \sup_{\|x\|_{\CC^{-\al_0}}\leq R} \|X_{det}(1;x)\|_{\CC^{2+\lambda}} \leq C.
\end{equation*}
\end{proposition}

\begin{proof} By \cite[Theorem 3.3, Theorem 3.9]{TW18} there exists $C\equiv C(R)>0$ such that 
\begin{equation*}
 \sup_{\|x\|_{\CC^{-\al_0}}\leq R}\sup_{t\leq 1} t^\gamma \|X_{\det}(t;x)\|_{\CC^\bt} \leq C.
\end{equation*}
Let $S(t) = \ee^{\Delta t}$. Using the mild form we write
\begin{equation*}
 X_{det}(1;x) = S(1/2)X_{det}\left(1/2;x\right) - \int_{1/2}^1 S(1-s)\left(X_{det}(s;x)^3 + X_{det}(s;x)\right) \dd s. 
\end{equation*}
Then
\begin{align*}
 & \|X_{det}(1;x)\|_{\CC^{2+\lambda}} \\
 & \quad \lesssim \|X_{det}\left(1/2;x\right)\|_{\CC^\bt} + \int_{1/2}^1 (1-s)^{-\frac{2+\lambda-\bt}{2}} 
 \left(\|X_{det}(s;x)\|_{\CC^\bt}^3 + \|X_{det}(s;x)\|_{\CC^\bt}\right)
\end{align*}
and if we choose $\lambda <\bt$ the above implies that 
\begin{equation*}
 \sup_{\|x\|_{\CC^{-\al_0}}\leq R} \|X_{det}(1;x)\|_{\CC^{2+\lambda}} \lesssim 
 \sup_{\|x\|_{\CC^{-\al_0}}\leq R}\sup_{t\leq 1} t^{3\gamma} \|X_{\det}(t;x)\|_{\CC^\bt}^3 
 +\sup_{\|x\|_{\CC^{-\al_0}}\leq R}\sup_{t\leq 1} t^\gamma \|X_{\det}(t;x)\|_{\CC^\bt}.
\end{equation*}
\end{proof}

\stepcounter{section}
\section*{\large Appendix \thesection}

\begin{proposition} \label{prop:exp_mom} For every $n\geq 1$ there exists $c\equiv c(n)>0$ such that 
\begin{equation*}
 \sup_{k\geq 0}\EE \exp\left\{c \left(\sup_{t\in[k,k+1]} (t\wedge 1)^{(n-1)\al'} \|\<n>(t)\|_{\CC^{-\al}}\right)^{\frac{2}{n}}\right\} 
 < \infty.
\end{equation*}
\end{proposition}

\begin{proof} Following step by step the proof of \cite[Theorem 2.1]{TW18} but using the explicit bound in Nelson's
estimate \cite[Equation (B.3)]{TW18} (see also \cite[Section 1.6]{Bo07}), we have that for every $p\geq 1$
\begin{equation*}
 \sup_{k\geq 0}\EE\left(\sup_{t\in[k,k+1]} (t\wedge 1)^{(n-1)\al'} \|\<n>(t)\|_{\CC^{-\al}}\right)^p \leq (p-1) ^{\frac{n}{2}p} C_n^\frac{p}{2},
\end{equation*}
for some $C_n>0$. Then for any $c>0$
\begin{align*} 
 & \EE\exp\left\{c\left(\sup_{t\in[k,k+1]} (t\wedge 1)^{(n-1)\al'} \|\<n>(t)\|_{\CC^{-\al}}\right)^{\frac{2}{n}}\right\} \\
 & \quad = 
 \sum_{k\geq0} \frac{c^p \EE\left(\sup_{t\in[k,k+1]} (t\wedge 1)^{(n-1)\al'}
 \|\<n>_{-\infty}(t)\|_{\CC^{-\al}}\right)^{\frac{2}{n}p}}{p!} 
 \leq \sum_{p\geq0} \frac{c^p (p-1)^p (C_n)^\frac{p}{n}}{p!}
\end{align*}
and by choosing $c\equiv c(n)>0$ sufficiently small the series converges.
\end{proof}

\stepcounter{section}
\section*{\large Appendix \thesection}

\begin{lemma}\label{lem:decr_exp} Let $g_1,\tilde g_1$ be positive random variables such that
\begin{equation*}
 \PP(g_1 \geq g) \leq \PP(\tilde g_1 \geq g)
\end{equation*}
for every $g\geq 0$ and let $F$ be a positive decreasing measurable function on $[0,\infty)$. Then
\begin{equation*}
 \int_0^\infty F(g) \, \mu_{g_1}(\dd g) \geq \int_0^\infty F(g) \, \mu_{\tilde g_1}(\dd g)
\end{equation*}
where $\mu_{g_1}$ and $\mu_{\tilde g_1}$ is the law of $g_1$ and $\tilde g_1$.
\end{lemma}

\begin{proof} We first assume that $F$ is smooth. Then $\frac{\dd}{\dd g}F(g) \leq 0$ for every $g\geq 0$. 
Hence 
\begin{align*}
 \int_0^\infty F(g) \, \mu_{g_1}(\dd g) & = F(0) + \int_0^\infty \frac{\dd}{\dd g}F(g) \, \PP(g_1 \geq g) \dd g \geq F(0) + \int_0^\infty \frac{\dd}{\dd g}F(g) \, \PP(\tilde g_1 \geq g) \dd g \\ 
 & = \int_0^\infty F(g) \, \mu_{\tilde g_1}(\dd g)
\end{align*}
which proves the estimate for $F$ differentiable. To prove the estimate for a general decreasing function $F$ we define $F_\delta = F*\eta_\delta$ for
some positive mollifier $\eta_\delta$ to preserve monotonicity and use the last estimate together with the dominated convergence theorem.
\end{proof}

\stepcounter{section}
\hypertarget{proof:X_approx_conv}{} 
\section*{\large Appendix \thesection} 

\begin{proof}[Proof of Proposition \ref{prop:X_approx_conv}] By \cite[Proposition 2.3]{TW18} for every $\al>0$, 
$p\geq 1$ and $T>0$
\begin{equation*}
\lim_{N\to\infty} \EE\left(\sup_{t\leq T} (t\wedge 1)^{(n-1)\al'} \|\<n>_N(t) - \<n>(t)\|_{\CC^{-\al}}\right)^p = 0.
\end{equation*}
Hence $\sup_{t\leq T} (t\wedge 1)^{(n-1)\al'} \|\<n>_N(t) - \<n>(t)\|_{\CC^{-\al}}$ convergences to $0$ in probability.

It is enough to prove that
\begin{equation*}
\lim_{N\to\infty} \sup_{x\in \aleph} \sup_{t\leq T} (t\wedge1)^\gm \|v_N(t;x_N) - v(t;x)\|_{\CC^\bt} = 0. 
\end{equation*}
This, convergence in probability of $\sup_{t\leq T}\|\<1>_N(t) - \<1>(t)\|_{\CC^{-\al}}$ to $0$ and the embedding 
$\CC^\bt \subset \CC^{-\al}$ (see \eqref{eq:alpha_beta_ineq}) imply the result.  

Let $S(t) = \ee^{\Delta t}$. For simplicity we write $v_N(t)$ and $v(t)$ to 
denote $v_N(t;x_N)$ and $v(t;x)$. Using the mild forms of \eqref{eq:v_approx} and \eqref{eq:rem_eq} we get
\begin{align}
 \|v_N(t) - v(t)\|_{\CC^\bt} & \leq \underbrace{\|S(t)(x_N - x)\|_{\CC^\bt}}_{=:I_1} 
 + \underbrace{\int_0^t \|S(t-s) [\Pi_N(v_N(s)^3) - v(s)^3]\|_{\CC^\bt} \dd s}_{=:I_2} \label{eq:gal_conv}\\
 & \quad + 3 \underbrace{\int_0^t \|S(t-s) \left[\Pi_N\left(v_N(s)^2\eps^{\frac{1}{2}}\<1>_N(s)\right) - v(s)^2\eps^{\frac{1}{2}}\<1>(s)\right]\|_{\CC^\bt} \dd s}_{=:I_3} \nonumber\\
 & \quad + 3 \underbrace{\int_0^t \|S(t-s) [\Pi_N(v_N(s) \eps \<2>_N(s)) - v(s) \eps \<2>(s)]\|_{\CC^\bt} \dd s}_{=:I_4}  \nonumber \\ 
 & \quad + \underbrace{\int_0^t \|S(t-s)\left(\Pi_N \eps^{\frac{3}{2}} \<3>_N(s) - \eps^{\frac{3}{2}} \<3>(s)\right)\|_{\CC^\bt} \dd s}_{=:I_5} \nonumber \\ 
 & \quad + \underbrace{2 \int_0^t \|S(t-s) \left(\eps^{\frac{1}{2}} \<1>_N(s) - \eps^{\frac{1}{2}} \<1>(s)\right)\|_{\CC^\bt} \dd s}_{=:I_6}  \nonumber \\
 & \quad + \underbrace{\int_0^t \|S(t-s) (v_N(s) - v(s))\|_{\CC^\bt} \dd s}_{=:I_7} . \nonumber
\end{align}
Let $\iota=\inf\{t>0:(t\wedge 1)^\gamma\|v_N(t) - v(t)\|_{\CC^\bt}\geq 1\}$ and $t\leq T\wedge \iota$.
We treat each of the terms in \eqref{eq:gal_conv} separately. Below the parameters $\al$ and $\lambda$ can be 
taken arbitrarily small and all the implicit constants depend on $\sup_{t\leq T}(t\wedge 1)^{(n-1)\al'} \|\<n>(t)\|_{\CC^{-\al}}$, 
and $\sup_{x\in \aleph}\sup_{t\leq T}(t\wedge 1)^\gamma \|v(t)\|_{\CC^\bt}$.

\underline{Term $I_1$}:
\begin{equation*}
 I_1 \stackrel{\eqref{eq:Heat_Smooth}}{\lesssim} (t\wedge 1)^{-\frac{\al_0+\bt}{2}} 
 \sup_{x\in \aleph} \|x_N - x\|_{\CC^{-\al_0}}
\end{equation*}
\underline{Term $I_2$}: 
\begin{align*}
 I_2 & \stackrel{\eqref{eq:Heat_Smooth}}{\lesssim}
 \int_0^t \left((t-s)^{-\frac{\lambda}{2}} \|\Pi_N(v_N(s)^3) - v_N(s)^3\|_{\CC^{\bt-\lambda}} + \|v_N(s)^3 - v(s)^3\|_{\CC^\bt}\right) \dd s\\
 & \stackrel{\eqref{eq:proj_diff_bound}}{\lesssim} 
 \int_0^t (t-s)^{-\frac{\lambda}{2}} \left(\frac{(\log N)^2}{N^\lambda} \|v_N(s)^3\|_{\CC^\bt} 
 + \|v_N(s)^3 - v(s)^3\|_{\CC^{\bt-\lambda}} \right) \dd s \\
 & \stackrel{\eqref{eq:mult_ineq_1}}{\lesssim} 
 \int_0^t \bigg[(t-s)^{-\frac{\lambda}{2}} \frac{(\log N)^2}{N^\lambda} \|v_N(s)\|_{\CC^\bt}^3 
 + \|v_N(s) - v(s)\|_{\CC^\bt} \\
 & \quad \times \left(\|v_N(s)\|_{\CC^\bt}^2 + \|v_N(s)\|_{\CC^\bt} \|v(s)\|_{\CC^\bt} 
 + \|v(s)\|_{\CC^\bt}^2\right)\bigg] \dd s \\
 & \lesssim \int_0^t \left((t-s)^{-\frac{\bt+\frac{2}{p}-1}{2}} \frac{(\log N)^2}{N^\lambda} (s\wedge 1)^{-3\gamma} 
 + (s\wedge 1)^{-2\gamma} \|v_N(s) - v(s)\|_{\CC^\bt} \right) \dd s.
\end{align*}
\underline{Term $I_3$}:
\begin{align*}
 I_3 & \stackrel{\eqref{eq:Heat_Smooth}}{\lesssim} 
 \int_0^t (t-s)^{-\frac{\al+\bt+\lambda}{2}} 
 \|\Pi_N(v_N(s)^2\<1>_N(s)) - v(s)^2\<1>(s)\|_{\CC^{-\al-\lambda}} \dd s\\
 & \stackrel{\eqref{eq:proj_diff_bound}}{\lesssim} \int_0^t (t-s)^{-\frac{\al+\bt+\lambda}{2}} \bigg(\frac{(\log N)^2}{N^\lambda} 
 \|v_N(s)^2\<1>_N(s)\|_{\CC^{-\al}} 
 + \|v_N(s)^2 (\<1>_N(s) - \<1>(s))\|_{\CC^{-\al}} \\
 & \quad + \|\<1>(s) (v_N(s)^2 - v(s)^2)\|_{\CC^{-\al}} \bigg) \dd s\\
 & \stackrel{\eqref{eq:mult_ineq_2},\eqref{eq:mult_ineq_1}}{\lesssim} 
 \int_0^t (t-s)^{-\frac{\al+\bt+\lambda}{2}} \bigg[\frac{(\log N)^2}{N^\lambda} 
 \|v_N(s)\|_{\CC^\bt}^2 \|\<1>_N(s)\|_{\CC^{-\al}} + \|v_N(s)\|_{\CC^\bt}^2 \|\<1>_N(s) - \<1>(s)\|_{\CC^{-\al}} \\
 & \quad + \left(\|v_N(s)\|_{\CC^\bt} + \|v(s)\|_{\CC^\bt}\right)
 \|v_N(s) - v(s)\|_{\CC^\bt} \|\<1>(s)\|_{\CC^{-\al}} \bigg]\dd s \\
 & \lesssim 
 \int_0^t (t-s)^{-\frac{\al+\bt+\lambda}{2}} \bigg(\frac{(\log N)^2}{N^\lambda} (s\wedge 1)^{-2\gamma} 
 + (s\wedge 1)^{-2\gamma} \|\<1>_N(s) - \<1>(s)\|_{\CC^{-\al}} \\
 & \quad + (s\wedge 1)^{-\gamma} \|v_N(s) - v(s)\|_{\CC^\bt}\bigg) \dd s.
\end{align*}
\underline{Term $I_4$}: Similarly to $I_3$, 
\begin{align*}
 I_4 & \lesssim \int_0^t (t-s)^{-\frac{\al+\bt+\lambda}{2}} \bigg(\frac{(\log N)^2}{N^\lambda} (s\wedge 1)^{-\gamma-\al'} 
 + (s\wedge 1)^{-\gamma} \|\<2>_N(s) - \<2>(s)\|_{\CC^{-\al}} \\
 & \quad + (s\wedge 1)^{-\al'} \|v_N(s) - v(s)\|_{\CC^\bt}\bigg) \dd s.
\end{align*}
\underline{Term $I_5$}:
\begin{align*}
 I_5 & \stackrel{\eqref{eq:Heat_Smooth}}{\lesssim} 
 \int_0^t (t-s)^{-\frac{\al+\bt+\lambda}{2}} \|\Pi_N\<3>_N(s) - \<3>(s)\|_{\CC^{-\al-\lambda}} \dd s \\
 & \stackrel{\eqref{eq:proj_diff_bound}}{\lesssim} \int_0^t (t-s)^{-\frac{\al+\bt+\lambda}{2}} \left(\frac{(\log N)^2}{N^\lambda} (s\wedge 1)^{-2\al'} 
 + \|\<3>_N(s) - \<3>(s)\|_{\CC^{-\al}}\right) \dd s.
\end{align*}
\underline{Terms $I_6$, $I_7$}:
\begin{align*}
 & I_6 \stackrel{\eqref{eq:Heat_Smooth}}{\lesssim} 
 \int_0^t (t-s)^{-\frac{\al+\bt}{2}} \|\<1>_N(s) - \<1>(s)\|_{\CC^{-\al}} \dd s. \\
 & I_7 \stackrel{\eqref{eq:Heat_Smooth}}{\lesssim} \int_0^t \|v_N(s) - v(s)\|_{\CC^\bt} \dd s. 
\end{align*}
Combining the above estimates we obtain that for $t\leq T\wedge \iota$ 
\begin{align*}
 \|v_N(t) - v(t)\|_{\CC^\bt} & \lesssim (t\wedge 1)^{-\frac{\al_0+\bt}{2}}  \sup_{x\in \aleph} \|x_N-x\|_{\CC^{-\al_0}} \\
 & \quad + T^{1-\frac{\al+\bt+\lambda}{2}-3\gm} \left(\frac{(\log N)^2}{N} + \sup_{t\leq T} (t\wedge 1)^{(n-1)\al'} \|\<n>_N(t) - \<n>(t)\|_{\CC^{-\al}}\right) \\
 & \quad + \int_0^t (t-s)^{-\frac{\al+\bt+\lambda}{2}} (s\wedge 1)^{-2\gm} \|v_N(s) - v(s)\|_{\CC^\bt} \dd s.
\end{align*}
By Lemma \ref{lem:grwl} on $f(t) = (t\wedge 1)^\gamma \|v_N(t) - v(t)\|_{\CC^\bt}$ we find $C\equiv C(T)>0$ such that
\begin{align*}
 & \sup_{t\leq T\wedge \iota}(t\wedge 1)^\gamma \|v_N(t) - v(t)\|_{\CC^\bt} \\
 & \quad \leq C 
 \left( \sup_{x\in \aleph} \|x_N-x\|_{\CC^{-\al_0}} + \frac{(\log N)^2}{N}
 + \sup_{t\leq T} (t\wedge 1)^{(n-1)\al'} \|\<n>_N(t) - \<n>(t)\|_{\CC^{-\al}}\right). 
\end{align*}
This and convergence of $\sup_{t\leq T}\|\<n>_N(t) - \<n>(t)\|_{\CC^{-\al}}$ to $0$ in probability 
imply the result.   
\end{proof}

\stepcounter{section}
\section*{\large Appendix \thesection}

In this section we fix $\bt\in\left(\frac{1}{3},\frac{2}{3}\right)$, $\gamma\in\left(\frac{\bt}{2},\frac{1}{3}\right)$ and $p\in(1,2)$ such that 
\begin{equation*}
 1-\frac{2}{3p} <\bt \text{ and } 1-\frac{\bt+\frac{2}{p}-1}{2}-2\gamma > 0.
\end{equation*}

The next proposition provides local existence of \eqref{eq:v_approx} in $\BB^\bt_{2,2}$ up to some time $T_*>0$
which is uniform in the regularisation parameter $N$. 

\begin{proposition} \label{prop:L_2_v_exist} Let $K, R, T>0$ such that $\|x\|_{\BB^{-\al_0}_{2,2}} \leq R$ and 
$\sup_{t\leq T} (t\wedge 1)^{(n-1)\al'}\|\<n>_N(t)\|_{\BB^{-\al}_{\infty,2}}\leq K$. Then there exist $T_*\equiv T_*(K,R) \leq T$
and $C\equiv C(K,R)>0$ such that \eqref{eq:v_approx} has a unique solution $v\in C((0,T_*];\BB^\bt_{2,2})$ satisfying
\begin{equation*}
 \sup_{t\leq T_*} (t\wedge 1)^\gamma \|v_N(t;x)\|_{\BB^\bt_{2,2}} \leq C.
\end{equation*}
\end{proposition}

\begin{proof} Let $S(t) = \ee^{\Delta t}$. We define  
\begin{align*}
 \mathscr{T}(v)(t) & := S(t)x - \int_0^t S(t-s)\Pi_N\left(v(s)^3 + 3 v(s)^2 \eps^{\frac{1}{2}}\<1>_N(s) + 3 v(s)\eps\<2>_N(s) + \eps^{\frac{3}{2}} \<3>_N(s)\right) \dd s \\
 & \quad + 2 \int_0^t S(t-s)\left(\eps^{\frac{1}{2}}\<1>_N(s) + v(s)\right) \dd s.
\end{align*}
It is enough to prove that there exists $T_*>0$ such that $\mathscr{T}$ is a contraction on 
\begin{equation*}
 \mathscr{B}_{T_*} := \left\{v: \sup_{t\leq T_*}(t\wedge 1)^\gamma\|v(t;x)\|_{\BB^\bt_{2,2}} \leq 1\right\}.
\end{equation*}
We first prove that for $T_*>0$ sufficiently small $\mathscr{T}$ maps $\mathscr{B}_{T_*}$ to itself. To do so we notice that
\begin{align*}
 \|\mathscr{T}(v)(t)\|_{\BB^\bt_{2,2}} & \lesssim \underbrace{\|S(t)x\|_{\BB^\bt_{2,2}}}_{=:I_1} + 
 \underbrace{\int_0^t \|S(t-s)v(s)^3\|_{\BB^\bt_{2,2}} \dd s}_{=:I_2} 
 + \underbrace{\int_0^t \|S(t-s)(v(s)^2\<1>_N(s))\|_{\BB^\bt_{2,2}} \dd s}_{=:I_3} \\
 & \quad + \underbrace{\int_0^t \|S(t-s)(v(s)\<2>_N(s))\|_{\BB^\bt_{2,2}} \dd s}_{=:I_4} 
 + \underbrace{\int_0^t \|S(t-s)\<3>_N(s)\|_{\BB^\bt_{2,2}} \dd s}_{=:I_5} \\
 & \quad + \underbrace{\int_0^t \|S(t-s)\<1>_N(s)\|_{\BB^\bt_{2,2}} \dd s}_{=:I_6} 
 + \underbrace{\int_0^t \|S(t-s)v(s)\|_{\BB^\bt_{2,2}} \dd s}_{=:I_7}
\end{align*}
where we use \eqref{eq:proj_bound_2} together with the relation $S(\cdot)\Pi_N = \Pi_NS(\cdot)$ to drop $\Pi_N$. 
We treat each term separately.

\underline{Term $I_1$}:
\begin{equation*}
 I_1 \stackrel{\eqref{eq:Heat_Smooth}}{\lesssim} (t\wedge 1)^{-\frac{\al_0+\bt}{2}} \|x\|_{\BB^{-\al_0}_{2,2}} \lesssim (t\wedge 1)^{-\frac{\al_0+\bt}{2}} R.
\end{equation*}
\underline{Term $I_2$}:
\begin{align*}
 I_2 & \stackrel{\eqref{eq:Besov_Emb}}{\lesssim} \int_0^t \|S(t-s)v(s)^3\|_{\BB^{\bt+\frac{2}{p}-1}_{p,2}} \dd s 
 \stackrel{\eqref{eq:Heat_Smooth}}{\lesssim} \int_0^t (t-s)^{-\frac{\bt+\frac{2}{p}-1}{2}} \|v(s)^3\|_{\BB^0_{p,2}} \dd s \\
 & \stackrel{\eqref{eq:mult_ineq_1}}{\lesssim} \int_0^t (t-s)^{-\frac{\bt+\frac{2}{p}-1}{2}} \|v(s)\|_{\BB^0_{3p,2}}^3 \dd s 
 \stackrel{\eqref{eq:Besov_Emb}}{\lesssim} \int_0^t (t-s)^{-\frac{\bt+\frac{2}{p}-1}{2}} \|v(s)\|_{\BB^{1-\frac{2}{3p}}_{2,2}}^3 \dd s \\
 & \stackrel{1-\frac{2}{3p}<\bt}{\lesssim} \int_0^t (t-s)^{-\frac{\bt+\frac{2}{p}-1}{2}} \|v(s)\|_{\BB^\bt_{2,2}}^3 \dd s 
 \lesssim \int_0^t (t-s)^{-\frac{\bt+\frac{2}{p}-1}{2}} (s\wedge 1)^{-3\gamma} \dd s.
\end{align*}
\underline{Term $I_3$}:
\begin{align*}
 I_3 & \stackrel{\eqref{eq:Besov_Emb},\eqref{eq:Heat_Smooth}}{\lesssim} 
 \int_0^t (t-s)^{-\frac{\bt+\frac{2}{p}-1+\al}{2}} \|v(s)^2\<1>_N(s)\|_{\BB^{-\al}_{p,2}} \dd s 
 \stackrel{\eqref{eq:mult_ineq_2},\eqref{eq:mult_ineq_1}}{\lesssim} 
 K \int_0^t (t-s)^{-\frac{\bt+\frac{2}{p}-1+\al}{2}} \|v(s)\|_{\BB^{\al+\lambda}_{2p,2}}^2 \dd s \\
 & \stackrel{\eqref{eq:Besov_Emb}}{\lesssim} 
 K \int_0^t (t-s)^{-\frac{\bt+\frac{2}{p}-1+\al}{2}} \|v(s)\|_{\BB^{\al+\lambda+1-\frac{1}{p}}_{2,2}}^2 \dd s 
 \stackrel{1-\frac{2}{3p}<\bt}{\lesssim} 
 K \int_0^t (t-s)^{-\frac{\bt+\frac{2}{p}-1+\al}{2}} \|v(s)\|_{\BB^\bt_{2,2}}^2 \dd s \\
 & \lesssim K \int_0^t (t-s)^{-\frac{\bt+\frac{2}{p}-1+\al}{2}} (s\wedge 1)^{-2\gamma} \dd s.
\end{align*}
\underline{Term $I_4$}:
\begin{align*}
 I_4 & \stackrel{\eqref{eq:Besov_Emb},\eqref{eq:Heat_Smooth}}{\lesssim} 
 \int_0^t (t-s)^{-\frac{\bt+\frac{2}{p}-1+\al}{2}} \|v(s)\<2>_N(s)\|_{\BB^{-\al}_{p,2}} \dd s \\
 & \stackrel{\eqref{eq:mult_ineq_2},\eqref{eq:Besov_Emb}}{\lesssim}
 K \int_0^t (t-s)^{-\frac{\bt+\frac{2}{p}-1+\al}{2}} (s\wedge1)^{-\al'} \|v(s)\|_{\BB^{\al+\lambda+1-\frac{2}{p}}_{2,2}} \dd s \\
 & \stackrel{1-\frac{2}{3p}<\bt}{\lesssim} K \int_0^t (t-s)^{-\frac{\bt+\frac{2}{p}-1+\al}{2}} (s\wedge1)^{-\al'} \|v(s)\|_{\BB^\bt_{2,2}} \dd s \\
 & \lesssim K \int_0^t (t-s)^{-\frac{\bt+\frac{2}{p}-1+\al}{2}} (s\wedge 1)^{-\gamma-\al'} \dd s.
\end{align*}
\underline{Terms $I_5$, $I_6$, $I_7$}:
\begin{align*}
 & I_5 \stackrel{\eqref{eq:Heat_Smooth}}{\lesssim} 
 \int_0^t (t-s)^{-\frac{\bt+\al}{2}} \|\<3>_N(s)\|_{\BB^{-\al}_{2,2}} \dd s 
 \lesssim K \int_0^t (t-s)^{-\frac{\bt+\al}{2}} (s\wedge 1)^{-2\al'} \dd s. \\
 & I_6 \stackrel{\eqref{eq:Heat_Smooth}}{\lesssim} 
 \int_0^t (t-s)^{-\frac{\bt+\al}{2}} \|\<1>_N(s)\|_{\BB^{-\al}_{2,2}} \dd s 
 \lesssim K \int_0^t (t-s)^{-\frac{\bt+\al}{2}} \dd s. \\
 & I_7 \stackrel{\eqref{eq:Heat_Smooth}}{\lesssim} 
 \int_0^t  \|v(s)\|_{\BB^\bt_{2,2}} \dd s \lesssim \int_0^t (s\wedge 1)^{-\gamma} \dd s.
\end{align*}
Combining all the above we find $C\equiv C(K,R)>0$ such that 
\begin{equation*}
 \sup_{t\leq T_*} (t\wedge 1)^\gamma \|\mathscr{T}(v)(t)\|_{\BB^\bt_{2,2}} \leq C T_*^\theta
\end{equation*}
for some $\theta \equiv \theta(\al,\al',\al_0,\bt,\gamma)\in (0,1)$. Choosing $T_*>0$ sufficiently small the above implies that 
\begin{equation*}
 \sup_{t\leq T_*} (t\wedge 1)^\gamma \|\mathscr{T}(v)(t)\|_{\BB^\bt_{2,2}} \leq 1.
\end{equation*}
Hence for this choice of $T_*$, $\mathscr{T}$ maps $\mathscr{B}_{T_*}$ to itself. In a similar way, but by possibly choosing a smaller value of $T_*$,
we prove that $\mathscr{T}$ is a contraction on $\mathscr{B}_{T_*}$. For simplicity we omit the proof. That way we obtain a unique solution 
$v\in C((0,T_*];\BB^\bt_{2,2})$. We can furthermore assume that $T_*$ is maximal in the sense that either $T_* = T$ or 
$\lim_{t\nearrow T_*} \|v(t;x)\|_{\BB^\bt_{2,2}} = \infty$.
\end{proof}

\begin{proposition} \label{prop:L2_to_Linfty} For every $t_0\in(0,1)$, and $K,R>0$ there exists $C\equiv C(t_0,K,R)>0$ such that if $\|x\|_{\BB^{-\al}_{2,2}}
\leq R$ and $\sup_{t\leq 1} t^{(n-1)\al'}\|\eps^{\frac{n}{2}}\<n>_N(t)\|_{\CC^{-\al}} \leq K$ then  
\begin{equation*}
 \sup_{\|x\|_{\BB^{-\al}_{2,2}} \leq R} \|X_N(t_0;x)\|_{\CC^{-\al}} \leq C.
\end{equation*}
\end{proposition}

\begin{proof} Using the a priori estimate in Proposition \ref{prop:X_approx_bd} we can assume that $T_* =1$ in 
Proposition \ref{prop:L_2_v_exist}. This implies that
\begin{equation}\label{eq:L2_bd}
 \sup_{\|x\|_{\BB^{-\al}_{2,2}}\leq R}\sup_{t\leq 1} t^\gamma\|v_N(t;x)\|_{\BB^\bt_{2,2}} \leq C.
\end{equation}
For simplicity we assume that $t_0=1$. Let $S(t) = \ee^{\Delta t}$. Using the mild form of \eqref{eq:v_approx} we obtain that
\begin{align*}
 & \|v_N(1)\|_{\CC^{-\al}} \\
 & \quad \lesssim \underbrace{\|S(1/2)v_N(1/2)\|_{\CC^{-\al}}}_{=:I_1} + 
 \underbrace{\int_{1/2}^1 \|S(1-s)\Pi_N(v_N(s))^3\|_{\CC^{-\al}} \dd s}_{=:I_2} \\ 
 & \quad \quad + \underbrace{\int_{1/2}^1 \|S(1-s)\Pi_N\left(v_N(s)^2\eps^{\frac{1}{2}}\<1>_N(s)\right)\|_{\CC^{-\al}} \dd s}_{=:I_3} 
 + \underbrace{\int_{1/2}^1 \|S(1-s)\Pi_N\left(v_N(s)\eps\<2>_N(s)\right)\|_{\CC^{-\al}} \dd s}_{=:I_4} \\
 & \quad \quad + \underbrace{\int_{1/2}^{1} \|S(1-s)\Pi_N\eps^{\frac{3}{2}}\<3>_N(s)\|_{\CC^{-\al}} \dd s}_{=:I_5} 
 + \underbrace{\int_{1/2}^1 \|S(1-s)\eps^{\frac{1}{2}}\<1>_N(s)\|_{\CC^{-\al}} \dd s}_{=:I_6} \\
 & \quad \quad + \underbrace{\int_{1/2}^1 \|S(1-s)v_N(s)\|_{\CC^{-\al}} \dd s}_{=:I_7}.
\end{align*}
We treat each term separately.

\underline{Term $I_1$}:
\begin{equation*}
 I_1 \stackrel{\eqref{eq:Besov_Emb}}{\lesssim} \|S(1/2)v_N(1/2)\|_{\BB^{-\al+1}_{2,\infty}} 
 \stackrel{\eqref{eq:Heat_Smooth}}{\lesssim} \|v_N(1/2)\|_{\BB^{-\al}_{2,\infty}} 
 \lesssim \|v_N(1/2)\|_{\BB^{-\al}_{2,2}}
\end{equation*}
\underline{Term $I_2$}:
\begin{align*}
 I_2 & \stackrel{\eqref{eq:Besov_Emb}}{\lesssim} \int_{1/2}^1 \|S(1-s)\Pi_N(v_N(s)^3)\|_{\BB^{-\al+\frac{2}{p}}_{p,\infty}} \dd s
 \stackrel{\eqref{eq:Heat_Smooth}}{\lesssim} \int_{1/2}^1 (1-s)^{-\frac{-\al+\frac{2}{p}+\lambda}{2}} 
 \|\Pi_N(v_N(s)^3)\|_{\BB^{-\lambda}_{p,\infty}} \dd s \\
 & \stackrel{\eqref{eq:proj_bound},\eqref{eq:mult_ineq_1}}{\lesssim} \int_{1/2}^1 (1-s)^{-\frac{-\al+\frac{2}{p}+\lambda}{2}} 
 \|v_N(s)\|_{\BB^0_{3p,\infty}}^3 \dd s
 \stackrel{\eqref{eq:Besov_Emb}}{\lesssim} \int_{1/2}^1 (1-s)^{-\frac{-\al+\frac{2}{p}+\lambda}{2}} 
 \|v_N(s)\|_{\BB^{1-\frac{2}{3p}}_{2,\infty}}^3 \dd s \\ 
 & \stackrel{1-\frac{2}{3p}<\bt}{\lesssim} \int_{1/2}^1 (1-s)^{-\frac{-\al+\frac{2}{p}+\lambda}{2}} 
 \|v_N(s)\|_{\BB^\bt_{2,2}}^3 \dd s.
\end{align*}
\underline{Term $I_3$}:
\begin{align*}
 I_3 & \stackrel{\eqref{eq:Besov_Emb}}{\lesssim} 
 \int_{1/2}^1 \|S(1-s)\Pi_N\left(v_N(s)^2\eps^{\frac{1}{2}}\<1>_N(s)\right)\|_{\BB^{-\al+\frac{2}{p}}_{p,\infty}} \dd s \\
 & \stackrel{\eqref{eq:Heat_Smooth}}{\lesssim} \int_{1/2}^1 (1-s)^{-\frac{\frac{2}{p}+\lambda}{2}} 
 \|\Pi_N\left(v_N(s)^2\eps^{\frac{1}{2}}\<1>_N(s)\right)\|_{\BB^{-\al-\lambda}_{p,\infty}} \dd s \\
 & \stackrel{\eqref{eq:proj_bound},\eqref{eq:mult_ineq_2},\eqref{eq:mult_ineq_1}}{\lesssim}
 \int_{1/2}^1 (1-s)^{-\frac{\frac{2}{p}+\lambda}{2}} 
 \|v_N(s)\|_{\BB^{\al+\lambda}_{2p,\infty}}^2 \|\eps^{\frac{1}{2}}\<1>_N(s)\|_{\CC^{-\al}} \dd s \\
 & \stackrel{\eqref{eq:Besov_Emb}}{\lesssim} \int_{1/2}^1 (1-s)^{-\frac{\frac{2}{p}+\lambda}{2}} 
 \|v_N(s)\|_{\BB^{\al+\lambda+1-\frac{1}{p}}_{2,\infty}}^2 \|\eps^{\frac{1}{2}}\<1>_N(s)\|_{\CC^{-\al}} \dd s \\
 & \stackrel{1-\frac{2}{3p}<\bt}{\lesssim} \int_{1/2}^1 (1-s)^{-\frac{\frac{2}{p}+\lambda}{2}} 
 \|v_N(s)\|_{\BB^\bt_{2,2}}^2 \|\eps^{\frac{1}{2}}\<1>_N(s)\|_{\CC^{-\al}} \dd s.
\end{align*}
\underline{Term $I_4$}: Similarly to $I_3$,
\begin{equation*}
 I_4 \lesssim \int_{1/2}^1 (1-s)^{-\frac{\frac{2}{p}+\lambda}{2}} 
 \|v_N(s)\|_{\BB^\bt_{2,2}} \|\eps\<2>_N(s)\|_{\CC^{-\al}} \dd s.
\end{equation*}
\underline{Terms $I_5$, $I_6$, $I_7$}:
\begin{align*}
 & I_5 \stackrel{\eqref{eq:Heat_Smooth}}{\lesssim} \int_{1/2}^1 (1-s)^{-\frac{\lambda}{2}} \|\Pi_N\eps^{\frac{3}{2}}\<3>_N(s)\|_{\CC^{-\al-\lambda}} \dd s
 \stackrel{\eqref{eq:proj_bound}}{\lesssim} \int_{1/2}^1 (1-s)^{-\frac{\lambda}{2}} \|\eps^{\frac{3}{2}}\<3>_N(s)\|_{\CC^{-\al-\lambda}} \dd s. \\
 & I_6 \stackrel{\eqref{eq:Heat_Smooth}}{\lesssim} \int_{1/2}^1 (1-s)^{-\frac{\lambda}{2}} \|\eps^{\frac{1}{2}}\<1>_N(s)\|_{\CC^{-\al-\lambda}} \dd s. \\
 & I_7 \stackrel{\eqref{eq:Besov_Emb}}{\lesssim} \int_{1/2}^1 \|S(1-s)v_N(s)\|_{\BB^{-\al+1}_{2,2}} \dd s
 \stackrel{\eqref{eq:Heat_Smooth}}{\lesssim} \int_{1/2}^1 (1-s)^{-\frac{-\al+1-\bt}{2}} \|v_N(s)\|_{\BB^\bt_{2,2}} \dd s.
\end{align*}
The proof is complete if we combine these estimates with \eqref{eq:L2_bd}. 
\end{proof}

\end{appendices}

\pdfbookmark{References}{references}

\bibliographystyle{alpha}
\bibliography{exp_loss_bibliography}{}

\begin{flushleft}
\small \normalfont
\textsc{Pavlos Tsatsoulis\\
University of Warwick\\
Coventry, UK}\\
\texttt{\textbf{p.tsatsoulis@warwick.ac.uk}}

\smallskip

\small \normalfont
\textsc{Hendrik Weber\\
University of Warwick\\
Coventry, UK}\\
\texttt{\textbf{hendrik.weber@warwick.ac.uk}}
\end{flushleft}

\end{document}